\documentclass[10pt,epsfig]{article}
\usepackage{amsmath,amsthm,amssymb}
\usepackage{amsmath} 
\usepackage{bm} 
\usepackage{color}
\usepackage{esint,graphicx,indentfirst}
\usepackage{esvect} 
\usepackage{hyperref}
\usepackage{xcolor}
\usepackage{graphicx, color, epstopdf}
\usepackage{booktabs}

\usepackage{subfigure}
\usepackage{empheq}
\usepackage{hyperref}
\usepackage{mathrsfs}
\usepackage{multirow}
\usepackage{multicol}

\usepackage[
backend=biber,
style=numeric,
sorting=none,
maxnames=99,
doi=false,
url=false,
isbn=false
]{biblatex}

\DeclareNameAlias{author}{family-given}

\DeclareSourcemap{
	\maps[datatype=bibtex]{
		\map{
			\step[fieldsource=archiveprefix, match=\regexp{arXiv}, final]
			\step[fieldset=eprinttype, fieldvalue=arxiv]
		}
	}
}

\DeclareFieldFormat[article,inbook,incollection,inproceedings]{title}{#1\addperiod}

\DeclareFieldFormat[article]{journaltitle}{\emph{#1}\addcomma}
\DeclareFieldFormat[article]{volume}{\textbf{#1}}
\DeclareFieldFormat[article]{number}{(#1)}
\DeclareFieldFormat{pages}{#1}

\DeclareFieldFormat{eprint:arxiv}{arXiv\addcolon\space #1}

\AtEveryBibitem{%
	\clearfield{doi}%
	\clearfield{url}%
	\clearfield{isbn}%
	\clearfield{issn}%
	\clearfield{month}%
	\clearfield{editor}%
	\ifentrytype{article}{%
		\clearfield{note}%
	}{}%
	\iffieldequalstr{eprinttype}{arxiv}{}{%
		\clearfield{eprint}%
		\clearfield{eprinttype}%
		\clearfield{eprintclass}%
	}%
}

\newcommand{\R}{\mathbb{R}}

\newcommand{\N}{\mathcal{N}}
\newcommand{\D}{\mathcal{D}}

\newcommand{\xb}{\mathbf{x}}

\renewenvironment{proof}[1][Proof]{\noindent\textit{#1. } }{\hfill$\square$}

\newtheorem{theorem}{Theorem}[section]
\newtheorem{lemma}{Lemma}[section]
\newtheorem{definition}[theorem]{Definition}
\newtheorem{example}{Example}[section]

\newtheorem{remark}{Remark}[section]
\newtheorem{corollary}{Corollary}[section]

\numberwithin{equation}{section}

\def\L{\mathcal{L}}

\newcommand{\JW}[1]{{#1}}
\newcommand{\YS}[1]{{#1}}

\title{{Asymptotic Compatibility of the Approximate-Ball Finite Element Method for 2D Nonlocal Poisson Problem with Neumann Boundary Conditions}
\thanks{This work was supported by the NSFC (No. 12571438), 
the Experiment Technology Project Funds of Wuhan University (No. WHU-2022-SYJS-0002).}}
\author{
Yuchen Shi\thanks{
			School of Mathematics and Statistics, Wuhan University, Wuhan 430072, China (yuchen\_shi@whu.edu.cn).}
			\and Jihong Wang\thanks{School of Mathematics and
				Statistics, Huazhong University of Science and Technology, Wuhan
				430074, China
				(jhwang00519@hust.edu.cn).}
				\thanks{Hubei Key Laboratory of Engineering Modeling and scientific Computing, Huazhong University of science and Technology, Wuhan 430074, China}
		\and Jiwei Zhang\thanks{
			School of Mathematics and Statistics, and Hubei Key Laboratory of Computational Science, Wuhan University, Wuhan 430072, China (jiweizhang@whu.edu.cn). }
}
\begin{document}
\maketitle
\begin{abstract}

In this paper, asymptotic compatibility error estimates of a finite element discretization is presented for 2D nonlocal Poisson problems with Neumann boundary conditions. To this end, we begin with deriving two kind of nonlocal Neumann boundary operators based on nonlocal Green’s identities, and establish the corresponding weak convergence to the classical Neumann operator as the horizon parameter $\delta$ vanishes for general convex domains. After that, we consider the asymptotic  properties (i.e. the so-called local limit) of two nonlocal Neumann boundary-value problems as $\delta$ approaches zero. Finally, we analyze the  asymptotical  compatable error estimates of the approximate-ball-strategy finite element discretization proposed in \cite{Cookbook}, and provide numerical examples to confirm the theoretical results.

\end{abstract}
\section{Introduction}
Over the past decades, nonlocal models have attracted growing attention due to their broad applications in material mechanics \cite{SILLING2000175,BaJiNonlocal,Dunonlocal,LocalNonlocalCoup,Chen2017AsymptoticallyCS,DElia2019ARO,CAO2023120}, stochastic processes \cite{zhao2016generalized,d2017nonlocal}, image processing \cite{gilboa2009nonlocal,tao2018nonlocal,d2020bilevel,you2022nonlocal}, traffic flow modeling \cite{huang2022stability,huang2024asymptotic}, etc. A fundamental challenge in these problems lies in the formulation of nonlocal boundary conditions, which are essential for ensuring well-posedness, conservation, and numerical stability. Unlike local models, nonlocal operators inherently couple interior and exterior regions through volumetric interactions, making boundary treatment crucial for maintaining mass balance, energy consistency, and flux continuity. While the nonlocal Dirichlet problem has been extensively studied, including the well-posedness, asymptotic compatibility, and convergence of numerical schemes \cite{tian2014asymptotically,Du2013APE,du2019conforming,yang2022uniform,GanDuShi,Foss2021ConvergenceAA,Cookbook,DuXieYin2022,du2024errorestimatesfiniteelement}, the nonlocal Neumann problem remains comparatively underexplored \cite{TAO2017282,DuZhZh,Yu2020,DeTianYu,FreVolVu,Shi2024nonlocal,DuTianZhou2023}. The main difficulty stems from the absence of a natural nonlocal analogue of the classical normal derivative, which makes it nontrivial to define Neumann-type operators that preserve conservation laws and remain asymptotically compatible with the local normal derivative.

There are mainly two kinds of ways to construct the Neumann boundary operators. The first way is to modify the nonlocal operator using boundary information and adjusting the body force term \cite{TAO2017282,Yu2020,Shi2024nonlocal}. For instance, a new nonlocal Neumann operator in \cite{Yu2020}  is provided  by means of a Taylor expansion along the tangential and normal directions, and  the body force is modified  by adding a term with the information of the local Neumann term. The second way is to directly construct nonlocal Neumann operator through nonlocal Green's first identity by comparing with local Green's identity \cite{ZhengDuMaZhang,DuTianZhou2023}.  Such a construction of the nonlocal Neumann operator is more consistent with nonlocal calculus fundamentals and imposes fewer geometric restrictions, enhancing versatility for complex domains. Recently, the asymptotic convergence to the local normal derivative in the weak sense is discussed in \cite{DuTianZhou2023}. The analysis of weak convergence in \cite{DuTianZhou2023} requires test functions belong to $C^3$. In this paper, we will reduce the regularity requirement of the test space to $C^2$,  and present the weak convergence in general two-dimensional domains.

After that, we focus on the local limit behaviors between  the nonlocal solutions and local solutions to the corresponding Neumann boundary problems as the horizon parameter $\delta$ vanishes. Many studies are carried out to consider this local limit of the nonlocal Neumann problems. For instance, Tao et al. \cite{TAO2017282} present the variational analysis  to prove the local limit as $\delta\rightarrow0$. Later, Du, Zhang and Zheng \cite{DuZhZh} revisit the nonlocal maximum principle and establish a sharp estimate on the asymptotical behavior of nonlocal Neumann problem in 1D. D'Elia et al.  \cite{DeTianYu} propose an efficient and flexible method to convert boundary data into volumetric constraints, and prove that the nonlocal solutions converge to the corresponding local solution in second–order rate with respect to $\delta$. One can refer to  \cite{Yu2020,Shi2024nonlocal,parks2025nonlocal} for more related works on local limits of the nonlocal Neumann problems. In this paper, we focus on two representative nonlocal Neumann problems (inner Neumann problem and outer Neumann problem) whose nonlocal Neumann operators are derived from nonlocal Green's identity, and investigate the second-order convergence of their solutions to the corresponding local Neumann solutions as $\delta\rightarrow 0$.


On the other hand, the concept of asymptotic compatible (AC) proposed in \cite{tian2014asymptotically} plays an important role to study the local limit when numerical methods are used to  correctly simulate the solutions of both nonlocal and local problems.  This is to say, the AC scheme can guarantee that  the numerical solutions of nonlocal problems can converge to the correct solution of the corresponding local problems as both $\delta$ and mesh size $h$ vanish. Tian et al. \cite{tian2014asymptotically} proves that the corresponding Galerkin finite element scheme is asymptotically compatible once the finite element space includes continuous piecewise linear functions. This result is based on  the exact integration of the nonlocal interaction neighborhoods $B_\delta(\mathbf{x})$. In the practical simulations  in \cite{DuXieYin2022, du2024errorestimatesfiniteelement} for 2D problems, the method using linear finite element is only a conditionally AC scheme under the condition $\delta = o(h)$ by using the recently developed Nocaps strategy in \cite{Cookbook}. As analyzed in \cite{DuXieYin2022, du2024errorestimatesfiniteelement}, the main reason is that it is hard to exactly compute integral over the interaction domain $B_\delta(\mathbf{x})$, we generally need to approximate the intersection between the mesh and $B_\delta(\mathbf{x})$. Thus, the error arose from the geometric approximation results in the finite element not to preserve the asymptotic compatibility.  Du et al.  \cite{DuXieYin2022, du2024errorestimatesfiniteelement} provide the detained analysis of the conditional AC  property for the approximate-ball-based finite element method for the nonlocal diffusion problem with Dirichlet boundary conditions. The AC analysis of nonlocal finite element schemes remains incomplete by using the approximate ball strategy to solve Neumann boundary problems. 

The focus of this paper is on developing an AC theory for approximate-ball-based finite element method for nonlocal problems with Neumann boundary conditions.
To the end, we present an error estimation of the inner and outer Neumann problems consisting of the model error, the finite element error and the approximate ball error by using the nonlocal Poincar\'e inequality. Numerical experiments are provided to verify the effectivess of the AC error estimate by using the Nocaps finite element method  and high-precision Approxcaps method. Numerical results are aligned with those in  \cite{du2024errorestimatesfiniteelement}. 

 The rest of this paper is organized as follows. In section 2, we first introduce the Neumann operator and the energy space, and consider the weak convergence properties of between nonlocal versus local Neumann boundary condition operators under $C^2$ space. 
 In section 3, we introduce nonlocal Neumann problems, and the corresponding local limit is studied. In section 4, we focus on investigating the geometric errors introduced by the approximate ball strategy in finite element discretization, and provide the corresponding AC  error estimate.  In section 5,  numerical experiments are provided to verify our theoretical findings. Conclusions are given in section 6 to end this paper.

\section{Nonlocal Neumann operators}
We begin with the nonlocal operator defined as 
 \begin{equation}\label{nloperator}
		\L_\delta u(\mathbf{x}):=\int_{\mathbf{y}\in  B_\delta(\mathbf{x})}[u(\mathbf{x})-u(\mathbf{y})]\gamma_\delta(\mathbf{x},\mathbf{y})d\mathbf{y},\quad \mathbf{x}\in \Omega\subset \R^d,
	\end{equation}
where $\Omega$ is an open nonempty bounded domain,  $B_\delta(\mathbf{x})$ is a ball  centered at point $\mathbf{x}$ with radius $\delta$. We then  introduce the following open sets:
\begin{eqnarray*}
&&\Omega^c=\{\mathbf{x}\in\R^d: \mathbf{x}\not\in\bar{\Omega}\},\quad\Omega^-=\{\mathbf{x}\in \Omega: B_\delta(\mathbf{x})\subset \Omega\},\quad\Omega_{\gamma}^-=\Omega\backslash\bar{\Omega}^-,\\
&&\Omega^+=\R^d\backslash\overline{\{\mathbf{x}\in \R^d: B_\delta(\mathbf{x})\cap\Omega=\emptyset\}},\quad\Omega_{\gamma}^+=\Omega^+\backslash\bar{\Omega}.
\end{eqnarray*}
Obviously, it holds that $\Omega^-\subset\Omega\subset\Omega^+$, $\Omega_\gamma^-=(\Omega^c)_\gamma^+,\quad
\Omega_\gamma^+=(\Omega^c)_\gamma^-.
$ 

The kernel function $\gamma_\delta$ we discussed in the paper is defined as follows:
\begin{definition}\label{kerneldef}
A ($d$-dimensional) kernel function $\gamma_\delta$ is referred to as a distribution on $\R^d\times \R^d$ with the following properties:
\begin{enumerate}
\item There exists $\widetilde{\gamma}_\delta(r)\in L^1(0,\delta)$, $\gamma_\delta(\mathbf{x},\mathbf{y})=\widetilde{\gamma}_\delta(|\mathbf{x}-\mathbf{y}|)$ for all $\mathbf{x},\mathbf{y}\in \R^d$. 
\item $\gamma_\delta(\mathbf{x},\mathbf{y})\ge 0$ for all $\mathbf{x},\mathbf{y}\in \R^d$;
\item The parameter $\delta$ of the kernel, referred to as the horizon, is positive, such that
$$
\gamma_\delta(\mathbf{x},\mathbf{y})=0,\quad\mathrm{if}\quad |\mathbf{x}-\mathbf{y}|>\delta,\quad \mathbf{x},\mathbf{y}\in \R^d;
$$
\item \JW{$\gamma_\delta$ has a bounded second order moment independent with the horizon $\delta$. Specifically, in two-dimensional, by using the polar coordinate transform, this implies that
\begin{align}\label{second moment}
\frac{1}{2}\int_{\R^2} \mathbf{x}\otimes \mathbf{x} \gamma_\delta(|\mathbf{x}|)d\mathbf{x}=&\frac{1}{2}\int_{\R^2} \left(\begin{array}{cc}
x_1^2  & x_1x_2  \\
x_1x_2 & x_2^2 
\end{array}\right) \gamma_\delta(|\mathbf{x}|)d\xb\nonumber\\=&\frac{1}{2}\int_{0}^{\delta}\int_{0}^{2\pi}r^3\widetilde{\gamma}_\delta(r)\left(\begin{array}{cc}
\cos^2(\theta)  & \cos(\theta)\sin(\theta)  \\
\cos(\theta)\sin(\theta) & \sin^2(\theta) 
\end{array}\right)d\theta dr \nonumber\\=&\left(\frac{\pi}{2}\int_{0}^\delta r^3 \widetilde{\gamma}_\delta(r) d r \right) I:=\sigma I.
\end{align}
}
\end{enumerate}
\end{definition}


%
For the local situation, the classical Neumann operator $\partial_\mathbf{n}$ can be derived by using the Green's first identity, i.e.,
\begin{equation*}
(-\triangle u, v )_\Omega = (\nabla u, \nabla v)_\Omega - 
\langle\partial_\mathbf{n} u, v\rangle_{\partial\Omega}. 
\end{equation*}
The nonlocal Neumann operators associated with nonlocal operator (\ref{nloperator}) can be similarly derived by 
\begin{eqnarray}\label{eqn1.3}
(\L_\delta u,v)_{\Omega}
&=&A_{\Omega,\Omega}(u,v)-(\N_{\Omega}^- u,v)_{\Omega_{\gamma}^-}\\
 \label{eqn1.4}&=&A_{\Omega,\Omega}(u,v)+2A_{\Omega_{\gamma}^-,\Omega_{\gamma}^+}(u,v)-(\N_{\Omega}^+ u,v)_{\Omega_{\gamma}^+}
\\ \label{eqn1.5}&
=&A_{\Omega^+,\Omega^+}(u,v)-(\N_{\Omega}^{*} u,v)_{\Omega_{\gamma}^+},  \label{intbypartN++}
\end{eqnarray}
where we have set
\begin{align}
&A_{\Omega_1,\Omega_2}(u,v)
=\frac{1}{2}\int_{\mathbf{x}\in\Omega_1}\int_{\mathbf{y}\in\Omega_2}[u(\mathbf{x})-u(\mathbf{y})][v(\mathbf{x})-v(\mathbf{y})]\gamma_\delta(\mathbf{x},\mathbf{y})d\mathbf{y}d\mathbf{x},\\
&\N_{\Omega}^- u(\mathbf{x})
=-\int_{\mathbf{y}\in\Omega_{\gamma}^+}[u(\mathbf{x})-u(\mathbf{y})]\gamma_\delta(\mathbf{x},\mathbf{y})d\mathbf{y},\quad \mathbf{x}\in \Omega_{\gamma}^-, \label{innbc}\\
&\N_{\Omega}^+ u(\mathbf{x})
=\int_{\mathbf{y}\in\Omega_{\gamma}^-}[u(\mathbf{x})-u(\mathbf{y})]\gamma_\delta(\mathbf{x},\mathbf{y})d\mathbf{y},\quad \mathbf{x}\in \Omega_{\gamma}^+, \label{outnbcfw}\\
&\N_{\Omega}^{*} u(\mathbf{x})=\int_{\mathbf{y}\in\Omega^+}[u(\mathbf{x})-u(\mathbf{y})]\gamma_\delta(\mathbf{x},\mathbf{y})d\mathbf{y},\quad \mathbf{x}\in \Omega_{\gamma}^+. \label{outnbc}
\end{align}
It can be observed that $A_{\Omega,\Omega},~ A_{\Omega,\Omega}+2A_{\Omega_{\gamma}^-,\Omega_{\gamma}^+}$ and $A_{\Omega^+,\Omega^+}$ are all symmetric bilinear functionals. This symmetry will be essential for the well-posedness of the nonlocal Neumann problem. \YS{Although the above definitions are valid in arbitrary spatial dimensions $d\ge 1$, 
we will henceforth restrict our attention to the two-dimensional case throughout this paper.}

The nonlocal Green's first identity (\ref{eqn1.3})-(\ref{eqn1.5}) can result in three different  nonlocal Neumann operators, i.e., $\mathcal{N}_\Omega^+$, $\mathcal{N}_\Omega^-$ and $\mathcal{N}_\Omega^*$. It is natural to ask  if these nonlocal Neumann operators converge to the classical normal derivative $\partial_\mathbf{n}$? Moreover, what is the convergence rate with respect to $\delta$ as $\delta$ tends to zero?  
For the next part of this section, we consider the local limit of nonlocal Neumann operators in two-dimensional, i.e., the asymptotically compatible analysis, by taking  $\mathcal{N}_\Omega^-$ for an example first and then $\mathcal{N}_\Omega^*$. 
\YS{To investigate the properties of the Neumann boundary operators, we define the following space:
 \begin{equation}
    S_N(\Omega) := \left\{ u \mid \sup_{x \in \Omega_\gamma^-} \left\{ |\mathcal{N}_{\Omega} u(x)| \right\} < \infty \right\},
\end{equation}
where $\N_{\Omega}$ represents the three types of nonlocal Neumann operators mentioned in (\ref{innbc})--(\ref{outnbc}). Like for the inner Neumann operator, the space are be noted as $S_N^-$.
}

\YS{Our target in the next of this section is to rigorously demonstrate the convergence result of the nonlocal Neumann operator proposed in (\ref{innbc})--(\ref{outnbc}) to the local norm derivative in the weak sense and in two-dimensional. The sketch of the proof will be devided in to two parts. In the first part , We will introduce our variable substitution strategy on a half-plane and provide a proof of the weak convergence property of the inner Neumann operator. Nextly, in the second part, the methology was extended to the convex polygons and the convex region with $C^2$ boundary for all three types of nonlocal Neumann boundary operators.}


\subsection{\YS{Local limit in the half-plane}}
\begin{figure}[h]
    \centering
    \includegraphics[width=80mm]{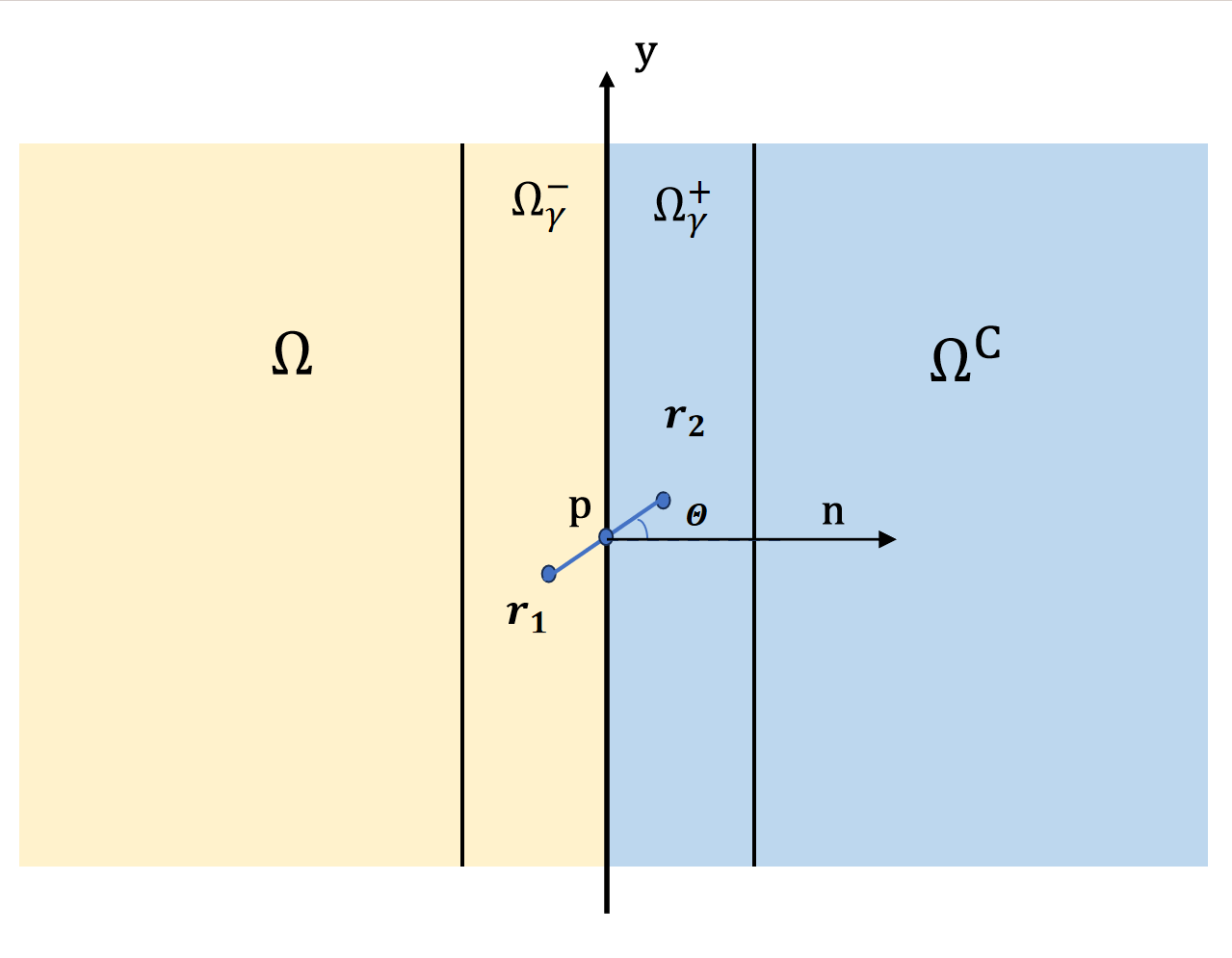}
    \caption{Schematic of the variable substitution}
    \label{fig1new:env}
 \end{figure}
 To make clear how the nonlocal Neumann operators converge to the correct local Neumann operator $\partial_\mathbf{n}$, we first consider the local limit in the half-plane domain. Taking $\Omega=(-\infty,0)\times(-\infty,+\infty)$, it has that  $\Omega_\gamma^{-}=(-\delta,0)\times(-\infty,+\infty)$ and $\Omega_\gamma^{+}=(0,\delta)\times(-\infty,+\infty)$.
 Let $\mathbf{r}_1=(x_1,y_1)\in\Omega_\gamma^{-}$ and $\mathbf{r}_2=(x_2,y_2)\in\Omega_\gamma^{+}$ be two points with $\|\mathbf{r}_2-\mathbf{r}_1\|<\delta$ for some $\delta>0$. Let $\mathbf{p}$ denote the intersection of the line segment $\overline{\mathbf{r}_1\mathbf{r}_2}$ with the boundary $\partial\Omega$. To better depict the relationship among $\mathbf{r}_1$, $\mathbf{r}_2$  and $\mathbf{p}$ (see a diagram in Figure \ref{fig1new:env}), we introduce the following variable transformation:
    \begin{align}\label{variable2}
        x_1 = -t\cos(\theta),  \; y_1 = \eta-t\sin(\theta); \; \quad x_2 = s \cos(\theta),\; y_2 = \eta+s \sin(\theta);
    \end{align}
where $\eta \in (-\infty, +\infty)$ is the $y$-coordinate of $\mathbf{p}$; \YS{$s \in (0, \delta)$ denotes the distance between $\mathbf{r}_2$ and $\mathbf{p}$ and $t \in (0, \delta - s)$ denotes the distance between $\mathbf{r}_1$ and $\mathbf{p}$};
$\theta\in (-\pi/2, \pi/2)$ is the counterclockwise angle between the line and the outer normal vector $\mathbf{n}(\mathbf{p}) = (1,0)^T$.  The corresponding Jacobian matrix is given by 
\begin{equation}\label{detJhalf}
\det(J) = \left|\frac{\partial(x_1,x_2,y_1,y_2)}{\partial(\eta,\theta,s,t)}\right|
= \left|\begin{array}{cccc}
0 & t\sin(\theta) & 0 & -\cos(\theta) \\
0 & -s\sin(\theta) & \cos(\theta) &0\\
1 & -t\cos(\theta)  & 0& -\sin(\theta)  \\
1 & s \cos(\theta)  & \sin(\theta) & 0
\end{array}\right|=(s+t)\cos(\theta). 
\end{equation}

\begin{theorem}\label{oneonone}
Assume $\Omega$ is a half plane with boundary $\Gamma = \{0\}\times(-\infty, +\infty)$, if $u \in C^2(\Omega_\gamma^+\cup\Omega_\gamma^-) \cap S_N^-(\Omega) $ and $v \in C_0^1(\Omega_\gamma^+\cup\Omega_\gamma^-)$, where the size of the support of $v$ is independent of $\delta$, then it holds 
    
\begin{equation}\label{resulthalfplanar}
    (\N_{\Omega}^- u,v)_{\Omega_{\gamma}^-}=\langle\sigma\partial_\mathbf{n} u,v\rangle_{\Gamma}+\mathcal{O}(\delta).
\end{equation}

\end{theorem}
\begin{proof}
Set $V_c(\theta):=(\cos(\theta),\sin(\theta))^T$ and 
      $   [u(\theta,\mathbf{p})]_\mathcal{H}:=(\cos(\theta),\sin(\theta)) \nabla^2u(\mathbf{p})(\cos(\theta),\sin(\theta))^T.$  Noting that $u\in C^2(\Omega_\gamma^+\cup\Omega_\gamma^-) $ and $v \in C^1(\Omega_\gamma^+\cup\Omega_\gamma^-) $, and using the variable transformation in (\ref{variable2}), we can do the Taylor expansion $u$ and $v$ at point $\mathbf{p}$ to produce 
\begin{align}\label{defweakformtaylor}
            (\N_{\Omega}^- u,v)_{\Omega_{\gamma}^-} =&\int_{s=0}^{\delta}\int_{t=0}^{\delta-s}\int_{\eta=-\infty}^{+\infty} \int_{\theta=-\frac{\pi}{2}}^{\frac{\pi}{2}}\left((s+t) V_c(\theta)^T \cdot \nabla u(\mathbf{p}(\eta)) +(s^2[u(\theta,\mathbf{\xi})]_\mathcal{H}-t^2[u(\theta,\mathbf{\zeta})]_\mathcal{H}\right) \nonumber\\&\qquad
    \det(J)\widetilde{\gamma}_\delta(s+t)\left(v(\mathbf{p}(\eta)) - tV_c(\theta)^T\cdot \nabla v(\mathbf{\mu}
    \right) d\theta d\eta dtds\nonumber \\:= & E_N,
    \end{align}
where $\xi$, $\zeta$ and $\mu$ represent the intermediate points generated in the Lagrange remainder terms of the Taylor expansions of $u(x)$, $u(y)$, and $v(x)$ at point $\mathbf{p}$, respectively.
In the next part, we will show that:
\begin{equation}\label{enestresult}
    E_N=\langle\sigma\partial_\mathbf{n} u,v\rangle_{\Gamma}+\mathcal{O}(\delta).
\end{equation}
We can divide $E_N$ into four parts, i.e., 
$
   E_N=I_1+I_2-I_3-I_4. 
$
For $I_1$, it follows from (\ref{detJhalf}) and substituting $t+s$ with $r$ that 
\begin{align}\label{firstpartana}
I_1&=\int_{s=0}^{\delta}\int_{t=0}^{\delta-s}\int_{\eta=-\infty}^{\infty} \int_{\theta=-\frac{\pi}{2}}^{\frac{\pi}{2}}(s+t) V_c(\theta)^T \cdot \nabla u(\mathbf{p}(\eta)) v(\mathbf{p}(\eta))\widetilde{\gamma}_\delta(s+t)
\det(J)d\theta d\eta dtds\nonumber\\&=\int_{r=0}^{\delta}\int_{\eta=-\infty}^{\infty} \int_{\theta=-\frac{\pi}{2}}^{\frac{\pi}{2}}r^3\widetilde{\gamma}_\delta(r)cos(\theta)V_c(\theta)^T
 \cdot \nabla u(\mathbf{p}(\eta)) v(\mathbf{p}(\eta))
  d\theta d\eta dr\nonumber\\&
=
\int_{r=0}^{\delta}\frac{\pi}{2}r^3\widetilde{\gamma}_\delta(r)dr\int_{\eta=-\infty}^{\infty} v(\mathbf{p}(\eta)) (n_{1p},n_{2p})\cdot \nabla u(\mathbf{p}(\eta)) d\eta\nonumber\\&=\sigma\int_{p\in\Gamma}\frac{\partial u}{\partial \mathbf{n}}(p) v(p)dp=\langle\sigma\frac{\partial u}{\partial \mathbf{n}},v\rangle_{\Gamma}.
\end{align}
For $I_2$, it has:
\begin{align}
I_2&=\int_{s=0}^{\delta}\int_{t=0}^{\delta-s}\int_{\eta=-\infty}^{\infty} \int_{\theta=-\frac{\pi}{2}}^{\frac{\pi}{2}}(s^2[u(\theta,\mathbf{\xi})]_\mathcal{H}-t^2[u(\theta,\mathbf{\zeta})]_\mathcal{H}) v(\mathbf{p}(\eta))\widetilde{\gamma}_\delta(s+t)
\det(J) d\theta d\eta dtds\nonumber.
\end{align}
Because of that $v \in C^1(\Omega_\gamma^+\cup\Omega_\gamma^-) $, denoting that the bounded region $V_{sp2}=\{\mathbf{p}|\mathbf{p}\in\Gamma, v(\mathbf{p})\neq 0\}$ and the bounded region $U_{sp2}=\{\mathbf{x}|\mathbf{x}\in B_\delta(\mathbf{p})|\mathbf{p}\in V_{sp2}\}$, we can conclude that:
\begin{align}\label{secondpartana}
|I_2|&=\left|\int_{s=0}^{\delta}\int_{t=0}^{\delta-s}\int_{\mathbf{p}\in V_{sp2}} \int_{\theta=-\frac{\pi}{2}}^{\frac{\pi}{2}}(s^2[u(\theta,\mathbf{\xi})]_\mathcal{H}-t^2[u(\theta,\mathbf{\zeta})]_\mathcal{H}) v(\mathbf{p})\widetilde{\gamma}_\delta(s+t)
\det(J) d\theta d\mathbf{p}dtds\right|\nonumber\\&\leq \int_{s=0}^{\delta}\int_{t=0}^{\delta-s}\int_{\mathbf{p}\in V_{sp2}} \int_{\theta=-\frac{\pi}{2}}^{\frac{\pi}{2}}4(s+t)^3\cos(\theta) |v(\mathbf{p})|\widetilde{\gamma}_\delta(s+t)
\|\nabla^2 u \|_{L^\infty(U_{sp2})}d\theta d\mathbf{p}dtds\nonumber\\&\leq 8\|\nabla^2 u \|_{L^\infty(U_{sp2})}\|v\|_{L^\infty(V_{sp2})}|V_{sp2}|\int_{r=0}^{\delta}r^4\widetilde{\gamma}_\delta(r)dr\nonumber\\&\leq C\delta \|\nabla^2 u \|_{L^\infty(U_{sp2})}\|v\|_{L^\infty(V_{sp2})}|V_{sp2}|.
\end{align}
For $I_3$, denoting that $V_{sp3}=\{\mathbf{x}|\mathbf{x}\in\Omega_\gamma^+\cup\Omega_\gamma^-,~v(x)\neq 0\}$ and $U_{sp3}=\{\mathbf{p}|B_\delta(\mathbf{p})\cap V_{sp3}\neq\emptyset\}$, it follows from substituting $t+s$ with $r$ and it has that 
\begin{align}\label{thirdpartana}
|I_3|&=\left|\int_{s=0}^{\delta}\int_{t=0}^{\delta-s}\int_{\eta=-\infty}^{\infty} \int_{\theta=-\frac{\pi}{2}}^{\frac{\pi}{2}}t(s+t) V_c(\theta)^T \cdot \nabla u(\mathbf{p}(\eta))V_c(\theta)^T\cdot \nabla v(\mathbf{\mu})\widetilde{\gamma}_\delta(s+t)
\det(J)  d\theta d\eta dtds\right|\nonumber\\&\leq 4\int_{s=0}^{\delta}\int_{t=0}^{\delta-s}t(s+t)^2\widetilde{\gamma}_\delta(s+t)dtds\int_{\mathbf{p}\in U_{sp3}} \int_{\theta=-\frac{\pi}{2}}^{\frac{\pi}{2}}\cos(\theta)\|\nabla u\|_{L^\infty(U_{sp3})} \|\nabla v\|_{L^\infty(V_{sp3})}d\mathbf{p}d\theta\nonumber\\&\leq4|U_{sp3}|\|\nabla u\|_{L^\infty(U_{sp3})} \|\nabla v\|_{L^\infty(V_{sp3})}\int_{r=0}^{\delta}r^4\widetilde{\gamma}_\delta(r)dr\nonumber\\&\leq C
\delta \|\nabla u\|_{L^\infty(U_{sp3})} \|\nabla v\|_{L^\infty(V_{sp3})}.
\end{align}
For $I_4$, it can be estimated by 
\begin{align}\label{fourthpartana}
    |I_4|&=\left|\int_{s=0}^{\delta}\int_{t=0}^{\delta-s}\int_{\eta=-\infty}^{\infty} \int_{\theta=-\frac{\pi}{2}}^{\frac{\pi}{2}}t(s^2[u(\theta,\mathbf{\xi})]_\mathcal{H}-t^2[u(\theta,\mathbf{\zeta})]_\mathcal{H})V_c(\theta)^T\cdot \nabla v(\mathbf{\mu})\widetilde{\gamma}_\delta(s+t)
\det(J) d\theta d\eta dtds\right|\nonumber\\&\leq8\int_{s=0}^{\delta}\int_{t=0}^{\delta-s}t(s+t)^3\widetilde{\gamma}_\delta(s+t)dtds\int_{\mathbf{p}\in U_{sp3}} \int_{\theta=-\frac{\pi}{2}}^{\frac{\pi}{2}}\|\nabla^2 u \|_{L^\infty(U_{sp3})}\|\nabla v\|_{L^\infty(V_{sp3})}d\mathbf{p}d\theta\nonumber\\&\leq4|U_{sp3}|\|\nabla^2 u \|_{L^\infty(U_{sp3})}\|\nabla v\|_{L^\infty(V_{sp3})}\int_{r=0}^\delta r^5\widetilde{\gamma}_\delta(r)dr\nonumber\\&\leq C\delta^2\|\nabla^2 u \|_{L^\infty(U_{sp3})}\|\nabla v\|_{L^\infty(V_{sp3})}.
\end{align}
Combining the results (\ref{firstpartana}), (\ref{secondpartana}), (\ref{thirdpartana}) and (\ref{fourthpartana}),  the proof of (\ref{enestresult}) is completed.

\end{proof}
\YS{\begin{remark}
    In this theorem, because of that the half-plane is an unbounded area, to make sure the boundness of $u$ and $v$ in the integration, we take $v\in C_0^1(\Omega_\gamma^-\cup\Omega_\gamma^+)$. However, in a bounded domain, we can directly take $v\in C_0^1(\Omega_\gamma^-\cup\Omega_\gamma^+)$ as the regularity assumption. 
\end{remark}}
\subsection{\YS{Local limit in general bounded domains}}

In practical computations, domains are typically more complicated than the half-plane, lacking the simple structure that allows for straightforward analysis. This complexity requires a more delicate, localized treatment of different parts of the region. In this section, we extend the approach developed for the half-plane to more general two-dimensional bounded domains. Specifically, we provide proofs of the weak convergence result~\eqref{C3resultnew2} for two different classes of bounded convex regions. The detailed arguments of the theorem are deferred to Appendix~\ref{append3} and Appendix~\ref{append4}.

\begin{theorem}[region with $C^2$ boundary]\label{mainresult}
    Considering a bounded region $\Omega\subset\R^2$ which is convex and the boundary $\partial \Omega$ is a closed $C^2$ curve. $u\in C^2(\Omega_\gamma^+\cup\Omega_\gamma^-)\cap S_N(\Omega)$, $v\in C^1(\Omega_\gamma^+\cup\Omega_\gamma^-)$. Then, it has the following result:
    \begin{equation}\label{C3result}
        (\N_{\Omega}u,v)_{\Omega_{\gamma}^-}=\langle\sigma\partial_\mathbf{n} u,v\rangle_{\partial\Omega}+\mathcal{O}(\delta),
    \end{equation}
    where $\N_{\Omega}$ represents the three types of nonlocal Neumann operators mentioned in (\ref{innbc})--(\ref{outnbc}).
\end{theorem}
     A detailed proof of this theorem is given in Appendix~\ref{append3}, where we focus on domains with smooth boundaries. In practical applications, however, boundaries with corners or singular points are frequently encountered. To address such situations, we establish a complementary theorem for convex polygonal domains.
\begin{theorem}[polygon] \label{rectangleright}
    Consider a convex polygon $\Omega\subset\R^2$  whose lowest vertex is the origin and one of its edges is the x-axis. Assuming that its boundary is $\Gamma_r$, and $u\in C^2(\Omega_\gamma^+\cup \Omega_\gamma^+)\cap S_N(\Omega),~v\in C^1(\Omega_\gamma^+\cup \Omega_\gamma^+)$, then it has the following result:
    \begin{equation}
       (\N_{\Omega} u,v)_{\Omega_\gamma^+}=\langle\sigma\partial_\mathbf{n} u,v\rangle_{\Gamma_r}+\mathcal{O}(\delta),
    \end{equation}
    where $\N_{\Omega}$ represents the three types of nonlocal Neumann operators mentioned in (\ref{innbc})--(\ref{outnbc}).
\end{theorem}
    The detailed proof is given in Appendix~\ref{append4}. \YS{The results in Theorem \ref{mainresult} and Theorem \ref{rectangleright} demonstrate the consistency of the three types of nonlocal Neumann operators and the local Neumann operators: for $v \in C^1(\Omega_\gamma^-\cup\Omega_\gamma^+)$, the inner Neumann operator converges weakly to the local one at rate $\mathcal{O}(\delta)$. This conclusion directly demonstrates the AC property of the given nonlocal Neumann operators and provides support for the subsequent construction of nonlocal Neumann problems.}
\begin{remark}
    Building on \cite{DuTianZhou2023}, where a similar weak convergence result was derived from the variational formulation \eqref{eqn1.3} under the condition that $u\in C^3(\Omega)$ and $v\in C^2(\Omega)$, we re-establish the result by analyzing the Neumann operator directly. Our approach yields the same convergence conclusion under weaker regularity assumptions on $u$ and $v$.
\end{remark}

\section{Second order Neumann boundary conditions  in $L^2$ norm}

We now consider the local limit of the nonlocal Poisson problem with Neumann boundary conditions. Without loss of generality, we assume $\sigma=1$ in (\ref{second moment}). The corresponding  local problem is given as \begin{equation}\label{local}
    \left\{
    \begin{aligned}
    &-\Delta u_0= f, \quad in~ \Omega,\\
    &\partial_\mathbf{n} u_0=g, \quad on ~ \partial \Omega,
\end{aligned}
    \right.
\end{equation}
 with $\int_{\mathbf{x}\in\Omega} u_0 d\mathbf{x}=0$. Here we assume the given functions $f,g$ smoothy enough such that $u_0\in C^4(\Omega).$


We also denote by $\widetilde{u_0}$ a $C^4$ extension of $u_0$ from $\Omega$ to $\Omega^+$, such that $\widetilde{u_0}|_\Omega=u_0$ and $\widetilde{u_0}\in C^4(\Omega^+)$. The corresponding nonlocal Neumann problems are defined in subsection \ref{sec21}, where we introduce the variational problem and establish the appropriate energy space for the analysis. After that, we select appropriate Neumann boundary data and analyze the compatibility conditions of the pure Neumann problem, which allows us to construct two types of nonlocal Neumann boundary problems corresponding to different boundary-definition regions. Finally, we establish the second-order asymptotic compatibility of both formulations with respect to the classical local Neumann Poisson problems.

\subsection{Neumann variational problem}\label{sec21}
We firstly consider the following nonlocal Poisson problem with Neumann operator $N^+_\Omega$ in (\ref{outnbcfw}): 
\begin{equation}
\begin{split}
&\L_\delta u(\mathbf{x})=f(\mathbf{x}),\quad \mathbf{x}\in\Omega,\\
&\N_\Omega^+u(\mathbf{x})=g(\mathbf{x}),\quad \mathbf{x}\in\Omega_\gamma^+.
\end{split}
\end{equation}
The weak form is to determine $u$, such that
\begin{eqnarray}
\label{t2}
(f,v)_{\Omega}+(g,v)_{\Omega_\gamma^+}
=A_{\Omega,\Omega}(u,v)+2A_{\Omega_{\gamma}^-,\Omega_{\gamma}^+}(u,v).
\end{eqnarray}
Let us consider the variational problem
$$
J_{\Omega^+}(u)=\frac{1}{2}\big[A_{\Omega^+,\Omega^+}(u,u)-A_{\Omega_\gamma^+,\Omega_\gamma^+}(u,u)\big]-(f,u)_\Omega-(g,u)_{\Omega_\gamma^+}\longrightarrow\inf J_{\Omega^+}(u).
$$
The critical point satisfies
$$
A_{\Omega,\Omega}(u,v)+2A_{\Omega_{\gamma}^-,\Omega_{\gamma}^+}(u,v)-(f,v)_\Omega-(g,v)_{\Omega_\gamma^+}=0,
$$
which is exactly (\ref{t2}).  Similarly, we can derive the corresponding inner and outer nonlocal Poisson problems  associated with other two Neumann operators in (\ref{innbc}) and (\ref{outnbc}) by 
\begin{equation}\label{nonlocalcpinner}
  \textit{(inner~Neumann~problem)} \quad \left\{
    \begin{aligned}
    &\L_\delta u= f_{in}, \quad in~\Omega,\\
    &\N^-_\Omega u=g_{in}, \quad in ~\Omega_\gamma^-,
\end{aligned}
    \right.
\end{equation}
and
\begin{equation}\label{nonlocalcp}
  \textit{(outer~Neumann~problem)}  \quad \left\{
    \begin{aligned}
    &\L_\delta u= f_{out}, \quad in~ \Omega,\\
    &\N^{*}_\Omega u=g_{out}, \quad in~ \Omega_\gamma^+,
\end{aligned}
    \right.
\end{equation}
with a constrained condition $\int_{\mathbf{x}\in\Omega} u d\mathbf{x}=0$ to ensure a unique solution. The energy seminorm associated with problem (\ref{nonlocalcpinner}) is as follows:
\begin{equation}
    |u|^2_{\delta,\Omega}=\int_{\Omega}\int_{\Omega}[u(\mathbf{x})-u(\mathbf{y})]^2\gamma_\delta(\mathbf{x},\mathbf{y})d\mathbf{y}d\mathbf{x},
\end{equation}
with the corresponding constrained energy space given by:
\begin{equation}
    S_{\delta}(\Omega)=\left\{u\in L^2(\Omega): |u|_{\delta,\Omega}<\infty, \int_\Omega ud\mathbf{x}=0\right\}.
\end{equation}
Correspondingly, the energy seminorm associated with problem (\ref{nonlocalcp}) is $|u|_{\delta,\Omega^+}$. The constrained energy space is that:
\begin{equation}
    S_{\delta}^*(\Omega^+)=\left\{u\in L^2(\Omega^+): |u|_{\delta,\Omega^+}<\infty, \int_{\Omega} ud\mathbf{x}=0\right\}.
\end{equation}
\begin{remark}
    It should be mentioned here that although we give out three kinds of Neumann boundary conditions, we will only analyze the inner Neumann problem (\ref{nonlocalcpinner}) with the operator $\N^-_\Omega$, and the outer Neumann problem (\ref{nonlocalcp}) with the operator $\N^*_\Omega$. It is mainly because that the operator $\N^+_\Omega$ shares similar properties with $\N^*_\Omega$ and is less representative in the existing works. 
\end{remark}


To proceed with the analysis, we next recall a nonlocal Poincar\'e inequality, which will play a crucial role in the derivation of the subsequent estimates.
\begin{lemma}[Nonlocal Poincar\'e inequality \cite{Tadele2010,ponce2004}]\label{ballpoincare}
    Assume that $\gamma$ satisfies Definition \ref{kerneldef}, then, there exists $\delta_0>0$, $C_1=C_1(\Omega)>0$, such that for all $ 0<\delta<\delta_0$,$\forall u\in S_\delta(\Omega)$, it holds that:
    \begin{equation}
        C_1\|u\|^2_{L^2(\Omega)}\leq A_{\Omega,\Omega}(u,u)=|u|_{\delta,\Omega}^2.
    \end{equation}
    Besides, if $u\in L^2(\Omega)~and~|u|_{\delta,\Omega}<\infty $, there exists  $C_2>0$ independent with $\delta$, such that
    \begin{equation}
       |u|_{\delta,\Omega}^2=A_{\Omega,\Omega}(u,u)\leq C_2\delta^{-2}\|u\|^2_{L^2(\Omega)}.
    \end{equation} 
    
\end{lemma}
\subsection{\YS{Compatible condition for the pure Neumann boundary problems}}
The asymptotic compatibility between nonlocal and local problems is an important criterion for measuring the rationality of a defined nonlocal problem. Taking the inner Neumann problem~\eqref{nonlocalcpinner} as an example, for the local problem~\eqref{local}, the Neumann boundary value $g$ is defined on the one-dimensional manifold $\partial\Omega$, 
while for the corresponding nonlocal problem, the boundary data are prescribed on the two-dimensional interaction region $\Omega_\gamma^-$. Hence, a direct analogy is not appropriate. 

From the weak convergence results established in Section~2, we know that the nonlocal Neumann operator 
$N^-_\Omega \widetilde{u_0}$ converges to the classical outer normal derivative value $g$ with first-order accuracy in the weak sense. This naturally suggests that the nonlocal boundary condition could be defined as $g_{in}=N^-_\Omega \widetilde{u_0} $. However, since we are dealing with a pure Neumann problem, one first needs to ensure the following compatible condition:
\begin{equation}
    \int_{\Omega} f_{in}d\mathbf{x}+\int_{\Omega_\gamma^-} g_{in}d\mathbf{x}=0,
\end{equation}
which indicates that the right-hand side $f_{in}$ must be modified by adding a supplementary term $df_{in}$ so as to satisfy the compatibility condition. Assume that $df_{in}$ is a constant irrelated to point $\mathbf{x}$, from the compatible condition, it has that:
\begin{equation}\label{innerconddef}
     f_{in}= f +\frac{1}{S(\Omega)} \Big(\int_{\partial\Omega} gd\mathbf{x}-\int_{\Omega_\gamma^-} N_{\Omega}^- \widetilde{u_0}d\mathbf{x}\Big):=f+df_{in}\quad in~ \Omega,\quad g_{in}=N^{-}_\Omega \widetilde{u_0} \quad in ~ \Omega_\gamma^-.
\end{equation}
Totally the same, as for problem (\ref{nonlocalcp}), we can choose the right-hand side term and Neumann boundary condition as follows:
\begin{equation}\label{outerconddef}
    \begin{aligned}
    &f_{out}= f+ \frac{1}{S(\Omega)} \Big(\int_{\partial\Omega} g d\mathbf{x}- \int_{\Omega_\gamma^+} \N_{\Omega}^{*} \widetilde{u_0} d\mathbf{x}\Big):=f+df_{out}\quad in~ \Omega,\quad g_{out}=N^{*}_\Omega \widetilde{u_0} \quad in ~ \Omega_\gamma^+.
\end{aligned}
\end{equation}
As for the estimations of the supplementary term $df_{in}$, from the definition of nonlocal operator, it has that:
    \begin{equation}
        (\L_\delta \widetilde{u_0},1)_{\Omega}=A_{\Omega,\Omega}(u_0,1)-(\N^-_\Omega\widetilde{u_0},1)_{\Omega_\gamma^-}= -(\N^-_\Omega\widetilde{u_0},1)_{\Omega_\gamma^-}.
    \end{equation}
    From Taylor expansion, when $u\in C^4(\Omega^+)$, it's easy to get the following classical results:
    \begin{equation}\label{classicalC^4result}
        \L_\delta \widetilde{u_0}(\mathbf{x})+ \Delta u_0(\mathbf{x})=\mathcal{O}(\delta^2), \quad \forall \mathbf{x}\in\Omega.
    \end{equation}
    Then, we have:
    \begin{equation}\label{df-result}
        df_{in}=\frac{1}{S(\Omega)}\left((g,1)_{\partial\Omega}-(\N^-_\Omega\widetilde{u_0},1)_{\Omega_\gamma^-}\right)=\frac{1}{S(\Omega)}(\L_\delta \widetilde{u_0}+ \Delta u_0,1)_{\Omega}=\mathcal{O}(\delta^2).
    \end{equation}
    Similarly, we can get that:
    \begin{equation}\label{df*result}
        df_{out}=\frac{1}{S(\Omega)}\left((g,1)_{\partial\Omega}-(\N^*_\Omega\widetilde{u_0},1)_{\Omega_\gamma^+}\right)=\mathcal{O}(\delta^2).
    \end{equation}
    \begin{remark}
        The regularity assumption ($u_0\in C^4(\Omega)$) in  estimation (\ref{df-result}) and (\ref{df*result}) is not a sharp assumption. Like the proof in Section 3, we can get that $df_{in/out}=\mathcal{O}(\delta^2)$ when $u\in C^3(\Omega)$. 
    \end{remark}
In the subsequent parts of this paper, when discussing the inner Neumann problem (\ref{nonlocalcpinner}) and outer Neumann problem (\ref{nonlocalcp}), we will consistently adopt the boundary conditions defined in (\ref{innerconddef}) and (\ref{outerconddef}). Our analysis will then focus on the asymptotic compatibility between the two nonlocal Neumann problems and the corresponding local problem (\ref{local}).

\subsection{ \YS{$L^2$-norm asymptotic convergence result for the inner and outer Neumann problems}}
    Assuming that $u_\delta^-$ is the solution of (\ref{nonlocalcpinner}), from (\ref{local}) and (\ref{nonlocalcpinner}), it has that:
    \begin{align}
        (\L_\delta u_\delta^-,v)_\Omega=(-\Delta u_0,v)_\Omega- \frac{1}{S(\Omega)}\left(\int_{\partial\Omega} gd\mathbf{x}- \int_{\Omega_\gamma^-} \N_{\Omega}^{-} \widetilde{u_0}d\mathbf{x},v\right)_\Omega.
    \end{align}
    By rearranging the above equation, we obtain
    \begin{equation}
        (\L_\delta (u_\delta^--\widetilde{u_0}),v)_\Omega+(\mathcal{O}(\delta^2),v)_\Omega= \frac{1}{S(\Omega)}\left( \int_{\Omega_\gamma^-} \N_{\Omega}^{-} \widetilde{u_0}d\mathbf{x}-\int_{\partial\Omega} gd\mathbf{x},v\right)_\Omega.
    \end{equation}
    Splitting the first term, from the Neumann boundary condition in problem (\ref{nonlocalcpinner}), we have:
    \begin{align}
        (\L_\delta (u_\delta^--\widetilde{u_0}),v)_\Omega&=\int_{\Omega}\int_{\Omega}[(u_\delta^--\widetilde{u_0})(\mathbf{x})-(u_\delta^--\widetilde{u_0})(\mathbf{y})]\gamma_\delta(\mathbf{x},\mathbf{y})v(\mathbf{x})d\mathbf{y}d\mathbf{x}-(\N^-_\Omega u_\delta-\N^-_\Omega \widetilde{u_0},v)_{\Omega_\gamma^-}\nonumber\\&=A_{\Omega,\Omega}(u_\delta^--u_0,v)=
        \left(\mathcal{O}(\delta^2)+\frac{ 1}{S(\Omega)}(\int_{\Omega_\gamma^-} \N_{\Omega}^{-} \widetilde{u_0}d\mathbf{x}-\int_{\partial\Omega} gd\mathbf{x}),v\right)_\Omega.
    \end{align}
    
    Taking $v=u_\delta-u_0$, from the Poincar\'e inequality (\ref{ballpoincare}) and the estimation of $df_{in}$ (\ref{df-result}), it can be obtained from the above equation that:
    \begin{align}\label{modelresultn-}
        \|u_\delta^--u_0\|^2_{L^2(\Omega)}&\leq C A_{\Omega,\Omega}(u_\delta^--\widetilde{u_0},u_\delta^--\widetilde{u_0})\nonumber\\&=C\left(\mathcal{O}(\delta^2)+\frac{1} {S(\Omega)}\left(\int_{\Omega_\gamma^-} \N_{\Omega}^{-} \widetilde{u_0}d\mathbf{x}-\int_{\partial\Omega} gd\mathbf{x}\right),u_\delta-u_0\right)_\Omega\nonumber\\&\leq C\delta^2\|u_\delta^--u_0\|_{L^2(\Omega)},
    \end{align}
    which means
    \begin{equation}
        \|u_\delta^--u_0\|_{L^2(\Omega)}\leq C\delta^2.
    \end{equation}
    It can be generalized as the following theorem:
    \begin{theorem}\label{thm41}
     Assuming that $u_0\in C^4(\Omega)$, the local problem (\ref{local}) is second-order asymptotically compatible with the nonlocal problem (\ref{nonlocalcpinner}), in other words
    \begin{equation}
        \|u_0-u_\delta^-\|_{L^2(\Omega)}=\mathcal{O}(\delta^2),
    \end{equation}
    where $u_\delta^-$ is the solution of (\ref{nonlocalcpinner}).
\end{theorem}
    Exactly the same as the inner Neumann boundary situation,  as for outer Neumann boundary operator, noting that $u_\delta^*$ is the solution of (\ref{nonlocalcp}), considering problem (\ref{local}) and (\ref{nonlocalcp}), it has that :
    \begin{equation}\label{eqn418def}
        (L_\delta (u_\delta^*-\widetilde{u_0}),v)_\Omega+(\mathcal{O}(\delta^2),v)_\Omega= \frac{1}{S(\Omega)}\left(\int_{\Omega_\gamma^+} \N_{\Omega}^{*} \widetilde{u_0}d\mathbf{x}-\int_{\partial\Omega} gd\mathbf{x} ,v\right)_\Omega=(\mathcal{O}(\delta^2),v)_\Omega.
    \end{equation}
    Splitting the first term, by the boundary condition in \ref{nonlocalcp}, (\ref{eqn418def}) can be simplified as:
    \begin{align}
        &\int_{\Omega^+}\int_{\Omega^+}[(u_\delta^*-\widetilde{u_0})(\mathbf{x})-(u_\delta^*-\widetilde{u_0})(\mathbf{y})]\gamma_\delta(\mathbf{x},\mathbf{y})v(\mathbf{x})d\mathbf{y}d\mathbf{x}=(\mathcal{O}(\delta^2),v)_\Omega.
    \end{align}
    Taking $v=u_\delta^*-\widetilde{u_0}$, from the Poincar\'e inequality (\ref{ballpoincare}) and the estimation of $df_{out}$ (\ref{df*result}), it can be obtained from the above equation that:
    \begin{align}\label{modelresult}
        \|u_\delta^*-u_0\|^2_{L^2(\Omega)}&\leq C A_{\Omega,\Omega}(u_\delta^*-\widetilde{u_0},u_\delta^*-\widetilde{u_0})\leq C A_{\Omega^+,\Omega^+}(u_\delta^*-\widetilde{u_0},u_\delta^*-\widetilde{u_0})\nonumber\\&=(\mathcal{O}(\delta^2),u_\delta-u_0)_\Omega\leq C\delta^2\|u_\delta^*-u_0\|_{L^2(\Omega)},
    \end{align}
    which means
    \begin{equation}
        \|u_\delta^*-u_0\|_{L^2(\Omega)}\leq C\delta^2.
    \end{equation}
    It can be generalized as the following theorem:
    \begin{theorem}\label{thm42}
      Assuming that $u_0\in C^4(\Omega)$, the local problem (\ref{local}) is second-order asymptotically compatible with the nonlocal problem (\ref{nonlocalcp}), in other words
    \begin{equation}
        \|u_\delta^*-u_0\|_{L^2(\Omega)}=\mathcal{O}(\delta^2),
    \end{equation}
    where $u_\delta^*$ is the solution of (\ref{nonlocalcp}).
\end{theorem}

\section{Finite element analysis}
In finite element computations for nonlocal models, standard element-wise quadrature may incur large errors due to spherical interaction neighborhoods. A common remedy is the ``approximate ball'' method of D'Elia et al.~\cite{Cookbook}, which replaces each intersection by a polygonal surrogate to enable consistent integration. 
We now consider the approximate-ball-strategy finite element method developed in \cite{Cookbook} to compute the above nonlocal Neumann problems, and present the corresponding AC error estimates. 

\subsection{Preliminaries}
We first revisit the approximate-ball FEM. Denote the spherical interaction neighborhood by
\(B_\delta(\mathbf{x})\),  its inscribed polygonal surrogate by \(B_{\delta,\mathrm{pol}}(\mathbf{x})\).
Let \(n_\delta^{\mathbf{x}}\) be the number of sides of \(B_{\delta,\mathrm{pol}}(\mathbf{x})\) and set
\(n_\delta := \sup_{\mathbf{x}\in\Omega} n_\delta^{\mathbf{x}}\).
For \(\mathbf{x}\in\Omega_\gamma^{+}\), denote \(B_{\delta,\mathrm{pol}}(\mathbf{x})\) by the inscribed
polygon of \(B_\delta(\mathbf{x})\cap\Omega^{+}\).
Precise geometric conditions (e.g., weak quasi-uniformity) will be stated later. Define the indicator and the modified polygon kernel function respectively by 
    \begin{equation*}
        \chi_{B_\delta,pol}(\mathbf{x},\mathbf{y})= \left\{
        \begin{aligned}
            &1, \quad&&\mathbf{y}\in B_{\delta,pol}(\mathbf{x}), \\
            &0, &&\mathbf{y}\notin B_{\delta,pol}(\mathbf{x}),
        \end{aligned}
        \right.
        \qquad  \gamma_{\delta,pol}(\mathbf{x},\mathbf{y})=\gamma_\delta(\mathbf{x},\mathbf{y})\chi_{B_\delta,pol}(\mathbf{x},\mathbf{y}).
    \end{equation*}

When Euclidean balls $B_\delta(x)$ are approximated by inscribed polygons, the induced
kernel $\gamma_{\delta,\mathrm{pol}}(x,y)$ on $B_\delta(x)\cap\Omega^+$ may lose the symmetry
$\gamma_{\delta,\mathrm{pol}}(x,y)=\gamma_{\delta,\mathrm{pol}}(y,x)$. Define the associated operator by
\begin{equation}\label{eq:L-pol}
\L_{\delta,\mathrm{pol}}u(x):=\int_{y\in\mathbb{R}^d}\big(u(\mathbf{x})-u(\mathbf{y})\big)\,\gamma_{\delta,\mathrm{pol}}(\mathbf{x},\mathbf{y})\,d\mathbf{y},
\qquad \mathbf{x}\in\Omega\subset\mathbb{R}^d .
\end{equation}
For a pure Neumann problem we require the compatibility
$(f,1)_\Omega + (g,1)_{\partial\Omega}=0$, which in the nonlocal formulation corresponds to
$(\L_\delta u,1)_\Omega + (\N_\Omega ^\pm u,1)_{\Omega_\gamma^\pm}=0$.
With \eqref{eq:L-pol}, taking inner Neumann problem as an example, the volume part becomes
\begin{align}\label{eqn507new}
(\L_{\delta,\mathrm{pol}}u,1)_\Omega+(\N_\Omega ^- u,1)_{\Omega_\gamma^-}
&=\frac12\int_\Omega\int_\Omega\!\big(u(\mathbf{x})-u(\mathbf{y})\big)\,\big(\gamma_{\delta,\mathrm{pol}}(\mathbf{x},\mathbf{y})-\gamma_{\delta,\mathrm{pol}}(\mathbf{y},\mathbf{x})\big)\,d\mathbf{y}\,d\mathbf{x},
\end{align}
which does not vanish in general if $\gamma_{\delta,\mathrm{pol}}$ is asymmetric.
Hence the compatibility may fail and the Neumann problem can become ill-posed/unstable.
To restore compatibility we adopt the variable–exchange symmetrization. Because of the symmetry of the kernel $\gamma_\delta$, (\ref{nloperator}) can be written as 
\begin{equation}
    \L_\delta u(\mathbf{x})=\int_{\Omega^+}(u(\mathbf{x})-u(\mathbf{y}))\frac{\gamma_\delta(\mathbf{x},\mathbf{y})+\gamma_\delta(\mathbf{y},\mathbf{x})}{2}d\mathbf{y}d\mathbf{x}.
\end{equation}
Make approximate balls to points $\mathbf{x}$ and $\mathbf{y}$ respectively, the kernel turns out to be the follows:

\begin{definition}The symmetric modified kernel is defined by 
    \begin{equation}
        \gamma_{\delta,sym}(\mathbf{x},\mathbf{y})=\frac{\gamma_{\delta,pol}(\mathbf{x},\mathbf{y})+\gamma_{\delta,pol}(\mathbf{y},\mathbf{x})}{2}.
    \end{equation}
\end{definition}
Thus, the modified  nonlocal operator for the approximate ball is as follows:
    \begin{equation}\label{nloperatoreu1}
        \L_{\delta,sym} u(\mathbf{x})=\int_{\mathbf{y}\in \R^d}[u(\mathbf{x})-u(\mathbf{y})]\gamma_{\delta,sym}(\mathbf{x},\mathbf{y})d\mathbf{y},\quad \mathbf{x}\in \Omega\subset R^d.
    \end{equation}
    Because of the symmetry of the kernel $\gamma_{\delta,pol}$, we can define the bilinear form as follows:
    \begin{equation}
        A^{sym}_{\Omega,\Omega}(u,v)=\int_\Omega\int_\Omega (u(\mathbf{x})-u(\mathbf{y}))(v(\mathbf{x})-v(\mathbf{y}))\gamma_{\delta,sym}(\mathbf{x},\mathbf{y})d\mathbf{y}d\mathbf{x}.\nonumber
    \end{equation}
    And the corresponding energy semi-norm is that 
    \begin{equation}
        |u|_{\delta,\Omega}^{sym}=\left(\int_\Omega\int_\Omega (u(\mathbf{x})-u(\mathbf{y}))^2\gamma_{\delta,sym}(\mathbf{x},\mathbf{y})d\mathbf{y}d\mathbf{x}\right)^{1/2},\nonumber
    \end{equation}
    which can introduce the energy space:
    \begin{equation}
        S_{\delta,sym}(\Omega)=\left\{ u\in L^2(\Omega):|u|_{\delta,\Omega}^{sym}<\infty,\int_\Omega ud\mathbf{x}=0\right\}.\nonumber
    \end{equation}
Thus, we can get the corresponding Neumann boundary condition by changing the nonlocal kernel: 
   \begin{equation}\label{}
    \left\{
    \begin{aligned}
    &\N^-_{\Omega,sym} u(\mathbf{x})=  -\int_{\mathbf{y}\in\Omega_{\gamma}^+}[u(\mathbf{x})-u(\mathbf{y})]\gamma_{\delta,sym}(\mathbf{x},\mathbf{y})d\mathbf{y},\quad \mathbf{x}\in \Omega_{\gamma}^-,\\
    &\N^*_{\Omega,sym} u(\mathbf{x})= \int_{\mathbf{y}\in\Omega^+}[u(\mathbf{x})-u(\mathbf{y})]\gamma_{\delta,sym}(\mathbf{x},\mathbf{y})d\mathbf{y},\quad \mathbf{x}\in \Omega_{\gamma}^+.
\end{aligned}
    \right.
\end{equation}
Finally, we can define the corresponding modified  nonlocal inner/outer Neumann Poisson equations by
\begin{equation}\label{problemsymN-}
    \left\{
    \begin{aligned}
    &\L_{\delta,sym} u= f_{in,sym}, \quad in \quad \Omega,\\
    &\N^-_{\Omega,sym} u=\N^-_{\Omega,sym} \widetilde{u_0}, \quad in \quad\Omega_\gamma^-, \quad\int_{\mathbf{x}\in\Omega} u d\mathbf{x}=0,
\end{aligned}
    \right.
\end{equation}
\begin{equation}\label{problemsymN*}
    \left\{
    \begin{aligned}
    &\L_{\delta,sym} u= f_{out,sym}, \quad in \quad \Omega,\\
    &\N^*_{\Omega,sym} u=\N^*_{\Omega,sym} \widetilde{u_0}, \quad in \quad \Omega_\gamma^+,\quad\int_{\mathbf{x}\in\Omega} u d\mathbf{x}=0,
\end{aligned}
    \right.
\end{equation}
Here the right-hand side term $f_{in,sym}$ and $f_{out,sym}$ are defined as follows:
\begin{equation}
    \left\{
    \begin{aligned}
        &f_{in,sym}:=f+\frac{1}{S(\Omega)}\left(\int_{\partial\Omega} gd\mathbf{x}-\int_{\Omega_\gamma^-} \N_{\Omega,sym}^- \widetilde{u_0}d\mathbf{x}\right):=f+df_{sym}^{in},\\
        &f_{out,sym}:=f+\frac{1}{S(\Omega)}\left(\int_{\partial\Omega} gd\mathbf{x}-\int_{\Omega_\gamma^+} \N_{\Omega,sym}^* \widetilde{u_0}d\mathbf{x}\right):=f+df_{sym}^{out}.
    \end{aligned}
    \right.
\end{equation}

By the symmetry of the kernel $\gamma_{\delta,\mathrm{sym}}$, (\ref{eqn507new}) holds identically; hence (\ref{problemsymN-}) and (\ref{problemsymN*}) are both compatible. \YS{Similar with the proof of (\ref{df-result}) and (\ref{df*result}), we state the auxiliary estimations for the compatitble term $df_{sym}^{in}$ and $df_{sym}^{out}$.}
\begin{lemma}\label{dfapproximation}
  For the compatible modified  term $df_{sym}^{in}$ and $df_{sym}^{out}$, it holds
    \begin{equation}
    \begin{aligned}
        &df_{sym}^{s}\leq \sup_{\mathbf{x}\in\Omega}   |\Delta u_0(\mathbf{x})+\L_{\delta,sym} \widetilde{u_0}(\mathbf{x})|\quad(s=in ~or~out).
    \end{aligned}
    \end{equation}
\end{lemma}
The proof of Lemma \ref{dfapproximation} is entirely the same as the proof of (\ref{df-result}) and (\ref{df*result}).

\subsection{Convergence of the nonlocal solutions with polygonal approximations}\label{section52} 

In this section, we discuss the asymptotic compatibility between solutions of the modified  inner Neumann problem (\ref{problemsymN-}) and solutions of the local Poisson-Neumann problem. Firstly, we  consider two basic bounds on the interaction kernel $\gamma_\delta$ as follows:\\
{\bf 1. Lower bound (interior positivity).} There exists a constant $c_0>0$, independent of $\delta$, such that
\begin{equation}\label{lowbound}
    \widetilde{\gamma}_\delta(r)\ \ge\ c_0\,\delta^{-4}, \qquad \forall\, r\in\Big[0,\tfrac{\delta}{2}\Big].
\end{equation}
Equivalently, $0<M_1<\widetilde{\gamma}_\delta(r)$ on $[0,\delta/2]$ with $M_1\delta^{4}=C$ for some constant $C>0$ independent of $\delta$. \\
{\bf 2. Upper bound (neighborhood mass).} There exists a constant $C>0$, independent of $\delta$, such that
\begin{equation}\label{supbound}
    G(\gamma_{\delta}) \;:=\; \sup_{{\bf x} \in \Omega^+} \int_{B_{\delta}({\bf x}) \cap \Omega^+} \gamma_{\delta}({\bf x}, {\bf y}) \, d{\bf y}
    \;\le\; C\,\delta^{-2}.
\end{equation}

We now introduce the concept of weakly quasi-uniform inscribed polygons \cite{du2024errorestimatesfiniteelement,DuXieYin2022} as follows:
\begin{definition}[\cite{DuXieYin2022,du2024errorestimatesfiniteelement}]
    A family of inscribed polygons \(\{B_{\delta,pol}(\mathbf{x})\}\)$(\mathbf{x} \in \Omega)$ is called weakly quasi-uniform if there exist two constants \(C_1\) and \(C_2 > 0\) such that \(\forall \delta > 0\), the following two bounds hold
\[
\sup_{\mathbf{x} \in \Omega} \frac{H_{\delta,\mathbf{x}}}{\delta \sin (\pi / n^x_\delta)} \leq 2C_1,
\]
where \(H_{\delta,\mathbf{x}}\) stands for the length of the longest side of the polygon \(B_{\delta,pol}(\mathbf{x})\), and

\[
\sup_{x \in \Omega} \frac{\delta}{r_{\delta,\mathbf{x}}} \leq C_2,
\]
where $r_{\delta,\mathbf{x}}$ stands for the radius of the largest inscribed circle (centered on $\mathbf{x}$) of $B_{\delta,pol}$ and we  define that 
\[
\underline{r}(n_\delta)=\inf_{\mathbf{x}\in \Omega} r_{\delta,\mathbf{x}}.
\]
\end{definition}
An illustration  of the above mentioned parameters is shown in Figure \ref{figrep:env}.
\begin{figure}[h]
    \centering
    \includegraphics[width=45mm]{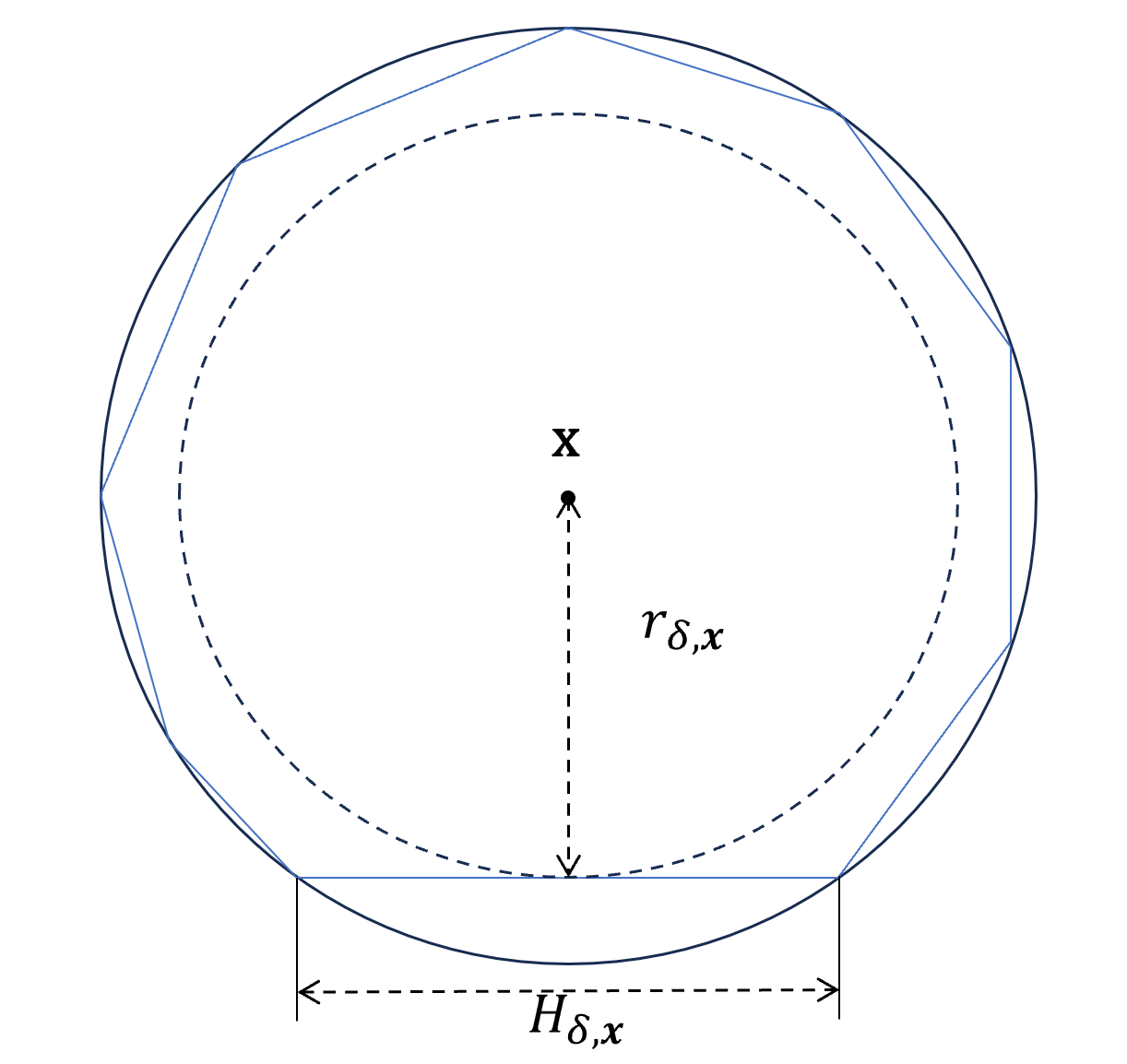}
    \caption{inscribed polygon in $B_\delta(\mathbf{x})$}
    \label{figrep:env}
    \end{figure}

Thus, we now present  the modified  Poincar\'e inequality for the symmetry kernel mentioned in \cite{DuXieYin2022}.

\begin{lemma}[\cite{DuXieYin2022,ponce2004,Tadele2010}] \label{poincaresysp}
    Assume that $\{\delta_k\}\rightarrow0$ is a sequence of nonlocal horizons. If  it holds:
    \begin{equation}
        \lim_{k \to \infty} \int_{\|\mathbf{y}-\mathbf{x}\|_2 \leq \underline{r}(n_\delta)} |y_i - x_i|^2 \gamma_{\delta_k}(\mathbf{x}, \mathbf{y}) \, d\mathbf{y} = 1, \quad i = 1, 2,
    \end{equation}
    then, there exist $\delta_0>0$ and $C>0$ such that $\forall \delta<\delta_0$ and $u\in S_{\delta,sym}(\Omega)$:
    \begin{equation}
        \|u\|_{L^2(\Omega)}\leq 
        C_1|u|^{sym}_{\delta,\Omega},
    \end{equation}
    where $C_1=C_1(\Omega,\gamma)>0$.
    Futhermore, if the conditions of Lemma \ref{ballpoincare} hold, it follows from Lemma \ref{ballpoincare}  that:
    \begin{equation}
        |u|^{sym}_{\delta,\Omega}\leq |u|_{\delta,\Omega}\leq C_2\delta^{-1}\|u\|_{L^2(\Omega)}.
    \end{equation}
\end{lemma}
\begin{remark}
    It should be further explained that in \cite{DuXieYin2022}, the modified  Poincar\'e inequality holds when the homogeneous Dirichlet boundary condition is satisfied. However, it establishes a precompact lemma which is still useful for $u\in S_{\delta,sym}(\Omega)$. It also gives out the connections between the precompact lemma and Theorem 1.2 in \cite{ponce2004}. So, combining the results in \cite{DuXieYin2022} and \cite{ponce2004}. The modified  Poincar\'e inequality above for $u\in S_{\delta,sym}(\Omega)$ can be proved  directly.
\end{remark}
We now introduce two functions in \cite{du2024errorestimatesfiniteelement} to describe the kernel, namely 
\begin{equation*}
    \gamma(s)=\delta^4\gamma_\delta(\delta s), \qquad \Phi(t)=\pi\int_0^t\rho^3\gamma(\rho)d\rho, \quad\Phi^{'}(t)=\pi t^3\gamma(t).
\end{equation*}
\begin{lemma}[\cite{du2024errorestimatesfiniteelement}]\label{Luerrorsym}
    Suppose that the kernel $\gamma_\delta $ satisfies (\ref{supbound}) and (\ref{lowbound}), $\widetilde{u_0}\in C^4(\Omega^+)$, $\{B_{\delta,pol}\}$ is a weakly quasi-uniform family of inscribed polygons. If $n_\delta\rightarrow \infty$ as $\delta\rightarrow0$, then it holds that:
    \begin{equation}\label{resultinduanalysis}
        -\Delta u_0(\mathbf{x})-\L_{\delta,sym} \widetilde{u_0}(\mathbf{x})=\mathcal{O}(\delta^{-1}n_\delta^{-\lambda})+\mathcal{O}(\delta^2), \mathbf{x}\in\Omega.
    \end{equation}
    with \(\lambda = 2\) when \(\Phi'(1) \neq 0\), while \(\lambda = 4\) when \(\Phi'(1) = 0\), \(\Phi''(1) \neq 0\). 
    \end{lemma}
    Replacing $f$ with $-\Delta u_0$ in the first equations in problem (\ref{problemsymN-}) and (\ref{problemsymN*}), then, using (\ref{resultinduanalysis}), it has 
    \begin{equation}\label{symfrontconclu}
    \left\{
    \begin{aligned}
        &\L_{\delta,sym} (u_{\delta,sym}^{-}-\widetilde{u_0})(\mathbf{x})=\mathcal{O}(\delta^{-1}n_\delta^{-\lambda})+\mathcal{O}(\delta^2)+df_{sym}^{in}, ~\mathbf{x}\in\Omega,
        \\& \L_{\delta,sym} (u_{\delta,sym}^{*}-\widetilde{u_0})(\mathbf{x})=\mathcal{O}(\delta^{-1}n_\delta^{-\lambda})+\mathcal{O}(\delta^2)+df_{sym}^{out}, ~\mathbf{x}\in\Omega.
    \end{aligned}
    \right.
    \end{equation}
From (\ref{symfrontconclu}), we have the following theorem.
\begin{theorem}\label{thm56approxball}
 Supposing the conditions of Lemma \ref{Luerrorsym} hold and noting that the solution of (\ref{problemsymN-}) is $u_{\delta,sym}^-$ and the solution of (\ref{problemsymN*}) is $u_{\delta,sym}^*$, it holds:
    \begin{equation}
    \left\{
    \begin{aligned}
        &\|u_{\delta,sym}^{-}-u_0\|_{L^2(\Omega)}=\mathcal{O}(\delta^{-1}n_\delta^{-\lambda})+\mathcal{O}(\delta^2),
         \\&\|u_{\delta,sym}^{*}-u_0\|_{L^2(\Omega)}=\mathcal{O}(\delta^{-1}n_\delta^{-\lambda})+\mathcal{O}(\delta^2).
    \end{aligned}
    \right.
    \end{equation}
\end{theorem}
\begin{proof}
    We consider the inner Neumann problem (\ref{problemsymN-}) first. By taking the inner product of $u_{\delta,sym}^{-}-u_0$ on both sides of (\ref{symfrontconclu}), it has that:
    \begin{align}\label{eqn5344}
        (\L_{\delta,sym} (u_{\delta,sym}^{-}-\widetilde{u_0}),u_{\delta,sym}^{-}-\widetilde{u_0})
        &=(\mathcal{O}(\delta^{-1}n_\delta^{-\lambda})+\mathcal{O}(\delta^2)+df_{sym}^{in}, u_{\delta,sym}^{-}-\widetilde{u_0})\nonumber\\&
        \leq C(\delta^{-1}n_\delta^{-\lambda}+\delta^2)\|u_{\delta,sym}^{-}-\widetilde{u_0}\|_{L^2(\Omega)},
    \end{align}
  where the second inequality uses the result in Lemma \ref{dfapproximation}.  By the boundary condition in (\ref{problemsymN-}), it hass
    \begin{align}
        (\L_{\delta,sym} (u_{\delta,sym}^{-}-\widetilde{u_0}),u_{\delta,sym}^{-}-\widetilde{u_0})
        &=(|u_{\delta,sym}^{-}-\widetilde{u_0}|^{sym}_{\delta,\Omega})^2-(\N^-_{\delta,sym} (u_{\delta,sym}^{-}-\widetilde{u_0}),u_{\delta,sym}^{-}-\widetilde{u_0})\nonumber\\&=(|u_{\delta,sym}^{-}-\widetilde{u_0}|^{sym}_{\delta,\Omega})^2.\nonumber
    \end{align}
    So that, from the Poincar\'e inequality in Lemma \ref{poincaresysp}, we can have the following result
    \begin{align}\label{symerrorL2}
        \|u_{\delta,sym}^{-}-\widetilde{u_0}\|_{L^2(\Omega)}^2\leq (|u_{\delta,sym}^{-}-\widetilde{u_0}|^{sym}_{\delta,\Omega})^2\leq C(\mathcal{O}(\delta^{-1}n_\delta^{-\lambda})+\mathcal{O}(\delta^2))\|u_{\delta,sym}^{-}-\widetilde{u_0}\|_{L^2(\Omega)}.
    \end{align}
Using the same argument, we obtain 
    \begin{equation}\label{symerrorL2n*}
        \|u_{\delta,sym}^{*}-\widetilde{u_0}\|_{L^2(\Omega)}^2\leq (|u_{\delta,sym}^{*}-\widetilde{u_0}|^{sym}_{\delta,\Omega^+})^2\leq C(\mathcal{O}(\delta^{-1}n_\delta^{-\lambda})+\mathcal{O}(\delta^2))\|u_{\delta,sym}^{*}-\widetilde{u_0}\|_{L^2(\Omega)}.
    \end{equation}

    This completes the proof.
\end{proof}

\subsection{Finite element analysis of the inner Neumann problem (\ref{problemsymN-})}
We now consider an error estimate of FEM for solving the nonlocal Neumann problem (\ref{problemsymN-}). 
Firstly, we consider the results without the approximate ball and define the test space as:
\begin{equation}
    S_{\delta}^\prime(\Omega)=\{u\in L^2(\Omega): |u|_{\delta,\Omega}<\infty\}.
\end{equation}
The variational form of problem (\ref{nonlocalcpinner}) is as follows:
\begin{equation}\label{nonlocalcpinnervar}
        \int_\Omega v(\mathbf{x})\left(\int_\Omega (u_\delta^-(\mathbf{x})-u_\delta^-(\mathbf{y}))\gamma_\delta(\mathbf{x},\mathbf{y})d\mathbf{y}-f_{in}\right)d\mathbf{x}=-\int_{\Omega_\gamma^-} v(\mathbf{x})\int_{\Omega^+_\gamma} (\widetilde{u_0}(\mathbf{x})-\widetilde{u_0}(\mathbf{y}))\gamma_\delta(\mathbf{x},\mathbf{y})d\mathbf{y}d\mathbf{x}, 
\end{equation}
where $\int_\Omega u_\delta^-\,d\mathbf{x}=0$ and $v\in S'_\delta(\Omega)$. 

Let $\mathcal{T}_\delta^h(\Omega)$ be an exact quasi-uniform triangulation of $\Omega$ with $n_t$ nodes, and $\{\phi_i\}_{i=1}^{n_t}$ be the continuous, piecewise-linear nodal basis. Set $V_h^\delta:=\mathrm{span}\{\phi_i\}$. 
The finite element discretization of \eqref{nonlocalcpinnervar} then reads: find $u_{h,\delta}^-\in V_h^\delta(\Omega)$ with $\int_\Omega u_{h,\delta}^-\,d\mathbf{x}=0$ such that, for all $v_h\in V_h^\delta(\Omega)$, it holds that
\begin{align}\label{nonlocalcpinnerfem}
    \frac{1}{2}&\int_\Omega\int_\Omega (v_h(\mathbf{x})-v_h(\mathbf{y}))(u_{h,\delta}^-(\mathbf{x})-u_{h,\delta}^-(\mathbf{y}))\gamma_\delta(\mathbf{x},\mathbf{y})d\mathbf{y}d\mathbf{x}\nonumber\\&=\int_\Omega v_h(\mathbf{x})(f_h(\mathbf{x})+df_h^{in})d\mathbf{x}-\int_{\Omega_\gamma^-} v_h(\mathbf{x})\int_{\Omega_\gamma^+}(\widetilde{u_{0,h}}(\mathbf{x})-\widetilde{u_{0,h}}(\mathbf{y}))\gamma_\delta(\mathbf{x},\mathbf{y})d\mathbf{y}d\mathbf{x},
\end{align}
where in (\ref{nonlocalcpinnerfem})
 $f_h$ is the projection of $f$ on space $V_h^\delta(\Omega)$. 
 
 To calculate the Neumann boundary condition term, we usually use two polygonal regions $D_o$ and $D_i$ which satisfies $D_i\subset\Omega^-\subset\Omega^+\subset D$ and define the quasi-uniform triangular mesh $\mathcal{T}_\delta^h(D)$  over $D$. Here the mesh $\mathcal{T}_\delta^h(D)$ exactly triangulates region $D_i$, $D_o$ and $\Omega$. Because that the support of the kernel $\gamma_\delta(\mathbf{x},\mathbf{y})$ is the circular neighbor of $\mathbf{x}$ $B_\delta(\mathbf{x})$, it can be found that:
 \begin{equation}
     \int_{\Omega_\gamma^-} v_h(\mathbf{x})\int_{\Omega_\gamma^+}(\widetilde{u_{0,h}}(\mathbf{x})-\widetilde{u_{0,h}}(\mathbf{y}))\gamma_\delta(\mathbf{x},\mathbf{y})d\mathbf{y}d\mathbf{x}=\int_{\Omega/D_i} v_h(\mathbf{x})\int_{\Omega_\gamma^+/\Omega}(\widetilde{u_{0,h}}(\mathbf{x})-\widetilde{u_{0,h}}(\mathbf{y}))\gamma_\delta(\mathbf{x},\mathbf{y})d\mathbf{y}d\mathbf{x}.\nonumber
 \end{equation}
 It means that the large domain $D_i$ and the small domain $D_o$ will not cause extra numerical error when we calculate the inner Neumann operator by finite element methods numerically. In (\ref{nonlocalcpinnerfem}), $\widetilde{u_{0,h}}$ is the projection of $\widetilde{u_{0}}$ on space $V_h^\delta(D)$ and $df_h^{in}$ are defined as follows:
\begin{equation}\label{nonlocalcpdf}
    df_h^{in}=\frac{1}{S(\Omega)}\Big(\int_{\partial\Omega} gd\mathbf{x}-\int_{\Omega_\gamma^-} \N_{\Omega}^- \widetilde{u_{0,h}}d\mathbf{x}\Big),
\end{equation}
where $\int_{\partial\Omega} gd\mathbf{x}$ can be calculated directly. 


\YS{\begin{theorem}[Error estimate of FEM for the inner Neumann problem without the approximate ball] \label{thmfemerror}
    Assuming that the conditions of Lemma \ref{Luerrorsym} hold, as for the solution of (\ref{nonlocalcpinnervar}) and (\ref{nonlocalcpinnerfem}), it holds
    \begin{equation}
       \|u_\delta^--u_{h,\delta}^-\|_{L^2(\Omega)}= \mathcal{O}(\delta^2)+\mathcal{O}(\delta^{-1.5}h^2),
    \end{equation}
    where $h$ is the mesh size, $u_\delta^-$ is the solution of inner Neumann problem (\ref{nonlocalcpinner}), $u_{h,\delta}^-$ is the finite element solution of inner Neumann problem in (\ref{nonlocalcpinnerfem}). 
\end{theorem}}
\begin{proof}
    For the convenience of analysis, we introduce an auxiliary problem that  $\forall v\in S^\prime_\delta(\Omega)$, finding $u_{au,\delta}^- \in S_\delta^\prime(\Omega)$, such that:
    \begin{align}\label{auxiliaryinnerfem}
        \frac{1}{2}&\int_\Omega\int_\Omega (v(\mathbf{x})-v(\mathbf{y}))(u_{au,\delta}^-(\mathbf{x})-u_{au,\delta}^-(\mathbf{y}))\gamma_\delta(\mathbf{x},\mathbf{y})d\mathbf{y}d\mathbf{x}\nonumber\\&=\int_\Omega v(\mathbf{x})(f_h(\mathbf{x})+df_h^{in})d\mathbf{x}-\int_{\Omega_\gamma^-} v(\mathbf{x})\int_{\Omega_\gamma^+}(\widetilde{u_{0,h}}(\mathbf{x})-\widetilde{u_{0,h}}(\mathbf{y}))\gamma_\delta(\mathbf{x},\mathbf{y})d\mathbf{y}d\mathbf{x},
    \end{align}
    where $u_{au,\delta}^-$ also satisfies a mean value of 0 in $\Omega$.
    Taking $v\in S_\delta^\prime(\Omega)$, it follows from (\ref{nonlocalcpinnervar}) and (\ref{auxiliaryinnerfem}) that
    \begin{align}\label{eqn549thm518}
         \frac{1}{2}&\int_\Omega\int_\Omega (v(\mathbf{x})-v(\mathbf{y}))((u_{\delta}^--u_{au,\delta}^-)(\mathbf{x})-(u_{\delta}^--u_{au,\delta}^-)(\mathbf{y}))\gamma_\delta(\mathbf{x},\mathbf{y})d\mathbf{y}d\mathbf{x}\nonumber\\&=\int_\Omega v(\mathbf{x})(f(\mathbf{x})-f_h(\mathbf{x}))d\mathbf{x}+\int_\Omega v(\mathbf{x})(df_h^{in}-df_h^{in})d\mathbf{x}+\int_{\Omega_\gamma^-} v(\mathbf{x})N^-_\Omega (\widetilde{u_0}-\widetilde{u_{0,h}})d\mathbf{x}.
    \end{align}
    Applying the interpolation error estimate and (\ref{supbound}) to the second term on the right-hand side, it has
    \begin{align}
        S(\Omega)|df^{in}-df_h^{in}|&\leq \left|\int_{\Omega_\gamma^-}\int_{\Omega_\gamma^+}((\widetilde{u_0}-\widetilde{u_{0,h}})(\mathbf{x})-(\widetilde{u_0}-\widetilde{u_{0,h}})(\mathbf{y}))\gamma_\delta(\mathbf{x},\mathbf{y})d\mathbf{y}d\mathbf{x}\right|\nonumber\\&
        \leq C|2\int_{\Omega_\gamma^-}h^2\|\widetilde{u_0}\|_{W^{2,\infty}}\int_{\Omega_\gamma^+}\gamma_\delta(\mathbf{x},\mathbf{y})d\mathbf{y}d\mathbf{x}|
        \leq C\delta^{-1}h^2.
    \end{align}
    And as for the third term of the right-hand side, we have
    \begin{align}
        |N^-_\Omega (\widetilde{u_0}-\widetilde{u_{0,h}})(x)|&=\left|\int_{\Omega_\gamma^+}((\widetilde{u_0}-\widetilde{u_{0,h}})(\mathbf{x})-(\widetilde{u_0}-\widetilde{u_{0,h}})(\mathbf{y}))\gamma_\delta(\mathbf{x},\mathbf{y})d\mathbf{y}\right|\nonumber\\&
        \leq 2Ch^2\|\widetilde{u_0}\|_{W^{2,\infty}(B_\delta(\mathbf{x}))}\int_{\Omega_\gamma^+}\gamma_\delta(\mathbf{x},\mathbf{y})d\mathbf{y}\leq Ch^2\delta^{-2}.
    \end{align}
    Thus, (\ref{eqn549thm518}) turns out to be
    \begin{align}\label{eqn555newmid}
        \frac{1}{2}&\int_\Omega\int_\Omega (v(\mathbf{x})-v(\mathbf{y}))((u_{\delta}^--u_{au,\delta}^-)(\mathbf{x})-(u_{\delta}^--u_{au,\delta}^-)(\mathbf{y}))\gamma_\delta(\mathbf{x},\mathbf{y})d\mathbf{y}d\mathbf{x}\nonumber\\&\leq C(h^2+\delta^{-1}h^2)\int_\Omega |v(\mathbf{x})|d\mathbf{x}+C\delta^{-2}h^2\int_{\Omega_\gamma^-} |v(\mathbf{x})|d\mathbf{x}.\nonumber\\&
        \leq C(h^2+\delta^{-1}h^2)\|v\|_{L^2(\Omega)}+C\delta^{-2}h^2\sqrt{S(\Omega_\gamma^-)}\|v\|_{L^2(\Omega_\gamma^-)}.\nonumber\\&
        \leq C(h^2+\delta^{-1.5}h^2)\|v\|_{L^2(\Omega)}.
    \end{align}
    From the nonlocal Poincar\'e inequality in Lemma \ref{ballpoincare}, taking $v=u_{\delta}^--u_{au,\delta}^-$, it arrives at 
    \begin{align}\label{auxresult}
        \|u_{\delta}^--u_{au,\delta}^-\|^2_{L^2(\Omega)}&\leq |u_{\delta}^--u_{au,\delta}^-|_{\delta,\Omega}
        \leq C(h^2+\delta^{-1.5}h^2)\|u_{\delta}^--u_{au,\delta}^-\|_{L^2(\Omega)}.
    \end{align}
    Because of that $V_h^\delta(\Omega)\subset S^\prime_\delta(\Omega)$, taking $v\in V_h^\delta(\Omega)$ in (\ref{nonlocalcpinnervar}) and (\ref{auxiliaryinnerfem}), and then subtracting (\ref{auxiliaryinnerfem}) from (\ref{nonlocalcpinnervar}), it has the following equation:
    \begin{equation}\label{eqn561neweq0}
        \frac{1}{2}\int_\Omega\int_\Omega (v(\mathbf{x})-v(\mathbf{y}))((u_{au,\delta}^--u_{h,\delta}^-)(\mathbf{x})-(u_{au,\delta}^--u_{h,\delta}^-)(\mathbf{y}))\gamma_\delta(\mathbf{x},\mathbf{y})d\mathbf{y}d\mathbf{x}=0, ~\forall v\in V_h^\delta(\Omega).
    \end{equation}
    Then, totally the same as the proof of (3.14) in \cite{du2024errorestimatesfiniteelement}, taking $v=u_{h,\delta}^--I_hu_0$ in (\ref{eqn561neweq0}), where $I_hu_0$ is the Lagrangian type interpolation of $u_0$ on $\mathcal{T}_\delta^h(\Omega)$, from the Cauchy-Schwartz inequality,
    we can get that:
    \begin{align}\label{minimizefemN-}
        |u_{au,\delta}^--u_{h,\delta}^-|_{\delta,\Omega}\leq |u_{au,\delta}^--I_hu_0|_{\delta,\Omega}.
    \end{align}
    With the right-hand side of the nonlocal Poincar\'e inequality in Lemma \ref{ballpoincare}, (\ref{modelresultn-}) and (\ref{auxresult}), we have
    \begin{align}\label{errorseperate}
        \|u_{au,\delta}^-&-u_{h,\delta}^-\|^2_{L^2(\Omega)}\leq (|u_{au,\delta}^--u_{h,\delta}^-|_{\delta,\Omega})^2\leq (|u_{au,\delta}^--I_hu_0|_{\delta,\Omega})^2\nonumber\\&
        \leq (|u_{au,\delta}^--u_\delta^-|_{\delta,\Omega})^2 +(|u_0-u_\delta^-|_{\delta,\Omega})^2 +(|u_0-I_hu_0|_{\delta,\Omega})^2\nonumber\\&
        \leq C(h^2+\delta^{-1.5}h^2)\|u_{\delta}^--u_{au,\delta}^-\|_{L^2(\Omega)}+C\delta^2\|u_\delta^--u_0\|_{L^2(\Omega)}+C\delta^{-2}\|u_0-I_hu_0\|_{L^2(\Omega)}^2\nonumber\\&\leq C((h^2+\delta^{-1.5}h^2)^2+\delta^4+\delta^{-2}h^4)\leq C(\delta^4+\delta^{-3}h^4).
    \end{align}
This implies
    \begin{align}
        \|u_{\delta}^--u_{h,\delta}^-\|_{L^2(\Omega)}&\leq \|u_\delta^--u_{au,\delta}^-\|_{L^2(\Omega)}+\|u_{au,\delta}^--u_{h,\delta}^-\|_{L^2(\Omega)}
        \leq C\delta^2+C\delta^{-1.5}h^2.
    \end{align}
    The proof is completed.
\end{proof}

Together with Theorem $\ref{thm41}$ , we can directly get the following corollary.
\begin{corollary}[Asymptotic compatibility for the inner Neumann problem without the approximate ball] \label{corollary41}
    \JW{Assuming that $u_0$ is the solution of the local problem and $u_{h,\delta}^-$ is the finite element solution of inner Neumann problem in (\ref{nonlocalcpinnerfem})}, if the conditions of Lemma \ref{Luerrorsym} hold, it has:
     \begin{equation}
        \|u_0-u_{h,\delta}^-\|_{L^2(\Omega)}= \mathcal{O}(\delta^2)+\mathcal{O}(\delta^{-1.5}h^2).
    \end{equation}
\end{corollary}
We point out that the AC error in Corollary \ref{corollary41} is aligned with the result in \cite{du2024errorestimatesfiniteelement}, which is underestimated. To obtain the sharp error estimate,  a nonlocal Aubin-Nitsche technique is further developed.


Next, we take the approximate ball error into consideration. The test space is defined as follows:
\begin{equation}
    S_{\delta,sym}^\prime(\Omega)=\{u\in L^2(\Omega): |u|_{\delta,\Omega}^{sym}<\infty\}.\nonumber
\end{equation}
The variational form of problem (\ref{problemsymN-}) is as follows:
\begin{equation}\label{nonlocalcpinnerballvar}
    \begin{aligned}
        &\int_\Omega v(\mathbf{x})\int_\Omega (u_{\delta,sym}^-(\mathbf{x})-u_{\delta,sym}^-(\mathbf{y}))\gamma_{\delta,sym}(\mathbf{x},\mathbf{y})d\mathbf{y}d\mathbf{x}\\&\quad=\int_\Omega v(\mathbf{x})f_{in,sym}(\mathbf{x})d\mathbf{x}-\int_{\Omega_\gamma^-} v(\mathbf{x})\int_{\Omega^+_\gamma} (\widetilde{u_0}(\mathbf{x})-\widetilde{u_0}(\mathbf{y}))\gamma_{\delta,sym}(\mathbf{x},\mathbf{y})d\mathbf{y}d\mathbf{x}, 
    \end{aligned}
\end{equation}
where the mean value of $u_{\delta,sym}^-$ on $\Omega$ is zero and $v\in S_{\delta,sym}^\prime(\Omega)$.
The finite element scheme (\ref{nonlocalcpinnerfem}) changes to the following form: Find $u^-_{h,b}\in V_h^\delta(\Omega)$ such that $\forall v_h\in V_h^\delta(\Omega)$, 
\begin{equation}\label{nonlocalcpinnerfemball}
\frac{1}{2}A_{\Omega,\Omega}^{sym}(u_{h,b}^-,v_h)=\int_{\Omega}v_h(\mathbf{x})(f_h(\mathbf{x})+df_{h,sym}^{in})d\mathbf{x}-\int_{\Omega_\gamma^-}v_h(\mathbf{x})\N^-_{\Omega,sym}u_{h,b}^-d\mathbf{x},
\end{equation}
where the mean value of $u_{h,b}^-$ on $\Omega$ is zeros and $df_{h,sym}^{in}$ is defined as follows:
\begin{equation}
    df_{h,sym}^{in}=\frac{1}{S(\Omega)}\Big(\int_{\partial\Omega} gd\mathbf{x}-\int_{\Omega_\gamma^-} \N_{\Omega,sym}^{-} \widetilde{u_{0,h}}d\mathbf{x}\Big).
\end{equation}
We now consider the convergence analysis of finite element method with the approximate ball. 
\begin{theorem}[Error estimate of FEM for inner Neumann problem with the approximate ball]\label{finalresultwithapproximateball}
    Assuming that the conditions of Lemma \ref{Luerrorsym} hold, as for the solution of (\ref{nonlocalcpinnerballvar}) and (\ref{nonlocalcpinnerfemball}), it has that:
    \begin{equation}
        \|u_{\delta,sym}^--u_{h,b}^-\|_{L^2(\Omega)}= \mathcal{O}(\delta^{-1}n_\delta^{-\lambda})+\mathcal{O}(\delta^2)+\mathcal{O}(\delta^{-1.5}h^2),
    \end{equation}
    where $u_{\delta,sym}^-$ is the solution of the modified  nonlocal inner Neumann problem (\ref{problemsymN-}) and $u_{h,b}^-$ is the finite element solution in (\ref{nonlocalcpinnerfemball}).
\end{theorem}
\begin{proof}We consider the error estimate of the inner problem first.
    For the convenience of analysis, we introduce an auxiliary problem that  $\forall v\in S^\prime_{\delta,sym}(\Omega)$, finding $u_{ab,\delta}^- \in S^\prime_{\delta,sym}(\Omega)$, such that:
    \begin{align}\label{auxiliaryinnerfemball}
        \frac{1}{2}&\int_\Omega\int_\Omega (v(\mathbf{x})-v(\mathbf{y}))(u_{ab,\delta}^-(\mathbf{x})-u_{ab,\delta}^-(\mathbf{y}))\gamma_{\delta,sym}(\mathbf{x},\mathbf{y})d\mathbf{y}d\mathbf{x}\nonumber\\&=\int_\Omega v(\mathbf{x})(f_h(\mathbf{x})+df_{h,sym}^{in})d\mathbf{x}-\int_{\Omega_\gamma^-} v(\mathbf{x})\int_{\Omega_\gamma^+}(\widetilde{u_{0,h}}(\mathbf{x})-\widetilde{u_{0,h}}(\mathbf{y}))\gamma_{\delta,sym}(\mathbf{x},\mathbf{y})d\mathbf{y}d\mathbf{x}.
    \end{align}
    where $u_{ab,\delta}^-$ also satisfies a mean value of 0 in $\Omega$.
    Taking $v=u_{\delta}^--u_{ab,\delta}^-\in S_{\delta,sym}(\Omega)$, because of the symmetry and the integration boundedness of the kernel $\gamma_{\delta,sym}$. Using the same argument as the proof in Theorem \ref{thmfemerror}, it follows from the Poincar\'e inequality in Lemma \ref{poincaresysp} that:
    \begin{align}\label{auxresultball}
        \|u_{\delta}^--u_{ab,\delta}^-\|^2_{L^2(\Omega)}&\leq |u_{\delta,sym}^--u_{ab,\delta}^-|_{\delta,\Omega}^{sym}
         \leq C(h^2+\delta^{-1.5}h^2)\|u_{\delta,sym}^--u_{ab,\delta}^-\|_{L^2(\Omega)}.
    \end{align}
    Similar with the proof of (\ref{minimizefemN-}) in \cite{du2024errorestimatesfiniteelement}, it has that
    \begin{align}
        |u_{ab,\delta}^--u_{h,\delta}^-|_{\delta,\Omega}^{sym}\leq |u_{ab,\delta}^--I_hu_0|_{\delta,\Omega}^{sym}.\nonumber
    \end{align}
    With the right-hand side of the nonlocal Poincar\'e inequality in Lemma \ref{poincaresysp}, (\ref{auxresultball}) and (\ref{symerrorL2}), we have
    \begin{align}\label{errorseperateball}
        \|u_{ab,\delta}^-&-u_{h,b}^-\|^2_{L^2(\Omega)}\leq (|u_{ab,\delta}^--u_{h,b}^-|_{\delta,\Omega}^{sym})^2\leq (|u_{ab,\delta}^--I_hu_0|_{\delta,\Omega}^{sym})^2\nonumber\\&
        \leq (|u_{ab,\delta}^--u_{\delta,sym}^-|_{\delta,\Omega}^{sym})^2 +(|u_0-u_{\delta,sym}^-|_{\delta,\Omega}^{sym} )^2+(|u_0-I_hu_0|_{\delta,\Omega}^{sym})^2\nonumber\\&
        \leq C(h^2+\delta^{-1.5}h^2)\|u_{\delta,sym}^--u_{ab,\delta}^-\|_{L^2(\Omega)}\nonumber\\&\quad+C(\delta^2+\delta^{-1}n_\delta^{-\lambda})\|u_{\delta,sym}^--u_0\|_{L^2(\Omega)}+C\delta^{-2}\|u_0-I_hu_0\|_{L^2(\Omega)}^2\nonumber\\&\leq C((h^2+\delta^{-1.5}h^2)^2+(\delta^2+\delta^{-1}n_\delta^{-\lambda})^2+\delta^{-2}h^4)\leq C(\delta^{-2}n_\delta^{-2\lambda}+\delta^4+\delta^{-3}h^4),
    \end{align}
     where $I_hu_0$ is the Lagrangian type interpolation of $u_0$ on $\mathcal{T}_\delta^h(\Omega)$. This implies
    \begin{align}
        \|u_{\delta,sym}^--u_{h,b}^-\|_{L^2(\Omega)}&\leq \|u_{\delta,sym}^--u_{ab,\delta}^-\|_{L^2(\Omega)}+\|u_{ab,\delta}^--u_{h,b}^-\|_{L^2(\Omega)}\nonumber\\&
        \leq C(\delta^{-1}n_\delta^{-\lambda}+\delta^2+\delta^{-1.5}h^2).
    \end{align}
    The proof is completed.
\end{proof}

From the above error estimation, together with Theorem $\ref{thm56approxball}$, we have the following corollary:
\begin{corollary}[Asymptotic compatibility for the inner Neumann problem with the approximate ball] \label{finalresultwithapproximateballcol}
Assuming that $u_0$ is the solution of the local problem and $u_{h,\delta}^-$ is the finite element solution of inner Neumann problem in (\ref{nonlocalcpinnerfem}), if the conditions of Lemma \ref{Luerrorsym} hold, it has:
    \begin{equation}
        \|u_0-u_{h,b}^-\|_{L^2(\Omega)}= \mathcal{O}(\delta^{-1}n_\delta^{-\lambda})+\mathcal{O}(\delta^2)+\mathcal{O}(\delta^{-1.5}h^2).
    \end{equation}
    When we choose that $n_\delta=\mathcal{O}(\delta/h)$, it has that:
    \begin{equation}
       \|u_0-u_{h,b}^-\|_{L^2(\Omega)}= \mathcal{O}(\delta^{-1-\lambda}h^\lambda)+\mathcal{O}(\delta^2)+\mathcal{O}(\delta^{-1.5}h^2).
    \end{equation}
\end{corollary}

\subsection{Finite element analysis of the outer Neumann problem (\ref{problemsymN*})}
We now consider the asymptotic compatibility analysis for the outer Neumann problem (\ref{problemsymN*}). In this situation, the point over $\Omega_\gamma^+$ participates in the actual calculation. 

It is generally hard to triangulate $\Omega$ and $\Omega^+$ exactly simultaneously.  To overcome this difficulty, we introduce a quasi-uniform triangular mesh $\mathcal{T}_\delta^h(D)$ over an extended bounded convex domain $D$ which is exactly triangulated by the mesh and satisfies $\Omega^+ \subset D$. This mesh also exactly triangulates domain $\Omega $. 
In addition, we assume that the region $D$ satisfies 
$D \subset \widehat{\Omega}_{\gamma,5/4}^+
:= \Omega^+ \cup
\left\{ \boldsymbol{x} \in \R^2\;\big|\;
  \delta \le \operatorname{dist}(\boldsymbol{x},\partial\Omega)
  < \tfrac{5}{4}\delta
\right\}.$ Assuming that $\widetilde{u_0}$ is the $C^4$ extension of the local solution $u_0$ on $D$, the Neumann boundary operator that we used in the actual finite element calculations is that:
\begin{equation}
    \N^{*,r}_{D} u(\mathbf{x})=\int_{D\cap B_\delta(\mathbf{x})}(u(\mathbf{x})-u(\mathbf{y}))\gamma_\delta(\mathbf{x},\mathbf{y})d\mathbf{y} \quad (\mathbf{x}\in D/\Omega).
\end{equation}
Then, the outer Neumann problem defined on $D$ is as follows:
\begin{equation}\label{problemNbigD}
    \left\{
    \begin{aligned}
    &\L_{\delta} u= f_{out}^{p}, \quad in\quad \Omega,\\
    &\N^{*,r}_{D} u=\N^{*,r}_{D} \widetilde{u_0}, \quad in \quad D/\Omega,\quad\int_{\mathbf{x}\in\Omega} u d\mathbf{x}=0.
\end{aligned}
    \right.
\end{equation}
Here the right-hand side term $f_{out}^{p}$ are defined as follows:
\begin{equation}\label{dfDdef}
    f_{out}^{p}:=f+\frac{1}{S(\Omega)}\Big(\int_{\partial\Omega} gd\mathbf{x}-\int_{D/\Omega} \N_{D}^{*,r} \widetilde{u_0}d\mathbf{x}\Big):=f+df_{out}^{p}.
\end{equation}
\begin{theorem}[Asymptotic compatibility for the outer Neumann problem defined on $D$]\label{thm59regionD}
    Assuming that $\gamma_\delta$ satisfies (\ref{supbound}) and (\ref{lowbound}), $u_0\in C^4(\Omega)$, the local problem (\ref{local}) is second-order asymptotically compatible with the nonlocal problem (\ref{problemNbigD}), in other words
    \begin{equation}
        \|u_0-u_{\delta}^{*,D}\|_{L^2(\Omega)}=\mathcal{O}(\delta^2),
    \end{equation}
    where $u_{\delta}^{*,D}$ is the solution of (\ref{problemNbigD}).
\end{theorem}
\begin{proof}
    Similar to the proof of Theorem \ref{thm42}, we first present the approximation of the modified  term $df_{out}^p$. It follows from the first equation in (\ref{local}) and (\ref{problemNbigD}) and the classical result (\ref{classicalC^4result}) that:
    \begin{align}\label{N^*considerbiggerregion}
        \L_\delta u_{\delta}^{*,D}=-\Delta u_0+df_{out}^p=\L_\delta \widetilde{u_0}+df_{out}^p+\mathcal{O}(\delta^2).
    \end{align}
By the definition of the nonlocal operator, it has:
    \begin{equation}
        (\L_\delta \widetilde{u_0},1)_{\Omega}=A_{D,D}(\widetilde{u_0},1)-(\N^{*,r}_D\widetilde{u_0},1)_{D/\Omega}= -(\N^{*,r}_D\widetilde{u_0},1)_{D/\Omega}.\nonumber
    \end{equation}
    Then we can get the following equation:
    \begin{align}
        S(\Omega)df_{out}^p=\int_{\partial\Omega} gd\mathbf{x}-\int_{D/\Omega} \N_{D}^{*,r} \widetilde{u_0}d\mathbf{x}=(\Delta u_0+\L_\delta \widetilde{u_0},1)_{\Omega}=\mathcal{O}(\delta^2).\nonumber
    \end{align}
    By taking the inner product of $u_{\delta}^{*,D}-u_0$ on both sides of (\ref{N^*considerbiggerregion}), using  the same as the proof of Theorem \ref{thm42}, and by the Poincar\'e inequality in Lemma \ref{poincaresysp}, it has 
    \begin{align}\label{DerrorL2}
         \|u_{\delta}^{*,D}-u_0\|_{L^2(\Omega)}^2\leq (|u_{\delta}^{*,D}-\widetilde{u_0}|^D_{\delta})^2=(\L_{\delta} (u_{\delta}^{*,D}-\widetilde{u_0}),u_{\delta}^{*,D}-u_0)_\Omega
        \leq C\delta^2\|u_{\delta}^{*,D}-u_0\|_{L^2(\Omega)}.
    \end{align}
    This completes the proof.
\end{proof}

The theorem above reveals that when we use a larger domain $D$ instead of $\Omega^+$ in the finite element scheme, the model error is still of the order $\delta^2$ and no additional geometric error is introduced. In the actual calculations and the analysis that follows, we will all consider the case where the problem is defined over the large region $D$. The variational form of problem (\ref{problemNbigD}) is as follows:
\begin{equation}\label{nonlocalcpvar}
    \begin{aligned}
        &\int_D v(\mathbf{x})\int_D (u_\delta^{*,D}(\mathbf{x})-u_\delta^{*,D}(\mathbf{y}))\gamma_\delta(\mathbf{x},\mathbf{y})d\mathbf{y}d\mathbf{x}\\&\quad=\int_\Omega v(\mathbf{x})f_{out}^p(\mathbf{x})d\mathbf{x}+\int_{D/\Omega} v(\mathbf{x})\int_{D} (\widetilde{u_0}(\mathbf{x})-\widetilde{u_0}(\mathbf{y}))\gamma_\delta(\mathbf{x},\mathbf{y})d\mathbf{y}d\mathbf{x}, 
    \end{aligned}
\end{equation}
where the mean value of $u_{\delta}^{*,D}$ on $\Omega$ is zero and $v\in S_{\delta}^\prime(\Omega)$.
Let $\{\phi_i\}_{i=1}^{n_t^D}$ denote the set of continuous linear basis functions defined on $\mathcal{T}_\delta^h(D)$, $n_t^D$ is the number of mesh nodes. The finite element space $V_h^\delta$ is then the linear span of $\{\phi_i\}_{i=1}^{n_t^D}$. 
Then, the finite element scheme for (\ref{nonlocalcpvar}) is to find $u_{h,\delta}^{*,D}\in V_{h}^\delta(D)$ such that:
\begin{align}\label{nonlocalcpouterfem}
    \frac{1}{2}&\int_{D}\int_{D} (v_h(\mathbf{x})-v_h(\mathbf{y}))(u_{h,\delta}^{*,D}(\mathbf{x})-u_{h,\delta}^{*,D}(\mathbf{y}))\gamma_\delta(\mathbf{x},\mathbf{y})d\mathbf{y}d\mathbf{x}\nonumber\\&=\int_\Omega v_h(\mathbf{x})(f_h(\mathbf{x})+df_{h}^{out,p})d\mathbf{x}+\int_{D/\Omega} v_h(\mathbf{x})\int_{D}(\widetilde{u_{0,h}}(\mathbf{x})-\widetilde{u_{0,h}}(\mathbf{y}))\gamma_\delta(\mathbf{x},\mathbf{y})d\mathbf{y}d\mathbf{x},
\end{align}
 and $df_{h}^{out,p}$
\begin{equation}
    df_{h}^{out,p}=\frac{1}{S(\Omega)}\Big(\int_{\partial\Omega} gd\mathbf{x}-\int_{D} \N_{\Omega}^{*,r} \widetilde{u_{0,h}}d\mathbf{x}\Big).
\end{equation}

Next, we will give out the error estimation between the finite element solution and local solution. Firstly, we will introduce a lemma that bridges the solution between the outer boundary layer and the inner region.
\begin{lemma}\cite{Tadele2010,Andreu2009}\label{middlelemma}
    Suppose that $\Omega$ is a convex region, $\gamma_\delta $ satisfies (\ref{supbound}) and (\ref{lowbound}), $u\in L^2(D)$, and $|u|_{\delta,\Omega^+}<\infty$. Then, there exists a constant $C=C(\Omega)$, such that the following inequality holds:
    \begin{equation}
        \|u\|^2_{L^2(D/\Omega)}\leq C(\delta^2A_{D/\Omega,D}(u,u)+\|u\|^2_{L^2(\Omega_{\gamma,1/4}^-)}),
    \end{equation}
    where $\Omega^-_{\gamma,1/4}:=\{\mathbf{x}\in\Omega:dist(\mathbf{x},\partial\Omega)<\frac{\delta}{4}\}$.
\end{lemma}
\begin{proof}
    Here we note that $\Omega^+_{\gamma,i/4}:=\{\mathbf{x}\in\Omega_\gamma^+:dist(\mathbf{x},\partial\Omega)<i\frac{\delta}{4}\}$. When $i=1$, it has that:
   \begin{align}\label{eqn543new}
       \int_{\Omega^+_{\gamma,1/4}}u(\mathbf{x})^2d\mathbf{x}&=\int_{\Omega^+_{\gamma,1/4}}(\int_{\Omega^-_{\gamma,1/4}\cap B_{\delta/2}(\mathbf{x})}u(\mathbf{x})-u(\mathbf{y})+u(\mathbf{y})d\mathbf{y})^2\frac{1}{S(\Omega^-_{\gamma,1/4}\cap B_{\delta/2}(\mathbf{x}))^2}d\mathbf{x}\nonumber\\&\leq \int_{\Omega^+_{\gamma,1/4}}\int_{\Omega^-_{\gamma,1/4}\cap B_{\delta/2}(\mathbf{x})}(u(\mathbf{x})-u(\mathbf{y})+u(\mathbf{y}))^2d\mathbf{y}\frac{1}{S(\Omega^-_{\gamma,1/4}\cap B_{\delta/2}(\mathbf{x}))}d\mathbf{x}\nonumber\\&\leq 2\int_{\Omega^+_{\gamma,1/4}}\int_{\Omega^-_{\gamma,1/4}\cap B_{\delta/2}(\mathbf{x})}(u(\mathbf{x})-u(\mathbf{y}))^2d\mathbf{y}\frac{1}{S(\Omega^-_{\gamma,1/4}\cap B_{\delta/2}(\mathbf{x}))}d\mathbf{x}\nonumber\\&\quad+ 2\int_{\Omega^+_{\gamma,1/4}}\int_{\Omega^-_{\gamma,1/4}\cap B_{\delta/2}(\mathbf{x})}u(\mathbf{y})^2d\mathbf{y}\frac{1}{S(\Omega^-_{\gamma,1/4}\cap B_{\delta/2}(\mathbf{x}))}d\mathbf{x}.
   \end{align}
   From boundedness of the kernel in (\ref{lowbound}), it has that:
   \begin{align}
       \int_{\Omega^+_{\gamma,1/4}}u(\mathbf{x})^2d\mathbf{x}&\leq 2\int_{\Omega^+_{\gamma,1/4}}\int_{\Omega^-_{\gamma,1/4}\cap B_{\delta/2}(\mathbf{x})}(u(\mathbf{x})-u(\mathbf{y}))^2\gamma_\delta(\mathbf{x},\mathbf{y})d\mathbf{y}\frac{1}{M_1S(\Omega^-_{\gamma,1/4}\cap B_{\delta/2}(\mathbf{x}))}d\mathbf{x}\nonumber\\&\quad+ 2\int_{\Omega^+_{\gamma,1/4}}\int_{\Omega^-_{\gamma,1/4}\cap B_{\delta/2}(\mathbf{x})}u(\mathbf{y})^2d\mathbf{y}\frac{1}{S(\Omega^-_{\gamma,1/4}\cap B_{\delta/2}(\mathbf{x}))}d\mathbf{x}\nonumber.
   \end{align}
   As for $S(\Omega^-_{\gamma,1/4}\cap B_{\delta/2}(\mathbf{x}))$, because that $\Omega$ is a convex region, from the interior cone condition, it has:
   \begin{equation}
       C_1(\Omega)\delta^2\leq S(\Omega^-_{\gamma,1/4}\cap B_{\delta/2}(\mathbf{x}))\leq C_2(\Omega)\delta^2.\nonumber
   \end{equation}
   Then, together with the asymptotically compatible lower bound of $\gamma_\delta$, by exchanging the integration order, it can be obtained that:
   \begin{align}
       \int_{\Omega^+_{\gamma,1/4}}u(\mathbf{x})^2d\mathbf{x}&\leq C\delta^2A_{\Omega^-_{\gamma,1/4},\Omega^+_{\gamma,1/4}}(u,u)+\frac{2S(B_{\delta/2}(\mathbf{x}))}{S(\Omega^-_{\gamma,1/4}\cap B_{\delta/2}(\mathbf{x}))}\int_{\Omega^-_{\gamma,1/4}}u(\mathbf{x})^2d\mathbf{x}\nonumber\\&\leq C(\delta^2 A_{\Omega^+,\Omega^+_{\gamma}}(u,u)+\int_{\Omega^-_{\gamma,1/4}}u(\mathbf{x})^2d\mathbf{x}).\nonumber
   \end{align}
   Totally the same, if we note that $\hat{\Omega}^+_{\gamma,i/4}:=\Omega^+_{\gamma,i/4}-\Omega^+_{\gamma,(i-1)/4},(i=2,3,4,5)$ and $\hat{\Omega}^+_{\gamma,1/4}=\Omega^+_{\gamma,1/4}$, as $i=2,3,4,5$, from the convexity of $\Omega$, we can get the following inequality:
   \begin{equation}
       \int_{\hat{\Omega}^+_{\gamma,i/4}\cap D}u(\mathbf{x})^2d\mathbf{x}\leq C(\delta^2 A_{D,D/\Omega}(u,u)+\int_{\hat{\Omega}^+_{\gamma,(i-1)/4}}u(\mathbf{x})^2d\mathbf{x})\leq C(\delta^2 A_{D,D/\Omega}(u,u)+\int_{\Omega^-_{\gamma,(i-1)/4}}u(\mathbf{x})^2d\mathbf{x}).\nonumber
   \end{equation}
   Summing both sides of the inequality over $i$ at the same time and the conclusion can be got.
\end{proof}

Then, we can give a finite element analysis of the outer Neumann problem (\ref{problemNbigD}) without considering the approximate ball. 

\begin{theorem}\label{thmfemerrorN*D}(Finite element error for the outer Neumann problem without considering the approximate ball )
    Assuming that the conditions of Lemma \ref{Luerrorsym} hold, as for the solution of (\ref{nonlocalcpvar}) and (\ref{nonlocalcpouterfem}), it has that:
    \begin{equation}
            \|u_\delta^{*,D}-u_{h,\delta}^{*,D}\|_{L^2(\Omega)}= \mathcal{O}(\delta^2)+\mathcal{O}(\delta^{-1.5}h^2),
    \end{equation}
    where $h$ is the maximum grid side length, $u_\delta^{*,D}$ is the solution of the outer Neumann problem defined on $D$ and $u_{h,\delta}^{*,D}$ is the finite element solution of (\ref{nonlocalcpouterfem}). 
\end{theorem}
\begin{proof}
Again, we introduce an auxiliary problem that  $\forall v\in S_\delta^\prime(D)$, finding $u_{au,\delta}^{*,D} \in S_\delta^\prime(D)$, such that:
    \begin{align}\label{auxiliaryouterfem}
        \frac{1}{2}&\int_{D}\int_{D} (v(\mathbf{x})-v(\mathbf{y}))(u_{au,\delta}^{*,D}(\mathbf{x})-u_{au,\delta}^{*,D}(\mathbf{y}))\gamma_\delta(\mathbf{x},\mathbf{y})d\mathbf{y}d\mathbf{x}\nonumber\\&=\int_\Omega v(\mathbf{x})(f_h(\mathbf{x})+df_h^{out,p})d\mathbf{x}+\int_{D/\Omega} v(\mathbf{x})\int_{D}(\widetilde{u_{0,h}}(\mathbf{x})-\widetilde{u_{0,h}}(\mathbf{y}))\gamma_\delta(\mathbf{x},\mathbf{y})d\mathbf{y}d\mathbf{x},
    \end{align}
    where $u_{au,\delta}^{*,D}$ also satisfies a mean value of 0 in $\Omega$. The same as the proof of inequality (\ref{eqn555newmid}), it has     \begin{align}
        \frac{1}{2}&\int_{D}\int_{D} ((u_{\delta}^{*,D}-u_{au,\delta}^{*,D})(\mathbf{x})-(u_{\delta}^{*,D}-u_{au,\delta}^{*,D})(\mathbf{y}))^2\gamma_\delta(\mathbf{x},\mathbf{y})d\mathbf{y}d\mathbf{x}\nonumber\\&\leq
        C(h^2+\delta^{-1}h^2)\|u_{\delta}^{*,D}-u_{au,\delta}^{*,D}\|_{L^2(\Omega)}+C\delta^{-2}h^2\sqrt{S(D/\Omega)}\|u_{\delta}^{*,D}-u_{au,\delta}^{*,D}\|_{L^2(D/\Omega)}.
    \end{align}
   
    From Lemma \ref{middlelemma}, we can get that:
    \begin{align}\label{eqn562new}
        \frac{1}{2}&\int_{D}\int_{D} ((u_{\delta}^{*,D}-u_{au,\delta}^{*,D})(\mathbf{x})-(u_{\delta}^{*,D}-u_{au,\delta}^{*,D})(\mathbf{y}))^2\gamma_\delta(\mathbf{x},\mathbf{y})d\mathbf{y}d\mathbf{x}\nonumber\\&\leq C\delta^{-1.5}h^2\sqrt{\delta^2A_{D/\Omega,D}(u_{\delta}^{*,D}-u_{au,\delta}^{*,D},u_{\delta}^{*,D}-u_{au,\delta}^{*,D})+\|u_{\delta}^{*,D}-u_{au,\delta}^{*,D}\|^2_{L^2(\Omega_{\gamma,1/4}^-)}}
        \nonumber\\&\quad+C(h^2+\delta^{-1}h^2)\|u_{\delta}^{*,D}-u_{au,\delta}^{*,D}\|_{L^2(\Omega)}\nonumber\\&\leq C(h^2+\delta^{-1}h^2+\delta^{-1.5}h^2)\sqrt{A_{D/\Omega,D}(u_{\delta}^{*,D}-u_{au,\delta}^{*,D},u_{\delta}^{*,D}-u_{au,\delta}^{*,D})+\|u_{\delta}^{*,D}-u_{au,\delta}^{*,D}\|^2_{L^2(\Omega)}}.
    \end{align}
     
    Mentioned that $A_{\Omega,\Omega}(u,u)+A_{D/\Omega,D}(u,u)\leq A_{D,D}(u,u)$, from equation (\ref{eqn562new}), we can get that
    \begin{align}
        \|u_{\delta}^{*,D}-u_{au,\delta}^{*,D}&\|^2_{L^2(\Omega)}\leq \|u_{\delta}^{*,D}-u_{au,\delta}^{*,D}\|^2_{L^2(\Omega)}+A_{D/\Omega,D}(u_{\delta}^{*,D}-u_{au,\delta}^{*,D},u_{\delta}^{*,D}-u_{au,\delta}^{*,D})\nonumber\\&
        \leq A_{\Omega,\Omega}(u_{\delta}^{*,D}-u_{au,\delta}^{*,D},u_{\delta}^{*,D}-u_{au,\delta}^{*,D})+A_{D/\Omega,D}(u_{\delta}^{*,D}-u_{au,\delta}^{*,D},u_{\delta}^{*,D}-u_{au,\delta}^{*,D})\nonumber\\&\leq A_{D,D}(u_{\delta}^{*,D}-u_{au,\delta}^{*,D},u_{\delta}^{*,D}-u_{au,\delta}^{*,D})\nonumber\\&\leq C(h^2+\delta^{-1.5}h^2)\sqrt{A_{D/\Omega,D}(u_{\delta}^{*,D}-u_{au,\delta}^{*,D},u_{\delta}^{*,D}-u_{au,\delta}^{*,D})+\|u_{\delta}^{*,D}-u_{au,\delta}^{*,D}\|^2_{L^2(\Omega)}},
    \end{align}
    which implies that
    \begin{align}\label{estimationneedalotstr}
        \|u_{\delta}^{*,D}-u_{au,\delta}^{*,D}\|^2_{L^2(\Omega)}&\leq A_{\Omega,\Omega}(u_{\delta}^{*,D}-u_{au,\delta}^{*,D},u_{\delta}^{*,D}-u_{au,\delta}^{*,D})+A_{D/\Omega,D}(u_{\delta}^{*,D}-u_{au,\delta}^{*,D},u_{\delta}^{*,D}-u_{au,\delta}^{*,D})\nonumber\\&\leq C(h^2+\delta^{-1.5}h^2)^2.
    \end{align}
By the nonlocal Poincar\'e inequality in Lemma \ref{ballpoincare}, Theorem \ref{thm59regionD} and  (\ref{estimationneedalotstr}), we have
    \begin{align}\label{errorseperateN*}
        \|u_{au,\delta}^{*,D}&-u_{h,\delta}^{*,D}\|^2_{L^2(\Omega)}\leq (|u_{au,\delta}^{*,D}-u_{h,\delta}^{*,D}|_\delta^{D})^2\leq (|u_{au,\delta}^{*,D}-I_hu_0|_\delta^{D})^2\nonumber\\&
        \leq (|u_{au,\delta}^{*,D}-u_\delta^{*,D}|_\delta^{D})^2 +(|u_0-u_\delta^{*,D}|_\delta^{D})^2 +(|u_0-I_hu_0|_\delta^{D})^2\nonumber\\&
        \leq C(h^2+\delta^{-1.5}h^2)^2+C\delta^4+C\delta^{-2}\|u_0-I_hu_0\|_{L^2(D)}^2\nonumber\\&\leq C((h^2+\delta^{-1.5}h^2)^2+\delta^4+\delta^{-2}h^4)\leq C(\delta^4+\delta^{-3}h^4),
    \end{align}
    where $I_hu_0$ is the Lagrangian type interpolation of $u_0$ on $\mathcal{T}_\delta^h(D)$
    This implies
    \begin{align}
        \|u_{\delta}^{*,D}-u_{h,\delta}^{*,D}\|_{L^2(\Omega)}&\leq \|u_\delta^{*,D}-u_{au,\delta}^{*,D}\|_{L^2(\Omega)}+\|u_{au,\delta}^{*,D}-u_{h,\delta}^{*,D}\|_{L^2(\Omega)}
        \leq C\delta^2+C\delta^{-1.5}h^2.
    \end{align}
    The proof is completed.
\end{proof}

Together with Theorem $\ref{thm59regionD}$, we can get the following corollary:
\begin{corollary}[Asymptotic Compatibility for the inner Neumann problem without the approximate ball]\label{corollary43}
    Assuming that $u_0$ is the solution of the local problem and $u_{h,\delta}^*$ is the finite element solution of outer Neumann problem in (\ref{nonlocalcpouterfem}), if the conditions of Lemma \ref{Luerrorsym} hold, it has:
     \begin{equation}
       \|u_0-u_{h,\delta}^{*, D}\|_{L^2(\Omega)} = \mathcal{O}(\delta^2) +\mathcal{O}(\delta^{-1.5}h^2).
    \end{equation}
\end{corollary}
We next discuss the error estimate of finite element method caused by the approximate ball with the Neumann boundary condition $N^*$. We first define the outer Neumann problem over domain $D$ by 
\begin{equation}\label{problemsymN*D}
    \left\{
    \begin{aligned}
    &\L_{\delta,sym} u= f_{out,sym}^p, \quad in~ \Omega,\\
    &\N^{*,r}_{D,sym} u=\N^{*,r}_{D,sym} \widetilde{u_0}, \quad in ~ D/\Omega,\quad\int_{\mathbf{x}\in\Omega} u d\mathbf{x}=0,
\end{aligned}
    \right.
\end{equation}
where the Neumann boundary operator considering both parts of geometry errors is 
\begin{equation}
    \N^{*,r}_{D,sym} u(\mathbf{x})=\int_{D\cap B_\delta(\mathbf{x})}(u(\mathbf{x})-u(\mathbf{y}))\gamma_{\delta,sym}(\mathbf{x},\mathbf{y})d\mathbf{y} \quad (\mathbf{x}\in D/\Omega),
\end{equation}
and the right-hand side term $f_{out,sym}^p$ is defined as 
\begin{equation}
    f_{out,sym}^{p}:=f+\frac{1}{S(\Omega)}\left(\int_{\partial\Omega} gd\mathbf{x}-\int_{D/\Omega} \N_{D,sym}^{*,r} \widetilde{u_0}d\mathbf{x}\right):=f+df_{sym}^{out,p}.
\end{equation}
By Lemma \ref{Luerrorsym} and using the similar proofs for Theorems \ref{thm56approxball} and  \ref{thm59regionD}, we have the following result. 
\begin{theorem}[Asymptotic compatibility for the outer Neumann problem with the approximate ball]
    Suppose conditions of Lemma \ref{Luerrorsym} hold, the solution of equation (\ref{problemsymN*D}) is $u_{\delta,sym}^{*,D}$, then it holds
    \begin{equation}\label{eqn599new}
    \|u_{\delta,sym}^{*,D}-u_0\|_{L^2(\Omega)}=\mathcal{O}(\delta^{-1}n_\delta^{-\lambda})+\mathcal{O}(\delta^2).
    \end{equation}
\end{theorem}

   The variational form of problem (\ref{problemsymN*D}) is as follows:
\begin{equation}\label{nonlocalcpouterballvar}
    \begin{aligned}
        &\int_\Omega v(\mathbf{x})\int_\Omega (u_{\delta,sym}^{*,D}(\mathbf{x})-u_{\delta,sym}^{*,D}(\mathbf{y}))\gamma_{\delta,sym}(\mathbf{x},\mathbf{y})d\mathbf{y}d\mathbf{x}\\&\quad=\int_\Omega v(\mathbf{x})f_{out,sym}^p(\mathbf{x})d\mathbf{x}+\int_{D/\Omega} v(\mathbf{x})\int_{D} (\widetilde{u_0}(\mathbf{x})-\widetilde{u_0}(\mathbf{y}))\gamma_{\delta,sym}(\mathbf{x},\mathbf{y})d\mathbf{y}d\mathbf{x}, 
    \end{aligned}
\end{equation} 
where the mean value of $u_{\delta,sym}^{*,D}$ on $\Omega$ is zero and $v\in S_{\delta,sym}^\prime(\Omega)$.
And the finite element scheme (\ref{nonlocalcpouterballvar}) changes to the following form: Find $u^{*,D}_{h,b}\in V_h^\delta(D)$ such that $\forall v_h\in V_h^\delta(D)$, 
\begin{equation}\label{nonlocalcpouterfemball}
    \begin{aligned}
        &\frac{1}{2}A_{D,D}^{sym}(u_{h,b}^{*,D},v_h)=\int_{\Omega}v_h(\mathbf{x})(f_h(\mathbf{x})+df_{h,sym}^{out,p})d\mathbf{x}-\int_{D/\Omega}v_h(\mathbf{x})\N^{*,r}_{D,sym}u_{h,b}^{*,D}d\mathbf{x},
    \end{aligned}
\end{equation}
where the mean value of $u_{h,b}^{*,D}$ on $\Omega$ is zero and 
\begin{equation}
    df_{h,sym}^{out,p}=\frac{1}{S(\Omega)}\Big(\int_{\partial\Omega} gd\mathbf{x}-\int_{D} \N_{D,sym}^{*,r} \widetilde{u_{0,h}}d\mathbf{x}\Big).
\end{equation}
When taking the approximate ball into account, similar to Lemma \ref{middlelemma}, it has the following corollary:
\begin{corollary}\label{lastproofremark}
    Suppose that the conditions of Lemma \ref{middlelemma} hold and the approximate ball is generated by the methods shown in \cite{ChenMaZhang}, if $\delta\geq 2h$, then, it holds 
    \begin{equation}
        \|u\|^2_{L^2(D)}\leq C(\delta^2A^{sym}_{D/\Omega,D}(u,u)+\|u\|^2_{L^2(\Omega_{\gamma,1/4}^-)}),
    \end{equation}
    where $\Omega^-_{\gamma,1/4}:=\{\mathbf{x}\in\Omega:dist(\mathbf{x},\partial\Omega)<\frac{\delta}{4}\}$.
\end{corollary}
\begin{proof}
    As $\mathbf{y}\in \Omega_{\gamma,1/4}^-$, $\mathbf{x}\in\Omega_{\gamma,1/4}^+$ when $\delta\geq h$, then $r_{\delta,\mathbf{y}}>\delta /2$, 
    which implies that:
    \begin{equation}
        \gamma_\delta(\mathbf{x},\mathbf{y})=\gamma_{\delta,pol}(\mathbf{x},\mathbf{y})\leq2\gamma_{\delta,sym}(\mathbf{x},\mathbf{y}).\nonumber
    \end{equation}
  It follows from (\ref{eqn543new}) that
    \begin{align}
       \int_{\Omega^+_{\gamma,1/4}}u(\mathbf{x})^2d\mathbf{x}&\leq 4\int_{\Omega^+_{\gamma,1/4}}\int_{\Omega^-_{\gamma,1/4}\cap B_{\delta/2}(\mathbf{x})}(u(\mathbf{x})-u(\mathbf{y}))^2\gamma_{\delta,sym}(\mathbf{x},\mathbf{y})d\mathbf{y}\frac{1}{M_1S(\Omega^-_{\gamma,1/4}\cap B_{\delta/2}(\mathbf{x}))}d\mathbf{x}\nonumber\\&\quad+ 2\int_{\Omega^+_{\gamma,1/4}}\int_{\Omega^-_{\gamma,1/4}\cap B_{\delta/2}(\mathbf{x})}u(\mathbf{y})^2d\mathbf{y}\frac{1}{S(\Omega^-_{\gamma,1/4}\cap B_{\delta/2}(\mathbf{x}))}d\mathbf{x}.
   \end{align}
 Thus, similar to the proof for Lemma \ref{middlelemma}, we have 
   \begin{align}\label{eqn551new}
       \int_{\Omega^+_{\gamma,1/4}}u(\mathbf{x})^2d\mathbf{x}&\leq C\Big(\delta^2 A^{sym}_{D,D/\Omega}(u,u)+\int_{\Omega^-_{\gamma,1/4}}u(\mathbf{x})^2d\mathbf{x}\Big).
   \end{align}
   Based on the proof presented in this corollary, the core of our proof lies in the assertion that for \( \mathbf{y} \in \hat{\Omega}_{\gamma,i/4}^+ \) and \( \mathbf{x} \in \hat{\Omega}_{\gamma,(i+1)/4}^+ \cap D\), where \( i = 1,2,3,4 \), the inequality \( \gamma_\delta(\mathbf{x},\mathbf{y}) \leq 2\gamma_{\delta,\mathrm{sym}}(\mathbf{x},\mathbf{y}) \) holds. It follows from \cite{ChenMaZhang} that, when $\delta\geq 2h$, the approximate ball strategies imply the following relationship:
   \begin{equation}
       B_{\delta/2}(\mathbf{x})\cap D\subset B_\delta^{pol}(\mathbf{x})\cap D\subset B_{\delta}(\mathbf{x})\cap D,~(\mathbf{x}\in D),\nonumber
   \end{equation}
   which means that $\forall \mathbf{x}\in \hat{\Omega}_{\gamma,(i+1)/4}\cap D$, if $\mathbf{y}\in \hat{\Omega}_{\gamma,i/4}\cap B_{\delta/2}(\mathbf{x})$, then, $\gamma_\delta(\mathbf{x},\mathbf{y})=\gamma_{\delta,pol}(\mathbf{x},\mathbf{y})\leq2\gamma_{\delta,sym}(\mathbf{x},\mathbf{y})$. 
Using the same argument as the proof of Lemma \ref{middlelemma} and Eq. (\ref{eqn551new}), we finish the proof.
   \end{proof}

\begin{theorem}[Error estimate of FEM for the outer Neumann problem with the approximate ball]\label{finalresultwithapproximateballN*}
    Assuming that the conditions of Lemma \ref{Luerrorsym} hold, as for the solution of (\ref{nonlocalcpouterballvar}) and (\ref{nonlocalcpouterfemball}), it holds
    \begin{equation}
       \|u_{\delta,sym}^{*,D}-u_{h,b}^{*,D}\|_{L^2(\Omega)}= \mathcal{O}(\delta^{-1}n_\delta^{-\lambda})+\mathcal{O}(\delta^2)+\mathcal{O}(\delta^{-1.5}h^2),
    \end{equation}
    where $u_{\delta,sym}^{*,D}$ is the solution of the nonlocal outer Neumann problem (\ref{problemsymN*D}) and $u_{h,b}^{*,D}$ is the finite element solution in (\ref{nonlocalcpouterfemball}).
\end{theorem}
\begin{proof}
Again, we introduce an auxiliary problem:
    $\forall v\in S^\prime_{\delta,sym}(D)$, finding $u_{ab,\delta}^{*,D} \in S^\prime_{\delta,sym}(D)$ such that:
    \begin{align}
        \frac{1}{2}&\int_{D}\int_{D} (v(\mathbf{x})-v(\mathbf{y}))(u_{ab,\delta}^{*,D}(\mathbf{x})-u_{ab,\delta}^{*,D}(\mathbf{y}))\gamma_{\delta,sym}(\mathbf{x},\mathbf{y})d\mathbf{y}d\mathbf{x}\nonumber\\&=\int_\Omega v(\mathbf{x})(f_h(\mathbf{x})+df_{h,sym}^{out,p})d\mathbf{x}-\int_{D/\Omega} v(\mathbf{x})\int_{D}(\widetilde{u_{0,h}}(\mathbf{x})-\widetilde{u_{0,h}}(\mathbf{y}))\gamma_{\delta,sym}(\mathbf{x},\mathbf{y})d\mathbf{y}d\mathbf{x},
    \end{align}
    where $u_{ab,\delta}^{*,D}$ also satisfies a mean value of 0 in $\Omega$.
    Because of Corollary \ref{lastproofremark}, Lemma \ref{middlelemma} is still holds. Totally the same with the proof of estimation (\ref{estimationneedalotstr}), it holds that
    \begin{align}\label{eqn580new}
        \|u_{\delta}^*-u_{ab,\delta}^*\|_{L^2(\Omega)}\leq C\delta^2+C\delta^{-1.5}h^2.
    \end{align}
    With the right-hand side of the nonlocal Poincar\'e inequality in Lemma \ref{ballpoincare}, (\ref{eqn580new}) and (\ref{eqn599new}), we have
    \begin{align}\label{errorseperateN*ball}
        \|u_{ab,\delta}^{*,D}&-u_{h,\delta}^{*,D}\|^2_{L^2(\Omega)}\leq (|u_{ab,\delta}^{*,D}-u_{h,b}^{*,D}|^{sym}_{\delta,D})^2\leq (|u_{ab,\delta}^{*,D}-I_hu_0|^{sym}_{\delta,D})^2\nonumber\\&
        \leq (|u_{ab,\delta}^{*,D}-u_{\delta,sym}^{*,D}|^{sym}_{\delta,D})^2 +(|u_0-u_{\delta,sym}^{*,D}|^{sym}_{\delta,D} )^2+(|u_0-I_hu_0|^{sym}_{\delta,D})^2\nonumber\\&
        \leq C(h^2+\delta^{-1.5}h^2)^2+C(\delta^2+\delta^{-1}n_\delta^{-\lambda})^2+C\delta^{-2}\|u_0-I_hu_0\|_{L^2(D)}^2\nonumber\\&\leq C((h^2+\delta^{-1.5}h^2)^2+(\delta^2+\delta^{-1}n_\delta^{-\lambda})^2+\delta^{-2}h^4)\leq C(\delta^{-2}n_\delta^{-2\lambda}+\delta^4+\delta^{-3}h^4),
    \end{align}
    where $I_hu_0$ is the Lagrangian type interpolation of $u_0$ on $\mathcal{T}_\delta^h(D)$. That implies
    \begin{align}
        \|u_{\delta,sym}^{*,D}-u_{h,b}^{*,D}\|_{L^2(\Omega)}&\leq \|u_{\delta,sym}^{*,D}-u_{ab,\delta}^{*,D}\|_{L^2(\Omega)}+\|u_{ab,\delta}^{*,D}-u_{h,\delta}^{*,D}\|_{L^2(\Omega)}\nonumber\\&
        \leq C(\delta^{-1}n_\delta^{-\lambda}+\delta^2+\delta^{-1.5}h^2).
    \end{align}
    
    The proof is completed.
\end{proof}

Together with Theorem $\ref{thm59regionD}$, we have the following corollary.
\begin{corollary}[Asymptotic compatibility for the outer Neumann problem considering the approximate ball] \label{finalresultwithapproximateballN*col}
Assume $u_0$ is the solution of the local problem and $u_{h,\delta}^{*,D}$ is the finite element solution of outer Neumann problem in (\ref{nonlocalcpouterfemball}), if the conditions of Lemma \ref{Luerrorsym} hold, it has:
    \begin{equation}
        \|u_0-u_{h,b}^{*,D}\|_{L^2(\Omega)}= \mathcal{O}(\delta^{-1}n_\delta^{-\lambda})+\mathcal{O}(\delta^2)+\mathcal{O}(\delta^{-1.5}h^2).
    \end{equation}
If taking  $n_\delta=\mathcal{O}(\delta/h)$, it holds
    \begin{equation}
        \|u_0-u_{h,b}^{*,D}\|_{L^2(\Omega)}= \mathcal{O}(\delta^{-1-\lambda}h^\lambda)+\mathcal{O}(\delta^2)+\mathcal{O}(\delta^{-1.5}h^2).
    \end{equation}
\end{corollary}

\section{Numerical experiment}
We now verify our theoretical results by providing some numerical examples. It follows from Eqs. (\ref{errorseperateball}) and (\ref{errorseperateN*ball}) that the error of the solution 
is mainly composed of three parts: (i) the model error between (\ref{problemNbigD}), (\ref{nonlocalcpinner}) and (\ref{local}), (ii) the error brought by the approximate ball and; (iii) the error arosn from the finite element method.  The first two numerical examples are  to investigate AC convergence order of the \emph{Nocaps} approximate ball strategy for the inner Neumann problem (\ref{nonlocalcpinner}) and the outer Neumann problem (\ref{nonlocalcp}) by choosing $u_0$ as a polynomial. In Section \ref{sectiont61}, other two examples are provided to investigate the finite element error and model error by using  the high-accuracy \emph{Approxcaps} strategy for the inner Neumann problem and the outer Neumann problem, respectively. 


\subsection{Numerical validation of the errors considering the approximate ball }\label{sec61}
Take $\Omega=(0,1)\times(0,1)$, the local solution  $u_0(x,y)=1.5x^2y-2xy+0.25$, and the kernel  $\gamma_\delta$ as 
\begin{equation}\label{constantkernel}
    \gamma_\delta(\mathbf{x},\mathbf{y})=\left\{
    \begin{aligned}
        &\frac{4}{\pi\delta^4}, \quad\quad &&|\mathbf{x}-\mathbf{y}|\leq\delta,\\
        &0, &&|\mathbf{x}-\mathbf{y}|>\delta.
    \end{aligned}
    \right.
\end{equation}
The auxiliary solutions $\widetilde{u}_0(x,y)=1.5x^2y-2xy+0.25$, 
 the Nocaps strategy is used to approximate the $\delta$-neighborhood. In order to maintain a more stable grid density, we use Delaunay triangulation method to generate a consistent mesh by Gmsh.
 
 
 
\begin{example}\label{exp64}
Here we  investigate the $L^2$-asymptotic compatible convergence order of the Nocaps  finite element strategy \eqref{nonlocalcpinnerfemball} for the nonlocal inner Neumann problem \eqref{nonlocalcpinner}.
To this end, the $\delta$-convergence is tested by fixing $m=\delta/h$ as a constant and refining $\delta$. In this situation with the parameter $\lambda=2$, Corollary \ref{finalresultwithapproximateballcol} shows the asymptotic compatible error estimate as 
    \begin{equation}
        \|u_0-u_{h,\delta}^-\|_{L^2(\Omega)}= \mathcal{O}(\delta^{-3}h^2)= \mathcal{O}(\delta^{-1}m^{-2}).
    \end{equation}
Table \ref{Tabel7} shows the errors in $L^2$-norm and the convergence orders for different $m$. On the one hand, when $h$ is fixed, the error decreases in a second-order rate as $m$ increases. On the other hand, when we fix $m$, the error has no obvious change as $h$ is refined. These phenomena indicate that the numerical convergence order seems to be 
 \begin{equation}\label{rightorderball}
     \|u_\delta^--u_{h,b}^-\|_{L^2(\Omega)}= O(m^{-2})+\mathcal{O}(\delta^2)+O(h^2)=\mathcal{O}(\delta^{-2}h^2)+\mathcal{O}(\delta^2),
 \end{equation}
 which is aligned with the results in \cite{du2024errorestimatesfiniteelement}. 

We further take $h=C\delta^\beta$ and change $\delta$. For $0<\beta<1$, we set the largest horizon $\delta_0$ and set $h_0=\delta/6$,  and compute the test pairs $\{\delta_i\}_{i=0}^4$ and $\{h_i\}_{i=0}^4$ by 
    \begin{equation}\label{testpairchoose}
        (\frac{\delta_0}{\delta_i})^\beta=\frac{(6-i)\delta_0}{6\delta_i}.
    \end{equation}
For $\beta>1$, we take  $h_{i}=1/(25(i+1))$ and set $\delta_0=0.08$. The relation $h=C\delta^\beta$ is then used to calculate the corresponding $\{\delta_i\}$. The Corollary 
\ref{finalresultwithapproximateballcol} shows that the theoretical convergence order is of  $2\beta-3$. However, Table \ref{Tabel8} indicates the convergence order in $L^2$-norm is close to be $2\beta-2$ for different $\delta$. Such results are aligned with Table 5 and Figure 9(b) in \cite{du2024errorestimatesfiniteelement} for nonlocal Dirichlet problem. This implies that the estimation in Corollary \ref{finalresultwithapproximateballcol} 
    might be overestimated and the numerical results shows that the estimate (\ref{rightorderball}) is a sharper order.
    

 We point out that the convergence order presented in Table \ref{Tabel8} oscillates near $2\beta-2$. This oscillation may arise from variations in the ratio between $n_\delta$ and $d/h$ due to mesh partitioning. The reason is that, noting the huge computational costs for nonlocal model, we cannot reduce $h$ and $\delta$ sufficiently small to achieve higher precision in capturing the convergence order. Nevertheless, the overall experimental results substantially align with the convergence order of $2\beta-2$. For the remainder results in the following examples, a similar phenomenon will be observed too.
\end{example}
\begin{table}[htbp]
\centering
\caption{Errors in $L^2$ norm and corresponding convergence rates for problem (\ref{problemsymN-}) with Nocaps strategy and $\delta=mh$}
\begin{tabular}{c|cc|cc|cc|cc|cc}
\toprule
 & \multicolumn{2}{c|}{$h$ =1/ 32} & \multicolumn{2}{c|}{$h$ =1/ 64} & 
 \multicolumn{2}{c|}{$h$ =1/ 96} & 
 \multicolumn{2}{c|}{$h$ =1/ 128} & 
 \multicolumn{2}{c}{$h$ =1/ 160} \\
\midrule
$m$ & Error & Order & Error & Order & Error & Order & Error & Order& Error & Order \\
\midrule
3 & 6.774e-4 & -- &6.918e-4&--&	7.101e-4&--&	7.173e-4 &--& 6.989e-4 &--\\
4 & 3.756e-4& -2.050	&4.057e-4& -1.855	&4.239e-4&-1.793& 4.083e-4&-1.959 & 4.139e-4 &-1.821\\ 
        
5 &2.820e-4& -1.285	&2.819e-4	&	-1.631&2.750e-4&	-1.940 &2.757e-4&	-1.761 &  2.776e-4&	-1.790  \\ 
6 &1.682e-4& -2.833	&1.867e-4&	-2.260	& 1.874e-4&-2.102 & 1.902e-4&	-2.037& 1.814e-4 &	-2.335  \\ 
\bottomrule
\end{tabular}
\label{Tabel7}
\end{table}
\begin{table}[htbp]
\centering
\caption{Errors in $L^2$ norm and corresponding convergence rates for problem (\ref{problemsymN-}) with Nocaps strategy and  $h=C\delta^\beta$}
\resizebox{\textwidth}{!}{
\begin{tabular}{c|cc|c|cc|c|cc|c|cc}
\toprule
 & \multicolumn{2}{c|}{$\beta$=0.3} & &\multicolumn{2}{c|}{$\beta$=0.5} & &
 \multicolumn{2}{c|}{$\beta$=1.4} & &
 \multicolumn{2}{c}{$\beta$=1.7}  \\
\midrule
$\delta$ & Error & Order& $\delta$& Error & Order&$\delta$ & Error & Order&$\delta$ & Error & Order \\
\midrule
0.125 &1.685e-4 & -- &0.125&1.685e-4&--&0.08&	1.580e-3&--&0.08	&1.580e-3 &-- \\
        
0.096 &2.918e-4& -2.108&0.087	&2.640e-4	&-1.232&0.049&9.240e-4&	1.083&0.053&8.640e-4&	1.480   \\ 
0.070 &3.410e-4& -0.489	&0.056&4.075e-4&	-0.973&0.037	& 8.321e-4&0.362 &0.042 &6.124e-4&	1.443  \\ 
0.046 &6.877e-4& -1.707	&0.031&7.101e-4&	-0.965&0.030	& 7.004e-4& 0.838&0.035 & 5.159e-4&	1.014  \\ 
0.026 &1.415e-3& -1.246&0.014	&1.397e-3&	-0.835&0.025	& 5.660e-4& 1.337& 0.031& 4.102e-4&	1.747  \\ 
\bottomrule
\end{tabular}
}
\label{Tabel8}
\end{table}

\begin{example}\label{exp65}

We use the same setup and conditions as in Example \ref{exp64} to investigate the $L^2$-asymptotically  compatible convergence order of the Nocaps finite element strategy \eqref{nonlocalcpouterfemball} for the nonlocal outer Neumann problem \eqref{problemNbigD}. The outer region $D$ in (\ref{problemNbigD}) is taken as $(-\delta,1+\delta)\times(-\delta,1+\delta)$. The $\delta$-convergence is  tested by fixing $m=\delta/h$ and refining $\delta$. Thus, from Corollary \ref{finalresultwithapproximateballN*col}, the error estimate turns out to be
    \begin{equation}
        \|u_0-u_{h,\delta}^*\|_{L^2(\Omega)}= \mathcal{O}(\delta^{-3}h^{2})=\mathcal{O}(\delta^{-1}m^{-2}).
    \end{equation}
   Table \ref{Tabel10} provides the errors in $L^2$-norm and the convergence orders for different $m$. For fixed $h$, when $m$ increases, the error decreases in a second-order rate. In the meantime, as $m$ is fixed, the error has no obvious change. These two phenomena are similar to phenomena in Table \ref{Tabel7}. It indicates that Corollary 
   \ref{finalresultwithapproximateballN*col} may be overestimated, and the numerical convergence order seems to be:
   \begin{equation}\label{example52newres}
   	\|u_0-u_{h,\delta}^*\|_{L^2(\Omega)}= \mathcal{O}(\delta^{-2}h^2)+\mathcal{O}(\delta^2).
   \end{equation} 
    
    We further take $h=C\delta^\beta$ and change $\delta$. The selection of specific $\delta-h$ pairs is consistent with Example \ref{exp64}. The Corollary 
    \ref{finalresultwithapproximateballN*col} shows that the theoretical convergence rate is of $2\beta-3$. Table \ref{Tabel9} indicates the convergence order in $L^2$-norm  is close to be  $2\beta-2$ for different $\delta$, which is similar with the phenomenon in Table \ref{Tabel8}. This further implies that the estimation in Corollary 
    \ref{finalresultwithapproximateballN*col}  may be overestimated, and (\ref{example52newres}) might be right.
\end{example}

\begin{table}[htbp]
\centering
\caption{Errors in $L^2$ norm and corresponding convergence rates for problem (\ref{problemsymN*D}) by Nocaps strategy with fixing $\delta=mh$}
\begin{tabular}{c|cc|cc|cc|cc|cc}
\toprule
 & \multicolumn{2}{c|}{$h$ =1/ 32} & \multicolumn{2}{c|}{$h$ =1/ 64} & 
 \multicolumn{2}{c|}{$h$ =1/ 96} & 
 \multicolumn{2}{c|}{$h$ =1/ 128} & 
 \multicolumn{2}{c}{$h$ =1/ 160} \\
\midrule
m & Error & Order & Error & Order & Error & Order & Error & Order& Error & Order \\
\midrule
3 &5.598e-4 & -- &	6.332e-4&--&	6.674e-4&--&	6.845e-4 &--& 6.740e-4 &--\\
4 & 2.997e-4& -2.172	&3.614e-4& -1.950	&3.923e-4&-1.847& 3.845e-4&-2.005 & 3.941e-4 &-1.866\\ 
        
5 &2.050e-4& -1.702	&2.426e-4	&	-1.786&2.487e-4&	-2.042 &2.554e-4&	-1.834 &  2.611e-4&	-1.845  \\ 
6 &1.237e-4& -2.771	&1.565e-4&	-2.405	& 1.661e-4&-2.215 & 1.740e-4&	-2.105& 1.687e-4 &	-2.394  \\ 
\bottomrule
\end{tabular}
\label{Tabel10}
\end{table}

\begin{table}[htbp]
\centering
\caption{Errors in $L^2$ norm and corresponding convergence rates for finite element solutions of problem (\ref{problemsymN-}) by Nocaps strategy with taking $h=C\delta^\beta$}
\resizebox{\textwidth}{!}{
\begin{tabular}{c|cc|c|cc|c|cc|c|cc}
\toprule
 & \multicolumn{2}{c|}{$\beta$=0.3} & &\multicolumn{2}{c|}{$\beta$=0.5} & &
 \multicolumn{2}{c|}{$\beta$=1.4} & &
 \multicolumn{2}{c}{$\beta$=1.7}  \\
\midrule
$\delta$ & Error & Order& $\delta$& Error & Order&$\delta$ & Error & Order&$\delta$ & Error & Order \\
\midrule
0.125 &1.343e-4 & -- &0.125&1.343e-4&--&0.08&	1.361e-3&--&0.08	&1.361e-3 &-- \\
        
0.096 &2.401e-4& -2.230&0.087	&2.231e-4	&-1.392&0.049&8.410e-4&	0.972&0.053&7.750e-4&	1.381   \\ 
0.070 &3.011e-4& -0.709	&0.056&3.682e-4&	-1.123&0.037	& 7.730e-4&0.291 &0.042 &5.598e-4&	1.364  \\ 
0.046 &6.288e-4& -1.792	&0.031&6.674e-4&	-1.033&0.030	& 6.585e-4& 0.780&0.035 & 4.809e-4&	0.898  \\ 
0.026 &1.360e-3& -1.332&0.014	&1.366e-3&	-0.883&0.025	& 5.352e-4& 1.301& 0.031& 3.875e-4&	1.645  \\ 
\bottomrule
\end{tabular}
}
\label{Tabel9}
\end{table}

\subsection{Numerical validation of model errors and finite element errors}\label{sectiont61}
Take $\Omega=(0,1)\times(0,1)$, the local solution $u_0(x,y)=x^2y+y^2-0.5$ and the kernel $\gamma_\delta$ as a constant kernel defined in (\ref{constantkernel}). The auxiliary solution $\widetilde{u_0}(x,y)=x^2y+y^2-0.5$.  $(x,y)\in \Omega^+$, the Approxcaps strategy in \cite{Cookbook} is used to approximate the $\delta$-neighborhood. Although the Approxcaps strategy still satisfies the estimation in Corollary \ref{finalresultwithapproximateballcol} and \ref{finalresultwithapproximateballN*col}, we can refine the caps  (here in our test, we refined $O_r>-\log_2h$ times) to reduce the effect of the approximate ball. This strategy is easier to implement programmatically and exhibits high numerical stability.

\begin{example}\label{exp61}
We now investigate the $L^2$ asymptotically compatible convergence order of the finite element strategy (\ref{nonlocalcpinnerfem}) for the nonlocal inner Neumann problem (\ref{nonlocalcpinner}). To this end, the $\delta$-convergence is tested by fixing $m=\delta/h$ as a constant and refine $h$. In this situation, when the approximate ball error is sufficiently small, Corollary \ref{corollary41} shows the asymptotic compatible error estimate as:
    \begin{equation}
        \|u_0-u_{h,\delta}^-\|_{L^2(\Omega)}= \mathcal{O}(\delta^{-1.5}h^2).
    \end{equation}
     Table \ref{Tabel1} shows the errors in $L^2$-norm and the convergence orders for different $h$. On the one hand, for fixed $m$, when $h$ becomes smaller, the asymptotic compatible error converges at a second-order rate. On the other hand, for fixed $h$, the errors have no obvious change as $m$ increases. Such two phenomena indicate that the result in Corollary \ref{corollary41} may be overestimated and the error estimation in $L^2$-norm might be
     \begin{equation}\label{wrongmodified result}
        \|u_0-u_{h,\delta}^-\|_{L^2(\Omega)}= O(h^2),
    \end{equation}
    when $u_0$ is a polynomial, which is aligned with Table $3$ in \cite{du2024errorestimatesfiniteelement}.
     
    Next, we take $h=C\delta^\beta$ and change $\delta$. Similar to the choice of $\delta-h$ pairs in Example \ref{exp64} and \ref{exp65}, when $0<\beta<1$, we choose the largest $\delta_0$ and the corresponding $h_0=\delta_0/6$. Then, we use (\ref{testpairchoose}) to calculate the next test pairs. 
    For $\beta>1$, as $\delta$ get smaller, the corresponding $h$ will get smaller much quicker which will cause a lot difficulties in computation. To reduce the difficulty of validation, after choosing $\delta_0$ and $h_0=\delta/2$, we take $h_i=1/(1/h_0+10i)$  and compute the corresponding $\delta$ from the condition $h=C\delta^\beta(\beta>1)$.  The Corollary \ref{corollary41} shows the error in $L^2$-norm as:
    \begin{equation}
        \|u_0-u_{h,\delta}^-\|_{L^2(\Omega)}=\mathcal{O}(\delta^{-1.5}h^2)=\mathcal{O}(\delta^{2\beta-1.5}).
    \end{equation}
    In Table \ref{Tabel3}, the errors show a convergence order close to $2\beta$. In other words, this indicates once again that the result in Corollary \ref{corollary41} may be overestimated, and (\ref{wrongmodified result}) might be  the sharper order.
    
\end{example}

\begin{table}[htbp]
\centering
\caption{Errors in $L^2$ norm and corresponding convergence rates for finite element solutions of problem (\ref{nonlocalcpinner}) using approxcaps strategy (fixing $\delta=mh$)}
\begin{tabular}{c|cc|cc|cc|cc}
\toprule
 & \multicolumn{2}{c|}{m=2} & \multicolumn{2}{c|}{m=3} & 
 \multicolumn{2}{c|}{m=4} & 
 \multicolumn{2}{c}{m=5}  \\
\midrule
h & Error & Order & Error & Order & Error & Order & Error & Order \\
\midrule
1/24 &1.592e-4 & -- &	1.431e-4&--&	1.528e-4&--&	1.478e-4 &-- \\
1/36 & 7.083e-5& 1.997	&6.153e-5& 2.081	&7.116e-5&1.886& 6.884e-5&1.884 \\ 
        
1/48 &4.278e-5& 1.753	&3.599e-5	&	1.865&4.004e-5&	1.999 &3.903e-5&	1.973   \\ 
1/60 &2.676e-5& 2.103	&2.320e-5&	1.967	& 2.530e-5&2.057 & 2.401e-5&	2.177  \\ 
1/72 &1.854e-5& 2.011	&1.624e-5&	1.955	& 1.816e-5& 1.818 & 1.705e-5&	1.880  \\ 
\bottomrule
\end{tabular}
\label{Tabel1}
\end{table}

\begin{table}[htbp]
\centering
\caption{Errors in $L^2$ norm and corresponding convergence rates for finite element solutions of problem (\ref{nonlocalcpinner}) using approxcaps strategy (taking $h=C\delta^\beta$)}
\resizebox{\textwidth}{!}{
\begin{tabular}{c|cc|c|cc|c|cc|c|cc}
\toprule
 & \multicolumn{2}{c|}{$\beta$=0.5} & &\multicolumn{2}{c|}{$\beta$=0.7} & &
 \multicolumn{2}{c|}{$\beta$=1.3} & &
 \multicolumn{2}{c}{$\beta$=1.8}  \\
\midrule
$\delta$ & Error & Order& $\delta$& Error & Order&$\delta$ & Error & Order&$\delta$ & Error & Order \\
\midrule
0.25 &1.493e-4 & -- &0.4&3.872e-4&--&0.4&	4.029e-3&--&0.4	&4.029e-3 &-- \\
        
0.174 &1.088e-4& 0.866&0.2	&1.357e-4	&1.513&0.172&4.113e-4&	2.700&0.217&3.746e-4&	3.893   \\ 
0.111 &7.116e-5& 0.952	&0.1&5.382e-5&	1.334&0.116	& 1.508e-4&2.554 &0.164 &1.327e-4&	3.656  \\ 
0.046 &3.599e-5& 1.185	&0.042&1.624e-5&	1.368&0.090	& 6.652e-5& 3.157&0.136 & 6.987e-5&	3.430  \\ 
0.026 &1.854e-5& 0.818&0.010	&2.922e-6&	1.237&0.074	& 4.518e-5& 2.006& 0.118& 4.653e-5&	2.909  \\ 
\bottomrule
\end{tabular}
}
\label{Tabel3}
\end{table}

\begin{example}\label{exp62}

Here we investigate the $L^2$ asymptotic compatible convergence order of the finite element method (\ref{nonlocalcpouterfem}) for the nonlocal outer Neumann problem (\ref{nonlocalcp}), with the same experiment setup and conditions as in Example \ref{exp61}. The region $D$ in (\ref{problemsymN*D}) is chosen as $(-\delta,1+\delta)\times(-\delta,1+\delta)$. From Corollary \ref{corollary43}, if we choose $u_0(x,y)=x^2y+y^2-0.5$ and the approximate ball error is sufficiently small, the asymptotic compatible error is expected that:
    \begin{equation}
        \|u_0-u_{h,\delta}^*\|_{L^2(\Omega)}= O(\delta^{-1.5}h^2).
    \end{equation}
Table \ref{Tabel2} provides the errors in $L^2$-norm and the convergence orders for different $h$. Similar to the phenomenon in Table \ref{Tabel1}, when $h$ becomes smaller, the errors converge at a second-order rate for fixed $m$. For fixed $h$, the errors have no obvious change as $m$ increases.
This two phenomena indicates that Corollary \ref{corollary43} may be overestimated. When $u_0$ is a polynomial, the error estimation in $L^2$-norm might improve as:
     \begin{equation}\label{lastexamplefinal}
        \|u_0-u_{h,\delta}^*\|_{L^2(\Omega)}= O(h^2).
    \end{equation} 
 We further take $h=C\delta^\beta$ and change $\delta$. The $\delta-h$ pairs are selected aligned  with Example \ref{exp61}. By Corollary \ref{corollary43}, the theoretical convergence order is of $2\beta-1.5$. In Table \ref{Tabel4} the convergence orders appear to be $2\beta$ which is similar to Table \ref{Tabel3}. This further indicates Corollary \ref{corollary43} may be overestimated and (\ref{lastexamplefinal})might be right.
\end{example}

\begin{table}[htbp]
\centering
\caption{Errors in $L^2$ norm and corresponding convergence rates for finite element solutions of problem (\ref{nonlocalcpinner}) using approxcaps strategy (fixing $\delta=mh$)}
\begin{tabular}{c|cc|cc|cc|cc}
\toprule
 & \multicolumn{2}{c|}{m=2} & \multicolumn{2}{c|}{m=3} & 
 \multicolumn{2}{c|}{m=4} & 
 \multicolumn{2}{c}{m=5}  \\
\midrule
h & Error & Order & Error & Order & Error & Order & Error & Order \\
\midrule
1/24 &1.516e-4 & -- &	1.441e-4&--&	1.501e-4&--&	1.454e-4 &-- \\
1/36 & 6.524e-5& 2.079	&6.057e-5& 2.138	&6.254e-5&2.159& 6.764e-5&1.888 \\ 
        
1/48 &3.832e-5& 1.850	&3.689e-5	&	1.724&3.808e-5&	1.724 &3.846e-5&	1.963   \\ 
1/60 &2.396e-5& 2.105	&2.211e-5&	2.294	& 2.338e-5&2.187 & 2.396e-5&	2.121  \\ 
1/72 &1.655e-5& 2.028	&1.554e-5&	1.935	& 1.637e-5& 1.955 & 1.643e-5&	2.068  \\ 
\bottomrule
\end{tabular}
\label{Tabel2}
\end{table}

\begin{table}[htbp]
\centering
\caption{Errors in $L^2$ norm and corresponding convergence rates for finite element solutions of problem (\ref{nonlocalcp}) using approxcaps strategy (taking $h=C\delta^\beta$)}
\resizebox{\textwidth}{!}{
\begin{tabular}{c|cc|c|cc|c|cc|c|cc}
\toprule
 & \multicolumn{2}{c|}{$\beta$=0.5} & &\multicolumn{2}{c|}{$\beta$=0.7} & &
 \multicolumn{2}{c|}{$\beta$=1.3} & &
 \multicolumn{2}{c}{$\beta$=1.8}  \\
\midrule
$\delta$ & Error & Order& $\delta$& Error & Order&$\delta$ & Error & Order&$\delta$ & Error & Order \\
\midrule
0.25 &1.480e-4 & -- &0.4&3.791e-4&--&0.4&	3.665e-3&--&0.4	&3.665e-3 &-- \\
        
0.174 &1.056e-4& 0.925&0.2	&1.342e-4	&1.498&0.172&3.889e-4&	2.654&0.217&3.375e-4&	3.909   \\ 
0.111 &6.254e-5& 1.174	&0.1&5.177e-5&	1.374&0.116	& 1.290e-4&2.810 &0.164 &1.265e-4&	3.459  \\ 
0.046 &3.689e-5& 0.917	&0.042&1.554e-5&	1.375&0.090	& 6.394e-5& 2.709&0.136 & 6.677e-5&	3.415  \\ 
0.026 &1.655e-5& 0.988&0.010	&2.377e-6&	1.354&0.074	& 4.168e-5& 2.214& 0.118& 4.035-5&	3.604  \\ 
\bottomrule
\end{tabular}
}
\label{Tabel4}
\end{table}

\section{Conclusion}
A systematic analysis of nonlocal Neumann problems governed by nonlocal Poisson equations is presented. First, a family of nonlocal Neumann operators are derived directly from nonlocal Green’s identities. By means of Taylor expansions and a localized geometric treatment of boundary layers, we establish first-order weak convergence of these operators to the classical normal derivative under weaker regularity requirements than previously assumed. Subsequently, we formulate interior and exterior nonlocal Neumann problems, prove their second-order asymptotic compatibility with the corresponding local limit, and thereby quantify the modeling error precisely. Second, the AC convergence analysis is established for the approximate ball-based  finite element method and a complete error estimation is provided that separately accounts for (i) model error, (ii) finite element approximation, and (iii) geometric quadrature induced by the approximate ball strategy. Domain perturbations caused by inexact representation of the outer boundary are additionally incorporated, yielding the first fully discrete asymptotic compatibility result for nonlocal Neumann problems. Numerical experiments corroborate the theoretical convergence rates and illustrate the practical relevance of the analysis.

Future work will address time-dependent and multi-physics extensions, together with the design of fast, robust solvers that reduce the computational burden of nonlocal boundary integrals and enhance the overall efficiency of nonlocal simulations.

\section{Appendix}

\subsection{Lemmas for the existence and uniqueness of intersection point $p$}\label{append2}
We firstly introduce the following variable transformation $V_S:(\mathbf{p},\theta,s,t)\rightarrow (\mathbf{r_1},\mathbf{r_2})$ as 
\begin{equation}\label{variable21}
\begin{aligned}
\mathbf{r_1} = 
\begin{pmatrix}
    x_1\\y_1
\end{pmatrix} =
\mathbf{p} - t
\begin{pmatrix}
    \cos(\theta)&-\sin(\theta)\\\sin(\theta)&\cos(\theta)
\end{pmatrix}
\begin{pmatrix}
    n_{1p}\\n_{2p}
\end{pmatrix}; \;
\mathbf{r_2} = 
\begin{pmatrix}
    x_2\\y_2
\end{pmatrix} =
\mathbf{p} +s
\begin{pmatrix}
    \cos(\theta)&-\sin(\theta)\\\sin(\theta)&\cos(\theta)
\end{pmatrix}
\begin{pmatrix}
    n_{1p}\\n_{2p}
\end{pmatrix}.
\end{aligned}
\end{equation}

 First, we take a segment \( \Gamma \) of length \( l<\infty \) along the \( y \)-axis. Within this segment, we assume that \( \Gamma \) is parameterized by arc length as \( (x(\eta), y(\eta)) \) where \( \eta \in [0, l] \). Consequently, the intersection point \( \mathbf{p} \in \Gamma \) defined in Eq. (\ref{variable2}) can be represented by the parameter \( \eta \) as 
\[
\mathbf{p} = \left(x(\eta), y(\eta)\right)^\top=(0,y(\eta))^\top.
\]  
Then, the outer normal vector at   $\mathbf{p}$ is that
\begin{equation}
    (n_{1p},n_{2p})=(-\frac{y^\prime(\eta)}{\sqrt{x^\prime(\eta)^2+y^\prime(\eta)^2}},\frac{x^\prime(\eta)}{\sqrt{x^\prime(\eta)^2+y^\prime(\eta)^2}})=(1,0).
\end{equation} 
The variable substitution (\ref{variable2}) then can be written as:
\begin{equation}\label{coordtrans}
\left\{
\begin{aligned}
    &x_1=x(\eta)-t(\cos(\theta)n_{1p}-\sin(\theta)n_{2p}),\quad y_1=y(\eta)-t(\sin(\theta)n_{1p}+\cos(\theta)n_{2p}),\\
    &x_2=x(\eta)+s(\cos(\theta)n_{1p}-\sin(\theta)n_{2p}), \quad y_2=y(\eta)+s(\cos(\theta)n_{1p}-\sin(\theta)n_{2p}).
\end{aligned}
\right.
\end{equation}
For simplicity, we use the following notation:
\begin{equation}
\begin{aligned}
        &N_1= \cos(\theta)\frac{\partial n_{1p}}{\partial \eta}-\sin(\theta)\frac{\partial n_{2p}}{\partial \eta},\quad
        N_2=\sin(\theta)n_{1p}+\cos(\theta)n_{2p},\\&
        N_3=\sin(\theta)\frac{\partial n_{1p}}{\partial \eta}+\cos(\theta)\frac{\partial n_{2p}}{\partial \eta},\quad
        N_4=\cos(\theta)n_{1p}-\sin(\theta)n_{2p}.
\end{aligned}
\end{equation}
Then, the Jacobian matrix is given by:
$$
J = \frac{\partial(x_1,y_1,x_2,y_2)}{\partial(\eta,\theta,s,t)}
= \left(\begin{array}{cccc}
x^\prime(\eta)-tN_1 & tN_2& 0 & -N_4 \\
y^\prime(\eta)-tN_3 & -tN_4 & 0 &-N_2\\
x^\prime(\eta)+sN_1 & -sN_2  & N_4& 0  \\
y^\prime(\eta)+sN_3 & sN_4  & N_2 & 0
\end{array}\right). 
$$
It's easy to find that the determinant of the Jacobian matrix:
\begin{align}\label{detJ}
\det(J)&
=(s+t)(x^\prime(\eta),y^\prime(\eta))\left(\begin{array}{cc}
\sin(\theta) & \cos(\theta)\\
-\cos(\theta)& \sin(\theta)
\end{array}\right)\left(\begin{array}{c} n_{1p}\\n_{2p}\end{array}\right)\nonumber\\&
=(s+t)\cos(\theta)\sqrt{(x^\prime(\eta))^2+(y^\prime(\eta))^2}(n_{1p}^2+n_{2p}^2)\nonumber\\&
=(s+t)\cos(\theta)\sqrt{(x^\prime(\eta))^2+(y^\prime(\eta))^2}.
\end{align}
\begin{lemma}\label{convex}
For a convex and bounded region \( \Omega \), if the line segment \( \overline{\mathbf{r_1}\mathbf{r_2}} \) has a finite number of intersection points with the boundary \( \partial\Omega \), then the intersection point \( \mathbf{p} \) under the mapping \( V_n \) is uniquely determined.
\end{lemma}

\begin{proof}
Without loss of generality, we work within the closure \( \overline{\Omega} \) of the region \( \Omega \). Suppose, for the sake of contradiction, that there are two distinct intersection points \( \mathbf{p_1} \) and \( \mathbf{p_2} \). Since \( \Omega \) is convex, all points on the line segments \( \overline{\mathbf{r_1}\mathbf{p_1}} \) and \( \overline{\mathbf{p_1}\mathbf{p_2}} \) are contained within \( \Omega \). Since \( \mathbf{p_1} \) lies on the boundary \( \partial\Omega \) , by the supporting hyperplane theorem of the convex region, the line segment \( \overline{\mathbf{p_1}\mathbf{p_2}} \) or \( \overline{\mathbf{r_1}\mathbf{p_1}} \) must also be part of \( \partial\Omega \). This conclusion contradicts our initial assumption that \( \overline{\mathbf{r_1}\mathbf{r_2}} \) has a finite number of intersection points with \( \partial\Omega \). Thus, the intersection point \( \mathbf{p} \) is unique.
\end{proof}

In the following, we will give the intersection property of a $C^2$ region with the neighborhood of boundary points.
\begin{lemma}\label{intersection}
We assume that $\Omega$ has a closed $C^2$ boundary $\partial \Omega=\Gamma(s)(0\leq s\leq \beta) $. 
Then, there exists $\delta_0>0$, for each $s_0~in~[0,\beta]$, $\forall\delta<\delta_0$, $B(\Gamma(s_0),\delta)$ has and only has two intersection points with $\Gamma(s)$.
\end{lemma}
\begin{proof}
Because that $\Gamma(s)$ is $C^2$ embedded in $\R^2$, for arbitrary $s_0\in[0,\beta]$, assuming that $\mathbf{x_0}=\Gamma(s_0)$, there exists an $r_0$ and a $C^2$ mapping $f(x):~\R\rightarrow \R$ such that upon relabeling and reorienting the coordinate axes if necessary, we have $\{B(\mathbf{x_0},r_0)\cap \Omega\}=\{\mathbf{x}=(x,y)\in B(\mathbf{x_0},r_0)|y>f(x) \}:=R_c(\mathbf{x_0}) $. Take $\mathbf{x_1}\in B(\mathbf{x_0},r_0)\cap \partial \Omega$, for arbitrary $\mathbf{x}\in B(\mathbf{x_0},r_0)\cap \partial \Omega$, (here we use $x/x_1$ to represent the x-coordinate of point $\mathbf{x}/\mathbf{x_1}$) using Taylor expansion, it has:
\begin{align*}
    f(x)= f(x_1)+f^{(1)}(\xi)(x-x_1),
\end{align*}
where $\xi$ is a point between $x$ and$x_1$. The intersection point between $\partial \Omega$ and $B(\mathbf{x_1},r)$ satisfies:
\begin{align*}
    (x-x_1)^2+f^{(1)}(\xi)^2(x-x_1)^2=r^2.
\end{align*}
Since $f\in C^2(\R)$, denote that  
there exists $C_{f1}>0$ and $C_{f2}>0$, $|f^{(1)}(x)|\leq C_{f1}$ and $|f^{(2)}(x)|\leq C_{f2}$ when $(x,f(x))\in B(\mathbf{x_0},r_0)\cap\partial\Omega$. Set $x-x_1=ry$, the equation above can be written as:
\begin{align}\label{geometryequation}
    (1+f^{(1)}(\xi)^2)y^2=1.
\end{align}
From the boundedness of $|f^{(1)}(x)|$, the solution $y_0$ of equation (\ref{geometryequation}) must satisfies $y_0\in[\frac{1}{\sqrt{1+C_{f1}^2}},1]$
Taking the left hand side of (\ref{geometryequation}) as $F(y)$ and considering the first order derivation of $F$:
\begin{align*}
    \frac{d}{dy}F(y)=2(1+f^{(1)}(\xi)^2)y+(1+2t_{\xi}rf^{(2)}(\xi))y^2,
\end{align*}
where $t_{\xi}\in[0,1]$ satisfies $\xi=x_1+t_{\xi}ry$. Since the boundedness of $|f^{(2)}(x)|$, there exists $r_1$, as $y>0$, $\frac{d}{dy}F(y)>0$, which means that (\ref{geometryequation}) has only one solution bigger than 0, the negative part is the same. Take $r(\mathbf{x_0})=\min\{r_0(\mathbf{x_0}), r_1(\mathbf{x_0}))\}$, $B(\mathbf{x_0},r_c)(r_c\leq r(\mathbf{x_0}))$ has two intersection points with $\Gamma$. 

It can be seen that in the proof above, for arbitrary point $\mathbf{x_1}$ in $B(\mathbf{x_0},r(\mathbf{x_0}))\cap \partial\Omega$, if we choose $r\leq r(\mathbf{x_0})-|\mathbf{x_1}-\mathbf{x_0}|$, $B(\mathbf{x_1},r)\subset B(\mathbf{x_0},r(\mathbf{x_0}))$ and has two intersection points with $\Gamma$.
From the compactness of $\partial\Omega$, for each $\mathbf{x}$, there exists the above $r_0(\mathbf{x})~and~r_1(\mathbf{x})$ and $\bigcup\limits_{x}B(\mathbf{x},r(\mathbf{x}))$ is an open cover of $\Gamma$. Then, there exists $\mathbf{p_1}\cdots \mathbf{p_n}$, $\bigcup\limits_{i=1\cdots n}B(\mathbf{p_i},r(\mathbf{p_i}))$ is the finite open cover. We assume $s_{i1}$ and $s_{i2}$ are the two intersection point of $B(\mathbf{p_i},r(\mathbf{p_i}))$ and $\Gamma(s)$. Without loss of generality, the increasing arrangement is
\begin{align*}
    s_{11},~s_{n2},~s_{21},~s_{12},~s_{31},~s_{22}\cdots ,~s_{(n-1)1},~s_{(n-2)2},~s_{n1},~s_{(n-1)2}.
\end{align*}
Renumber them as $a_1\cdots a_{2n}$ and take $\delta_0=0.49\min\{a(i)-a(i-1)(i=2\cdots 2n),~a(1)+\beta-a(2n)\}$, which is desired.
\end{proof}

Nextly, we will propose another lemma which build up the bridge between $\N_\Omega^-$ and $\N_\Omega^*$:
\begin{theorem}\label{locallimitnew}
       Considering that either the conditions in Theorem \ref{mainresult} or in Theorem \ref{mainresult} is satisfied, it has the following result:
    \begin{equation}\label{C3resultnew2}
       (\N_{\Omega}^{*} u,v)_{\Omega_{\gamma}^+}=(\sigma\frac{\partial u}{\partial \mathbf{n}},v)_{\partial\Omega}+\mathcal{O}(\delta).
    \end{equation}
\end{theorem}
\begin{proof}
    We can first divide the integration region into two parts as follows:
    \begin{align}\label{eqndiv}
        (\N_{\Omega}^{*} u,v)_{\Omega_{\gamma}^+}&=\int_{\mathbf{x}\in\Omega_\gamma^+}\int_{\mathbf{y}\in\Omega^+}[u(\mathbf{x})-u(\mathbf{y})]\gamma_\delta(\mathbf{x},\mathbf{y})v(\mathbf{x})d\mathbf{y}d\mathbf{x}\nonumber\\&
        =\int_{\mathbf{x}\in\Omega_\gamma^+}\int_{\mathbf{y}\in\Omega_\gamma^+}[u(\mathbf{x})-u(\mathbf{y})]\gamma_\delta(\mathbf{x},\mathbf{y})v(\mathbf{x})d\mathbf{y}d\mathbf{x}\nonumber\\&\quad+\int_{\mathbf{x}\in\Omega_\gamma^+}\int_{\mathbf{y}\in\Omega_\gamma^-}[u(\mathbf{x})-u(\mathbf{y})]\gamma_\delta(\mathbf{x},\mathbf{y})v(\mathbf{x})d\mathbf{y}d\mathbf{x}.
    \end{align}
    Similar with the proof of Theorem \ref{oneonone}, Theorem \ref{mainresult} and Theorem \ref{rectangleright}, the last part in (\ref{eqndiv}) can be approximated as follows:
    \begin{equation}\label{pastre}
        \int_{\mathbf{x}\in\Omega_\gamma^+}\int_{\mathbf{y}\in\Omega_\gamma^-}[u(\mathbf{x})-u(\mathbf{y})]\gamma_\delta(\mathbf{x},\mathbf{y})v(\mathbf{x})d\mathbf{y}d\mathbf{x}=(\sigma\frac{\partial u}{\partial \mathbf{n}},v)_{\partial\Omega}+\mathcal{O}(\delta).
    \end{equation}
    As for the first term to the right of the equation, we have:
    \begin{align}\label{eqnapart}
        \int_{\mathbf{x}\in\Omega_\gamma^+}&\int_{\mathbf{y}\in\Omega_\gamma^+}[u(\mathbf{x})-u(\mathbf{y})]\gamma_\delta(\mathbf{x},\mathbf{y})v(\mathbf{x})d\mathbf{y}d\mathbf{x}\nonumber\\&=\frac{1}{2} \int_{\mathbf{x}\in\Omega_\gamma^+}\int_{\mathbf{y}\in\Omega_\gamma^+}[u(\mathbf{x})-u(\mathbf{y})][v(\mathbf{x})-v(\mathbf{y})]\gamma_\delta(\mathbf{x},\mathbf{y})d\mathbf{y}d\mathbf{x}\\&=\frac{1}{2} \int_{\mathbf{x}\in\Omega_\gamma^+}\int_{\mathbf{y}\in\Omega_\gamma^+}[(\mathbf{x}-\mathbf{y})^T\nabla u(\mathbf{x})+(\mathbf{x}-\mathbf{y})^T H_u(\mathbf{\xi})(\mathbf{x}-\mathbf{y})]\nonumber\\&\quad\times
        [(\mathbf{x}-\mathbf{y})^T\nabla v(\mathbf{x})+(\mathbf{x}-\mathbf{y})^TH_v(\mathbf{\eta})(\mathbf{x}-\mathbf{y})]\gamma_\delta(\mathbf{x},\mathbf{y})d\mathbf{y}d\mathbf{x},\nonumber
    \end{align}
    where $H_u(\mathbf{x})$ is the hessian matrix of function $u$ in point $\mathbf{x}$. Denoting $(\cos(\theta),\sin(\theta))$ as $V(\theta)$ and substituting $\mathbf{y}$ with $\mathbf{x}+V(\theta)^T$, we have:
    \begin{align}
         \int_{\mathbf{x}\in\Omega_\gamma^+}&\int_{\mathbf{y}\in\Omega_\gamma^+}[u(\mathbf{x})-u(\mathbf{y})]\gamma_\delta(\mathbf{x},\mathbf{y})v(\mathbf{x})d\mathbf{y}d\mathbf{x}\nonumber\\&= \frac{1}{2} \int_{\mathbf{x}\in\Omega_\gamma^+}\int_{\theta=-\pi}^\pi\int_{s=0}^{r_u(\mathbf{x},\theta)}s^3[V(\theta)\cdot\nabla u(\mathbf{x})+sV(\theta) H_u(\xi)V(\theta)^T]\nonumber\\&\quad\times
        [V(\theta)\cdot\nabla v(\mathbf{x})+sV(\theta) H_v(\eta)V(\theta)^T]\widetilde{\gamma}_\delta(s)dsd\theta d\mathbf{x},
    \end{align}
    where $r_u(\mathbf{x},\theta)=\max_r\{\mathbf{x}+rV(\theta)^T\in\Omega^+_\gamma\}$. From the regularity of function u and v, with the condition that $\delta$ is small enough, we can get that:
    \begin{align}\label{eqn450}
        \int_{\mathbf{x}\in\Omega_\gamma^+}&\int_{-\pi}^\pi\int_{0}^{r_u(\mathbf{x},\theta)}s^3[V(\theta)\cdot\nabla u(\mathbf{x})+sV(\theta) H_u(\xi)V(\theta)^\prime]\nonumber\\&\quad\cdot
        [V(\theta)\cdot\nabla v(\mathbf{x})+sV(\theta) H_v(\eta)V(\theta)^\prime]\widetilde{\gamma}_\delta(s)dsd\theta d\mathbf{x}\nonumber\\&
        \leq \int_{\mathbf{x}\in\Omega_\gamma^+}\int_{-\pi}^\pi\int_{0}^{r_u(\mathbf{x},\theta)} (s^3\vert V(\theta)\cdot\nabla u(\mathbf{x})V(\theta)\cdot\nabla v(\mathbf{x})\vert+Cs^4)\widetilde{\gamma}_\delta(s)dsd\theta d\mathbf{x}\nonumber\\&
        \leq \int_{\mathbf{x}\in\Omega_\gamma^+}C +\mathcal{O}(\delta)d\mathbf{x}\leq C\delta.
    \end{align}
    From (\ref{eqndiv}), (\ref{pastre}) and (\ref{eqn450}), it has that:
    \begin{equation}
        (\N_{\Omega}^{*} u,v)_{\Omega_{\gamma}^+}=(\sigma\frac{\partial u}{\partial \mathbf{n}},v)_{\partial\Omega}+\mathcal{O}(\delta).
    \end{equation}
    The theorem has been proved.
\end{proof}
\subsection{The proof of Theorem \ref{mainresult}}\label{append3}
\begin{proof}
   \begin{figure}[h]
    \centering
    \includegraphics[width=120mm]{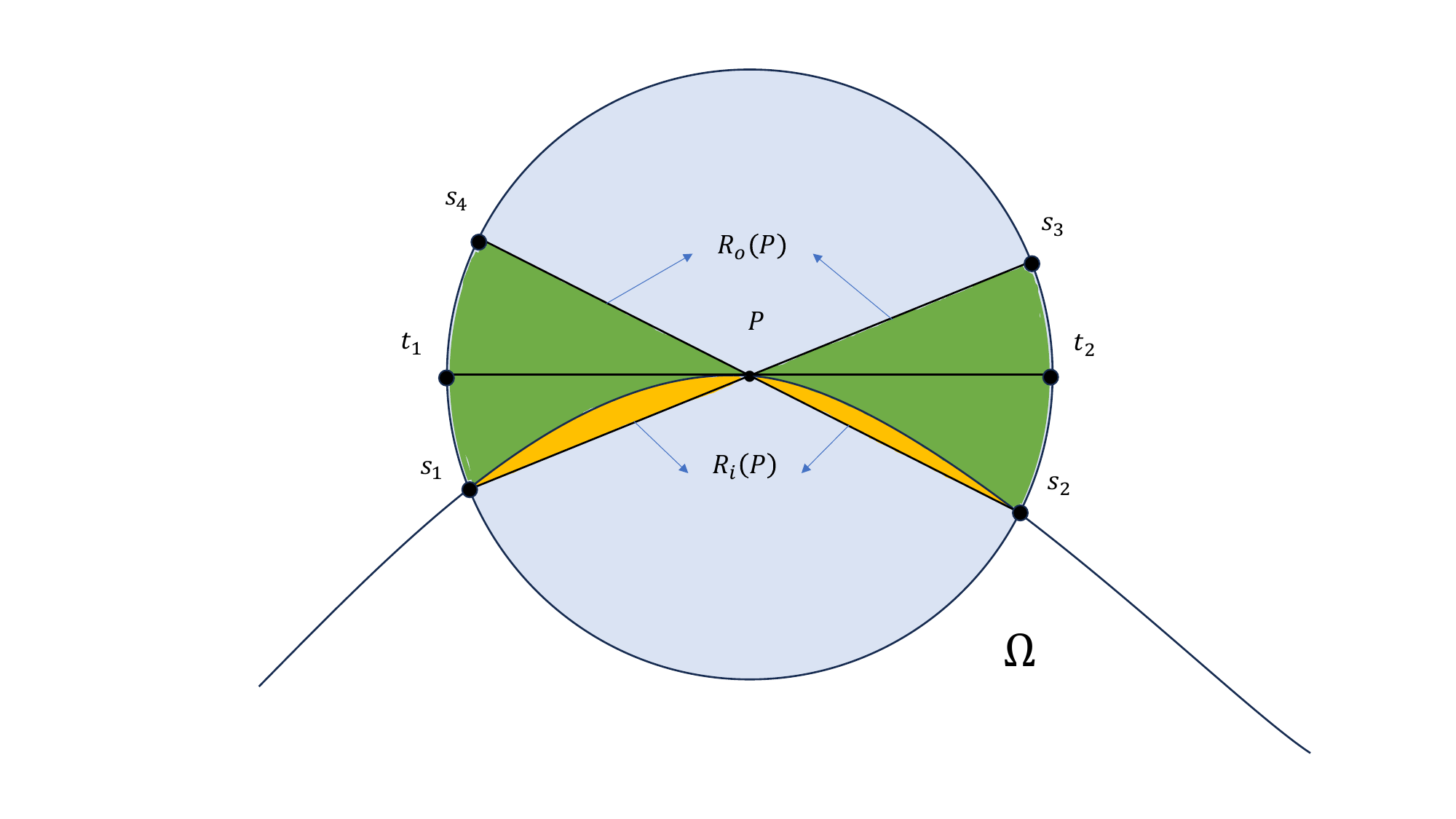}
    \caption{Schematic when the region is $\Omega$}
    \label{fig2:env}
    \end{figure}
    Taking $\mathbf{p}$ as a point in $\partial\Omega$, from Lemma \ref{convex} and \ref{intersection}, it can be got that there exists $\delta_1>0$, when $\delta<\delta_1$, there are  always two intersection points of $B(\mathbf{p},\delta)$ and $\partial\Omega$.
    With the convex assumption, the line whose vertices are $\forall \mathbf{p_1} \in \Omega_\gamma^+$ and $\forall \mathbf{p_2} \in \mathring{\Omega}_\gamma^-$ has only one intersection point with $\partial\Omega$. Then, the variable substitution (\ref{coordtrans}) is reasonable. As illustrated in Figure \ref{fig2:env}, denote the two intersection points of $B(\mathbf{p}, \delta)$ with $\partial\Omega$ as $\mathbf{s_1}$ and $\mathbf{s_2}$, and the two intersection points with the tangent line of $\partial\Omega$ at $\mathbf{p}$ as $\mathbf{t_1}$ and $\mathbf{t_2}$. For $\mathbf{p_1}$ lying in the sector $\mathbf{s_1}\mathbf{x}\mathbf{s_2}$, the variables $(s, t)$ are restricted to $s \in (0, \delta)$ and $t \in (0, \delta - s)$, ensuring the mapping $V_n$ has a straightforward structure within this region. Using $\overset{\frown}{abc}$ to denote the arc sector $abc$, the remaining region is partitioned into two components:
    \begin{equation}
    \left\{
    \begin{aligned}
    &R_i(\mathbf{p})=(\mathring{\Omega}_\gamma^-\cap B(\mathbf{p},\delta))/\overset{\frown} {\mathbf{s_1}\mathbf{p}\mathbf{s_2}},\\
    &R_o(\mathbf{p})=(\Omega_\gamma^+\cap B(\mathbf{p},\delta))/\overset{\frown} {\mathbf{s_3}\mathbf{p}\mathbf{s_4}},
\end{aligned}
    \right.
\end{equation}
where $\mathbf{s_3}$ and $\mathbf{s_4}$ are the another points line $\mathbf{s_1} \mathbf{p}$ and $\mathbf{s_2} \mathbf{p}$ intersects with $B(\mathbf{p},\delta)$. 
 Divided $\mathring{\Omega}_\gamma^-\cap B(\mathbf{p},\delta)$ into two parts and use $\mathbf{p}$, $\mathbf{x}$ and $|\mathbf{p}\mathbf{y}|$ to represent $\mathbf{y}$, it can be obtained that:
    \begin{align}\label{divide2}
    &(\N_{\Omega}^- u,v)_{\Omega_{\gamma}^-}\nonumber\\&
    =\int_{\mathbf{x}\in\Omega_{\gamma}^-}\int_{\mathbf{y}\in\Omega_{\gamma}^+}[u(\mathbf{y})-u(\mathbf{x})]v(\mathbf{x})\gamma_\delta(\mathbf{x},\mathbf{y})d\mathbf{y}d\mathbf{x}\nonumber\\&
    =\int_{\mathbf{p}\in\partial\Omega}\int_{\mathbf{x}\in \Omega_{\gamma}^-}\int_{s=0}^{rg_s}[u(\mathbf{p}+s(\frac{\mathbf{p}-\mathbf{x}}{|\mathbf{p}-\mathbf{x}|}))-u(\mathbf{x})]v(\mathbf{x})\gamma_\delta(\mathbf{x},\mathbf{p}+s(\frac{\mathbf{p}-\mathbf{x}}{|\mathbf{p}-\mathbf{x}|}))\det(J_y)dtd\mathbf{x}d\mathbf{p}\nonumber
    \\&
    =\int_{\mathbf{p}\in\partial\Omega}\int_{\mathbf{x}\in \overset{\frown} {\mathbf{s_1}\mathbf{p}\mathbf{s_2}}}\int_{s=0}^{rg_s}[u(\mathbf{p}+s(\frac{\mathbf{p}-\mathbf{x}}{|\mathbf{p}-\mathbf{x}|}))-u(\mathbf{x})]v(\mathbf{x})\gamma_\delta(\mathbf{x},\mathbf{p}+s(\frac{\mathbf{p}-\mathbf{x}}{|\mathbf{p}-\mathbf{x}|}))\det(J_y)dsd\mathbf{x}d\mathbf{p}\nonumber\\&
    \quad+\int_{\mathbf{p}\in\partial\Omega}\int_{\mathbf{x}\in R_i(\mathbf{p})}\int_{s=0}^{rg_s}[u(\mathbf{p}+s(\frac{\mathbf{p}-\mathbf{x}}{|\mathbf{p}-\mathbf{x}|}))-u(\mathbf{x})]v(\mathbf{x})\gamma_\delta(\mathbf{x},\mathbf{p}+s(\frac{\mathbf{p}-\mathbf{x}}{|\mathbf{p}-\mathbf{x}|}))\det(J_y)dsd\mathbf{x}d\mathbf{p},
    \end{align} 
where $rg_s=\max\{0, \delta-|\mathbf{p}-\mathbf{x}|\}$, $\det(J_y)$ is the determinant of Jacobian matrix of the $\mathbf{y}$-substitution process. Noting $\mathbf{p}+s(\frac{\mathbf{p}-\mathbf{x}}{|\mathbf{p}-\mathbf{x}|})$ as $\mathbf{y}(\mathbf{p},\mathbf{x},s)$, then, (\ref{divide2}) can be written as follows:
    \begin{align}\label{divide3part328}
        &(\N_{\Omega}^- u,v)_{\Omega_{\gamma}^-}\nonumber\\&
        =\int_{\mathbf{p}\in\partial\Omega}\int_{\mathbf{x}\in \overset{\frown} {\mathbf{t_1}\mathbf{p}\mathbf{t_2}}}\int_{s=0}^{rg_s}[u(\mathbf{y}(\mathbf{p},\mathbf{x},s))-u(\mathbf{x})]v(\mathbf{x})\gamma_\delta(\mathbf{x},\mathbf{y}(\mathbf{p},\mathbf{x},s))\det(J_y)dsd\mathbf{x}d\mathbf{p} \nonumber\\&
\quad-\int_{\mathbf{p}\in\partial\Omega}\int_{\mathbf{x}\in \overset{\frown} {\mathbf{t_1}\mathbf{p}\mathbf{s_1}}\cup\overset{\frown} {\mathbf{s_2}\mathbf{p}\mathbf{t_2}}}\int_{s=0}^{rg_s}[u(\mathbf{y}(\mathbf{p},\mathbf{x},s))-u(\mathbf{x})]v(\mathbf{x})\gamma_\delta(\mathbf{x},\mathbf{y}(\mathbf{p},\mathbf{x},s))\det(J_y)dsd\mathbf{x}d\mathbf{p}\nonumber\\&
\quad
+\int_{\mathbf{p}\in\partial\Omega}\int_{\mathbf{x}\in R_i(\mathbf{p})}\int_{s=0}^{rg_s}[u(\mathbf{y}(\mathbf{p},\mathbf{x},s))-u(\mathbf{x})]v(\mathbf{x})\gamma_\delta(\mathbf{x},\mathbf{y}(\mathbf{p},\mathbf{x},s))\det(J_y)dsd\mathbf{x}d\mathbf{p}\\&=:A_1-A_2+A_3.\nonumber
    \end{align}
We consider the term $A_1$ firstly and assume that the parametric expressions for $\partial\Omega$ is $\Gamma(\eta)=(x(\eta),y(\eta)),~\eta\in[0,\beta]$. Comparing with the map $g$ mentioned in Theorem \ref{oneonone}, it can be easily found that as we substitute $\mathbf{x}$ with $\mathbf{p}(\eta)-tV_c(\theta,\eta)$, where $t=|\mathbf{p}\mathbf{x}|$, $A_1$ is totally the same as (\ref{defweakformtaylor}) in Theorem \ref{oneonone}. From the result in Theorem \ref{oneonone}, it has:
\begin{align}\label{oldresult}
    A_1&=\int_{s=0}^{\delta}\int_{t=0}^{\delta-s}\int_{\eta=0}^{\beta} \int_{\theta=-\frac{\pi}{2}}^{\frac{\pi}{2}}\left[(s+t) V_c(\theta,\eta)^T \cdot \nabla u(\mathbf{p}(\eta)) +(s^2-t^2)[u(\theta,\eta)]_\mathcal{H}+\mathcal{O}(\delta^3) \right] \nonumber\\&\quad
\det(J)\widetilde{\gamma}_\delta(s+t)\left[v(\mathbf{p}(\eta)) - tV_c(\theta,\eta)^T\cdot \nabla v(\mathbf{p}(\eta))+ \mathcal{O}(\delta^2)
\right] d\theta d\eta dtds\nonumber
\\&=\langle\sigma\partial_\mathbf{n} u,v\rangle_{\partial\Omega}+\mathcal{O}(\delta).
\end{align}
Next, we proceed to derive an upper bound in terms of $\delta$  for the remaining two components of (\ref{divide2}). Consider a Cartesian coordinate system centered at $\mathbf{p}\in\partial\Omega$
, where the x-axis is aligned with the tangent vector at $\mathbf{p}$ and the y-axis with the outer normal vector at $\mathbf{p}$. 
    Then, we can get that:
    \begin{align}
        \vert A_2-A_3\vert&=|\int_{\mathbf{p}\in\partial\Omega}\int_{\mathbf{x}\in R_i(\mathbf{p})}\int_{s=0}^{rg_s}[u(\mathbf{y}(\mathbf{p},\mathbf{x},s))-u(x)]v(x)\gamma_\delta(\mathbf{x},\mathbf{y}(\mathbf{p},\mathbf{x},s))\det(J_y)dsd\mathbf{x}d\mathbf{p}\nonumber\\&\quad-\int_{\mathbf{p}\in\partial\Omega}\int_{x\in \overset{\frown} {\mathbf{t_1}\mathbf{p}\mathbf{s_1}}\cup\overset{\frown} {\mathbf{s_2}\mathbf{p}\mathbf{t_2}}}\int_{s=0}^{rg_s}[u(\mathbf{y}(\mathbf{p},\mathbf{x},s))-u(x)]v(x)\gamma_\delta(\mathbf{x},\mathbf{y}(\mathbf{p},\mathbf{x},s))\det(J_y)dsd\mathbf{x}d\mathbf{p}|\nonumber\\&
        =\int_{\eta=0}^{\beta} \int_{\theta\in ts(\eta)}\int_{s=0}^{\delta}\int_{t=\min\{\delta-s,tmax\{p,\theta\}\}}^{\delta-s}\widetilde{\gamma}_\delta(s+t)|\det(J)(v(\mathbf{p}(\eta))+\mathcal{O}(\delta)| \nonumber\\&\quad \times\left|(s+t) V_c(\theta,\eta)^T\cdot \nabla u(\mathbf{p}(\eta))+\mathcal{O}(\delta^2)\right| d\theta d\eta dtds, 
\end{align}
where 
$tmax=\mathop{\sup}\limits_{t} \{\mathbf{p}(\eta)-tV_c(\theta,\eta) \in \Omega_{\gamma}^-\}$ and $ts(\eta)=(-\frac{\pi}{2},-\frac{\pi}{2}+\angle t_1\mathbf{p}(\eta)s_1)\cup(\frac{\pi}{2}-\angle t_2\mathbf{p}(\eta)s_2,\frac{\pi}{2}) $.
Then, by the enlargement of the integration area, it has the following inequality:
    \begin{align}\label{eqn331new}
         \vert A_2-A_3\vert &\leq
         \int_{s=0}^{\delta}\int_{t=0}^{\delta-s}\int_{\eta=0}^{\beta} \int_{\theta\in ts(\eta)}\widetilde{\gamma}_\delta(s+t)
\vert v(\mathbf{p}(\eta))+\mathcal{O}(\delta)\vert\nonumber\\&\quad \times\left|\det(J)((s+t) V_c(\theta,\eta)^T\cdot \nabla u(\mathbf{p}(\eta))+\mathcal{O}(\delta^2))\right|  d\theta d\eta dtds  \nonumber\\&
\leq
\int_{r=0}^{\delta}\int_{\eta=0}^{\beta} \int_{\theta\in ts(\eta)}r^3\widetilde{\gamma}_\delta(r)\sqrt{(x^\prime(\eta))^2+(y^\prime(\eta))^2}
\vert v(\mathbf{p}(\eta))\vert\left|\cos(\theta) V_c(\theta,\eta)^T\cdot \nabla u(\mathbf{p}(\eta))\right| d\theta d\eta dr\nonumber\\&\quad+\int_{s=0}^{\delta}\int_{t=0}^{\delta-s}\int_{\eta=0}^{\beta} \int_{\theta\in ts(\eta)}\widetilde{\gamma}_\delta(s+t)
\vert \det(J)(\mathcal{O}(\delta)(s+t) V_c(\theta,\eta)^T\cdot \nabla u(\mathbf{p}(\eta))\nonumber\\&\quad\quad+\mathcal{O}(\delta^2)v(\mathbf{p}(\eta))+\mathcal{O}(\delta^3))\vert d\theta d\eta dtds,
    \end{align}
    Noting that in this coordinate system, $u(x)$ is the expression of boundary and $\vert x_1-x\vert\leq\delta$, using taylor expansion, it can be obtained that:
    \begin{align}
        u(x_1)&=u(x)+\frac{du}{d\mathbf{x}}(x)(x_1-x)+\frac{1}{2}\frac{d^2u}{d\mathbf{x}^2}(x)(x_1-x)^2+O(x_1-x)^3\nonumber\\&
        =\frac{1}{2}\kappa(x)(x_1)^2+O(x_1^3).
    \end{align}
    Next, we can approximate $\angle \mathbf{t_1}\mathbf{p}\mathbf{s_1}$ and $\angle \mathbf{t_2}\mathbf{p}\mathbf{s_2}$. From the properties of convex functions, we can get that:
    \begin{equation}\label{angleapp}
        \max\{\tan\angle \mathbf{t_1}\mathbf{p}\mathbf{s_1},\tan\angle \mathbf{t_2}\mathbf{p}\mathbf{s_2}\}\leq \frac{u(\delta)}{\delta}=\frac{1}{2}\kappa(\mathbf{p})\delta+\mathcal{O}(\delta^2).
    \end{equation}
From (\ref{angleapp}) and the regularity of the boundary, we can get the approximation to sine and cosine values of the angle:
\begin{align}
    &\max\{{\angle \mathbf{t_1}\mathbf{p}\mathbf{s_1},\angle \mathbf{t_2}\mathbf{p}\mathbf{s_2}}\}\leq \arctan \frac{u(\delta)}{\delta}\leq \frac{K}{2}\delta+\mathcal{O}(\delta^2),\\
    &\min\{\cos\angle \mathbf{t_1}\mathbf{p}\mathbf{s_1},\cos\angle \mathbf{t_2}\mathbf{p}\mathbf{s_2}\}\geq \frac{1}{\sqrt{1+(\frac{u(\delta)}{\delta})^2}}\geq1-\frac{K^2}{8}\delta^2+\mathcal{O}(\delta^3).
\end{align}
Then, as for $\theta\in (\frac{\pi}{2}-\angle \mathbf{t_2}\mathbf{p}\mathbf{s_2},\frac{\pi}{2})$, it has that:
\begin{equation}
    \begin{aligned}
    &\cos(\theta)=\sin(\frac{\pi}{2}-\theta)\leq \sin(\angle \mathbf{t_2}\mathbf{p}\mathbf{s_2})\leq \frac{K}{2}\delta+\mathcal{O}(\delta^2),\\&
    1\geq\sin(\theta)=\cos(\frac{\pi}{2}-\theta)\geq \cos(\angle \mathbf{t_2}\mathbf{p}\mathbf{s_2})\geq 1-\frac{K^2}{8}\delta^2+\mathcal{O}(\delta^3).
    \end{aligned}
\end{equation}
And for $\theta\in ts(\eta)=(-\frac{\pi}{2},-\frac{\pi}{2}+\angle \mathbf{t_1}\mathbf{p}\mathbf{s_1})$, it has that:
\begin{equation}
    \begin{aligned}
        &\cos(\theta)=\sin(\frac{\pi}{2}-\theta)\leq \sin(\angle \mathbf{t_1}\mathbf{p}\mathbf{s_1})\leq \frac{K}{2}\delta+\mathcal{O}(\delta^2),\\&
    -1\leq\sin(\theta)=\cos(\frac{\pi}{2}-\theta)\leq \cos(\angle \mathbf{t_1}\mathbf{p}\mathbf{s_1})\leq 1-\frac{K^2}{8}\delta^2+\mathcal{O}(\delta^3).
    \end{aligned}
\end{equation}
Then, we have:
\begin{align}
    &\int_{\theta\in ts(\eta)}\vert\cos(\theta) \left(\begin{array}{cc}
    \cos(\theta) & \sin(\theta) \\
     -\sin(\theta)& \cos(\theta)
\end{array}\right)\vert d\theta\nonumber\\&
 \leq\int_{\theta\in ts(\eta)}(\frac{K}{2}\delta+\mathcal{O}(\delta^2))\vert\left(\begin{array}{cc}
    \cos(\theta) & \sin(\theta) \\
     -\sin(\theta)& \cos(\theta)
\end{array}\right)\vert d\theta\nonumber\\&
\leq \frac{1}{4}\left(\begin{array}{cc}
   K^3\delta^3+\mathcal{O}(\delta^4) & 2K^2\delta^2+\mathcal{O}(\delta^3) \\
     2K^2\delta^2+\mathcal{O}(\delta^3)&  K^3\delta^3+\mathcal{O}(\delta^4)
\end{array}\right).
\end{align}
Due to the boundedness of $\Omega$ and the regularity $u \in C^3(\Omega^+)$, $v \in C^2(\Omega^+)$, we impose the a priori bounds $|\partial^\alpha u(\mathbf{x})| \leq C_u$ for $|\alpha| \leq 2$, $\mathbf{x}\in\Omega^+$ and $|\partial^\alpha v(\mathbf{x})| \leq C_v$ for $|\alpha| \leq 1$, $\mathbf{x}\in\Omega^+$. We have the following estimate for the first term at the right end of the last inequality for (\ref{eqn331new}):
\begin{align}\label{rest0}
\int_{r=0}^{\delta}&\int_{\eta=0}^{\beta} \int_{\theta\in ts(\eta)}r^3\widetilde{\gamma}_\delta(r)\sqrt{(x^\prime(\eta))^2+(y^\prime(\eta))^2}
\vert v(\mathbf{p}(\eta))\vert\left|\cos(\theta) V_c(\theta,\eta)^T\cdot \nabla u(\mathbf{p}(\eta))\right| d\theta d\eta dr
    \nonumber\\&\leq 
    \int_{r=0}^{\delta}r^3\widetilde{\gamma}_\delta(r)\int_{\eta=0}^{\beta} 2\bigg\vert\int_{\theta\in ts(\eta)}\vert \cos(\theta)\left(\begin{array}{cc}
    \cos(\theta) & \sin(\theta) \\
     -\sin(\theta)& \cos(\theta)
\end{array}\right)\vert d\theta\bigg\vert_{\infty}\nonumber
 \\&\quad\cdot \| \left(
     n_{1p} ~ n_{2p}
\right)\|_{1}\|\nabla u(\mathbf{p}(\eta))\|_{1}\nonumber \sqrt{(x^\prime(\eta))^2+(y^\prime(\eta))^2}
\vert v(\mathbf{p}(\eta))\vert  d\eta dr\nonumber\\&
\leq 
\int_{r=0}^{\delta}r^3\widetilde{\gamma}_\delta(r)\int_{\eta=0}^{\beta}8C_uC_v(\frac{K^2\delta^2}{2}+\mathcal{O}(\delta^{3}))\sqrt{(x^\prime(\eta))^2+(y^\prime(\eta))^2} d\eta dr\nonumber\\&
\leq 
4\beta C_uC_vK^2\delta^2 \int_{r=0}^{\delta}r^3\widetilde{\gamma}_\delta(r)(1+\mathcal{O}(\delta)) dr=\mathcal{O}(\delta^2).
\end{align}
Totally the same, as for the second part in (\ref{eqn331new}), by direct scaling and the fact that the length of the integration interval of $\theta$ is $\mathcal{O}(\delta)$, it has that:
\begin{align}\label{eqn340new}
    \int_{s=0}^{\delta}\int_{t=0}^{\delta-s}\int_{\eta=0}^{\beta} &\int_{\theta\in ts(\eta)}\widetilde{\gamma}_\delta(s+t)
\vert \det(J)( \mathcal{O}(\delta)(s+t) V_c(\theta,\eta)^T\cdot \nabla u(\mathbf{p}(\eta))\nonumber\\&\quad\quad+\mathcal{O}(\delta^2)v(\mathbf{p}(\eta))+\mathcal{O}(\delta^3))\vert d\theta d\eta dtds=\mathcal{O}(\delta^2).
\end{align}
From (\ref{divide3part328}), (\ref{oldresult}), (\ref{rest0}) and (\ref{eqn340new}), it's easy to get (\ref{C3result}), the proof is completed.
\end{proof}

\subsection{The proof of Theorem \ref{rectangleright}}\label{append4}

\begin{proof}
    We assuming that the number of vertices is n, the angle whose vertex $\mathbf{a}$ is the vertex of polygon $P$ is $\theta_a$ and the point $\mathbf{p}\in\Gamma(r)$ can be parameterized by arc length as $\mathbf{p(\eta)}=(x(\eta),y(\eta)),~(0\leq\eta\leq\beta)$. Let's first divide the boundary to the following four parts:
    \begin{equation}
    \left\{
    \begin{aligned}
        &\Gamma_{1\delta r}=\{\mathbf{p|\mathbf\{p}~ is~the~vertex~of~polygon\},\\
        &\Gamma_{2\delta r}=\{\mathbf{p}|\mathbf{p}\in\Gamma_r, \mathbf{P}\in \Gamma_{1\delta r}, \theta_P\geq \pi/2, dist(\mathbf{p},\mathbf{P})< \delta\},\\
        &\Gamma_{3\delta r}=\{\mathbf{p}|\mathbf{p}\in\Gamma_r, \mathbf{P}\in \Gamma_{1\delta r}, \theta_P< \pi/2, dist(\mathbf{p},\mathbf{P})< \delta/\sin(\theta_P)\},\\
        &\Gamma_{4\delta r}=\Gamma_r-\Gamma_{1\delta r}-\Gamma_{2\delta r}-\Gamma_{3\delta r}.
    \end{aligned}
    \right.
    \end{equation}
    Fixing $\delta$ is much smaller than the length of the sides of the region 
    without loss of generality, let $\Omega$ be an open region. It is readily shown that the range of mapping $V_n$ is totally the same with the half-plane situation if the intersection point $\mathbf{p}$ lies in $\Gamma_{4\delta r}$. For $\mathbf{p} \in \Gamma_{1\delta r}$, assume the outer normal vector coincides with the angle bisector at vertex $\mathbf{p}$. There exists an angle $\theta_p < \pi/2$ such that $g$ is bijective when $\theta$ ranges over $(-\theta_p, \theta_p)$. Consequently, the inner product $(\N_{\Omega}^- u, \D_\Omega^-v)_{\Omega_\gamma^+}$ can be decomposed into four components:
\begin{align}
&(\N_{\Omega}^- u,\D_\Omega^-v)_{\Omega_\gamma^+}\nonumber\\
&=\int_{s=0}^{\delta}\int_{t=0}^{\delta-s}\int_{\mathbf{p}\in\Gamma_{4\delta r}} \int_{\theta=-\frac{\pi}{2}}^{\frac{\pi}{2}}\left[(s+t)  G_c(\theta,\mathbf{p})+\mathcal{O}(\delta^2)\right]\widetilde{\gamma}_\delta(s+t)\det(J)
(v(\mathbf{p})+\mathcal{O}(\delta))  d\theta d\mathbf{p} dtds\nonumber\\&\quad+
\int_{s=0}^{\delta}\int_{t=0}^{\delta-s}\int_{\mathbf{p}\in\Gamma_{1\delta r}} \int_{\theta=-\theta_p}^{\theta_p}\left[(s+t)  G_c(\theta,\mathbf{p})+\mathcal{O}(\delta^2)\right]\widetilde{\gamma}_\delta(s+t)\det(J)
(v(\mathbf{p})+\mathcal{O}(\delta))  d\theta d\mathbf{p} dtds \nonumber\\& \quad+
(\int_{\theta=-\frac{\pi}{2}}^{\frac{\pi}{2}}\int_{\mathbf{p}\in\Gamma_{2\delta r}} \int_{s=0}^{\delta}\int_{t=0}^{t_{up}}\left[(s+t)  G_c(\theta,\mathbf{p})+\mathcal{O}(\delta^2)\right]\widetilde{\gamma}_\delta(s+t)\det(J)
(v(\mathbf{p})+\mathcal{O}(\delta))  d\theta d\mathbf{p} dtds\nonumber\\& \quad+
(\int_{\theta=-\frac{\pi}{2}}^{\frac{\pi}{2}}\int_{\mathbf{p}\in\Gamma_{2\delta r}} \int_{s=0}^{\delta}\int_{t=0}^{t_{up}}\left[(s+t)  G_c(\theta,\mathbf{p})+\mathcal{O}(\delta^2)\right]\widetilde{\gamma}_\delta(s+t)\det(J)
(v(\mathbf{p})+\mathcal{O}(\delta))  d\theta d\mathbf{p} dtds,
\end{align}
where $G_c(\theta,\mathbf{p})=V_c(\theta,\mathbf{p})^T\cdot \nabla u(\mathbf{p})$ $t_{up}=\min\{\delta-s,tmax(\mathbf{p},\theta)\}$, $tmax=\mathop{\sup}\limits_{t} \{\mathbf{p}(\eta)-tV_c(\theta,\eta) \in \Omega_{\gamma}^- \in \Omega_\gamma^+\}$. As for the second integration part, because the measurement of $\Gamma_{1\delta r}$ equals zero, the result of integration equals zero. Comparing the first integration  with the variable substitution equation (\ref{defweakformtaylor}) in Theorem \ref{oneonone}, it has that:
\begin{align}
\int_{s=0}^{\delta}\int_{t=0}^{\delta-s}&\int_{\mathbf{p}\in\Gamma_{4\delta r}} \int_{\theta=-\frac{\pi}{2}}^{\frac{\pi}{2}}\left[(s+t)  G_c(\theta,\mathbf{p})+\mathcal{O}(\delta^2)\right]\widetilde{\gamma}_\delta(s+t)\det(J)
(v(\mathbf{p})+\mathcal{O}(\delta))  d\theta d\mathbf{p} dtds\nonumber\\&=(\sigma\frac{\partial u}{\partial\mathbf{n}},v)_{\Gamma_{4\delta r}}+\mathcal{O}(\delta).
\end{align}
Together with (\ref{defweakformtaylor}), (\ref{enestresult}) and (\ref{enestresult}),  we can obtain that:
\begin{align}\label{difference_rec}
    &\vert\langle\sigma\partial_\mathbf{n} u,v\rangle_{\Gamma_r}-
    (\N_{\Omega}^- u,\D_\Omega^-v)_{\Omega_\gamma^+}\vert\nonumber\\&
    \leq
    |\int_{\theta=-\frac{\pi}{2}}^{\frac{\pi}{2}}\int_{p\in\Gamma_{2\delta r}} \int_{s=0}^{\delta}\int_{t=t_{up}}^{\delta-s}\left[(s+t)  G_c(\theta,\mathbf{p})+\mathcal{O}(\delta^2)\right]\widetilde{\gamma}_\delta(s+t)\det(J)
(v(\mathbf{p})+\mathcal{O}(\delta))  d\theta d\mathbf{p} dtds
+\mathcal{O}(\delta)|
\nonumber\\&\quad+|\int_{\theta=-\frac{\pi}{2}}^{\frac{\pi}{2}}\int_{p\in\Gamma_{3\delta r}} \int_{s=0}^{\delta}\int_{t=t_{up}}^{\delta-s}\left[(s+t)  G_c(\theta,\mathbf{p})+\mathcal{O}(\delta^2)\right]\widetilde{\gamma}_\delta(s+t)\det(J)
(v(\mathbf{p})+\mathcal{O}(\delta))  d\theta d\mathbf{p} dtds
+\mathcal{O}(\delta)|.
\end{align}
Next, we analyze the first term at the right hand side of the inequality. Substituting $t+s$ with $r$, we have:
\begin{align}\label{eqn349new}
    |&\int_{\theta=-\frac{\pi}{2}}^{\frac{\pi}{2}}\int_{p\in\Gamma_{2\delta r}} \int_{s=0}^{\delta}\int_{t=t_{up}}^{\delta-s}\left[(s+t)  G_c(\theta,\mathbf{p})+\mathcal{O}(\delta^2)\right]\widetilde{\gamma}_\delta(s+t)\det(J)
(v(\mathbf{p})+\mathcal{O}(\delta))  d\theta d\mathbf{p} dtds
+\mathcal{O}(\delta)|\nonumber\\&\leq 
|\int_{\theta=-\frac{\pi}{2}}^{\frac{\pi}{2}}\int_{p(\eta)\in\Gamma_{2\delta r}} \int_{r=t_{up}}^{\delta}\left[r^2(r-t_{up})  G_c(\theta,\mathbf{p})+\mathcal{O}(\delta^3)\right]\widetilde{\gamma}_\delta(r)\cos(\theta)\sqrt{(x^\prime(\eta))^2+(y^\prime(\eta))^2}\nonumber\\&\quad
\times(v(\mathbf{p})+\mathcal{O}(\delta))  d\theta d\mathbf{p} dtds
+\mathcal{O}(\delta)|.
\end{align}
 Due to the boundedness of $\Omega$ and the regularity $u \in C^3(\Omega^+)$, $v \in C^2(\Omega^+)$, we impose the a priori bounds $|\partial^\alpha u(\mathbf{x})| \leq C_u$ for $|\alpha| \leq 2$, $\mathbf{x}\in\Omega^+$ and $|\partial^\alpha v(\mathbf{x})| \leq C_v$ for $|\alpha| \leq 1$, $\mathbf{x}\in\Omega^+$. From the definition of $V_c(\theta,\mathbf{p})^T$, it has that:
 \begin{align}
     |\cos(\theta)G_c(\theta,\mathbf{p})|\leq2C_u.
 \end{align}
 Then, (\ref{eqn349new}) turns out to be:
 \begin{align}\label{eqn349neww}
    |&\int_{\theta=-\frac{\pi}{2}}^{\frac{\pi}{2}}\int_{p\in\Gamma_{2\delta r}} \int_{s=0}^{\delta}\int_{t=t_{up}}^{\delta-s}\left[(s+t)  G_c(\theta,\mathbf{p})+\mathcal{O}(\delta^2)\right]\widetilde{\gamma}_\delta(s+t)\det(J)
(v(\mathbf{p})+\mathcal{O}(\delta))  d\theta d\mathbf{p} dtds
+\mathcal{O}(\delta)|\nonumber\\&\leq 
\int_{r=t_{up}}^{\delta}\int_{p(\eta)\in\Gamma_{2\delta r}}(2\pi C_ur^2(r-t_{up})+\mathcal{O}(\delta^3))(C_v+\mathcal{O}(\delta))\widetilde{\gamma}_\delta(r)dpdr+\mathcal{O}(\delta)\nonumber\\&\leq L(\Gamma_{2\delta r})(2\pi C_uC_v\int_{r=t_{up}}^{\delta}r^2(r-t_{up})\widetilde{\gamma}_\delta(r)dr+\mathcal{O}(\delta))+\mathcal{O}(\delta)\nonumber\\&\leq 2n\delta (2\pi C_uC_v\sigma+\mathcal{O}(\delta))+\mathcal{O}(\delta)\leq C\delta.
\end{align}
Totally the same, as for the second term at the right hand side of the inequality (\ref{difference_rec}), if $\Gamma_{3\delta r}\neq \emptyset$, then, it has the following estimation:
\begin{align}\label{eqn352new}
    |&\int_{\theta=-\frac{\pi}{2}}^{\frac{\pi}{2}}\int_{p\in\Gamma_{3\delta r}} \int_{s=0}^{\delta}\int_{t=t_{up}}^{\delta-s}\left[(s+t)  G_c(\theta,\mathbf{p})+\mathcal{O}(\delta^2)\right]\widetilde{\gamma}_\delta(s+t)\det(J)
(v(\mathbf{p})+\mathcal{O}(\delta))  d\theta d\mathbf{p} dtds
+\mathcal{O}(\delta)|\nonumber\\&\leq 2\frac{\delta}{\sin(C_p)} (2\pi C_uC_v\sigma+\mathcal{O}(\delta))+\mathcal{O}(\delta)\leq C\delta.
\end{align}
From (\ref{difference_rec}), (\ref{eqn349neww}) and (\ref{eqn352new}), it has that:
\begin{equation}
     (\N_{\Omega}^- u,\D_\Omega^-v)_{\Omega_\gamma^+}=\langle\sigma\partial_\mathbf{n} u,v\rangle_{\Gamma_r}+\mathcal{O}(\delta).
\end{equation}
The proof has completed.
\end{proof}

\printbibliography

\end{document}